\documentclass[11pt]{amsart}
\usepackage{geometry}
\usepackage[utf8]{inputenc}
\usepackage{tgtermes}
\usepackage[english]{babel}
\usepackage{amsmath,
	mathtools,
	amssymb,
	graphicx,
	bbm,
	xcolor,
	amsthm,
	mathrsfs,
	imakeidx,
	mathdots}
\usepackage[all]{xy}
\usepackage{hyperref}
\usepackage{tikz-cd}
\usetikzlibrary{decorations.pathmorphing}
\usepackage{dynkin-diagrams}
\usepackage{listings}
\usepackage{bm}
\theoremstyle{plain}

\newtheorem{theorem}{Theorem}[section]
\newtheorem{corollary}[theorem]{Corollary}
\newtheorem{definition}[theorem]{Definition}
\newtheorem{definition-proposition}[theorem]{Definition-Proposition}
\newtheorem{lemma}[theorem]{Lemma}
\newtheorem{proposition}[theorem]{Proposition}

\newtheorem{remark}[theorem]{Remark}
\newtheorem{hypothesis}[theorem]{Hypothesis}
\newtheorem{example}[theorem]{Example}

\newtheorem*{theorem-no-num}{Theorem}
\newtheorem*{proposition-no-num}{Proposition}
\newtheorem*{corollary-no-num}{Corollary}

\makeindex

\begin{document}
	\title{Hida theory for Shimura varieties of Hodge type}
	\author{Xiaoyu ZHANG}
	\email{xiaoyu.zhang@uni-due.de}
	\address{Universit\"{a}t Duisburg-Essen,
		Fakult\"{a}t f\"{u}r Mathematik,
		Mathematikcarr\'{e}e
		Thea-Leymann-Straße 9,
		45127 Essen,
		Germany}

	\begin{abstract}
		In this article, we generalize the work of H.Hida and V.Pilloni
		to construct $p$-adic families of $\mu$-ordinary
		modular forms on Shimura varieties
		of Hodge type
		$Sh(G,X)$
		associated to a Shimura datum
		$(G,X)$ where $G$ is a connected reductive group over
		$\mathbb{Q}$
		and is unramified at $p$,
		such that the adjoint quotient
		$G^\mathrm{ad}$ has no simple factors
		isomorphic to $\mathrm{PGL}_2$.
	\end{abstract}

	\maketitle
	
	\tableofcontents

	\section{Introduction}\label{Introduction}
	
	The theory of $p$-adic families of
	automorphic forms plays an important role
	in recent developments of algebraic number theory.
	The first example of such families was considered by
	Serre in \cite{Serre1972},
	where
	an Eisenstein series
	was interpolated continuously
	in a $p$-adic family by considering
	its $q$-expansion.
	The work of Hida
	(\cite{HidaIwasawaModules,Hida1986GaloisRepresentations})
	exploited the $p$-adic interpolation by a certain family
	(Hida family) of
	cuspidal automorphic eigenforms.
	The corresponding $p$-adic family of Galois representations
	led Mazur to develop the theory of
	Galois deformations
	(\cite{Mazur1989})
	and these works inspired further breakthroughs in
	modularity results
	by Wiles and Taylor-Wiles
	(\cite{Wiles1995,TaylorWiles1995}).
	There are other important applications of the work of
	Hida (Hida theory),
	such as the construction of $p$-adic $L$-functions
	in various settings
	(\cite{EischenHarrisLiSkinner2016,Liu2015})
	just as the initial motivation for
	\cite{Serre1972};
	the proof of certain cases of the Iwasawa Main Conjecture
	by
	Skinner-Urban
	(\cite{SkinnerUrban});
	the proof of the
	Mazur-Tate-Teitelbaum conjecture
	by Greenberg-Stevens
	(\cite{GreenbergStevens1993}),
	etc.
	
	This article deals with the construction of
	Hida theory for Shimura varieties of Hodge type,
	generalizing the works of
	\cite{Hida2002,Pilloni2012},
	where we work over the
	$\mu$-ordinary locus of the Shimura variety instead of
	the ordinary locus
	(which may be empty).
	Let
	$(G,X)$
	be a (mixed) Shimura datum
	where
	$G$ is a connected reductive group over 
	the rationals $\mathbb{Q}$
	and
	$X$ is one $G(\mathbb{R})$-conjugacy class of
	homomorphisms
	from the Deligne torus $\mathbb{S}$
	to
	$G_\mathbb{R}$
	satisfying certain conditions
	(see §\ref{Shimura varieties of Hodge type} for
	details).	
	For a compact open subgroup
	$K\subset G(\mathbb{A}_\mathrm{f})$,
	suppose that the Shimura variety
	$Sh_K(G,X)=G(\mathbb{Q})\backslash
	(X\times G(\mathbb{A}_\mathrm{f})/K)$
	has an integral model
	$Sh$ over
	$\mathcal{O}_\mathfrak{p}
	=
	\mathcal{O}_{E_\mathfrak{p}}$
	where 
	$E$ is the reflex field
	$E=E(G,X)$ of $(G,X)$
	and $\mathfrak{p}$ is a prime of $E$ over
	$p$,
	and assume moreover
	$Sh$
	has a toroidal compactification
	$Sh^\Sigma$
	with respect to a certain cone decomposition
	$\Sigma$.
	The general strategy of constructing $p$-adic families of
	automorphic forms is,
	according to
	\cite{Hida2002}:
	\begin{enumerate}
		\item 
		define an
		open dense ordinary locus
		$Sh^{\Sigma,\mathrm{ord}}$
		of the compactification
		$Sh^\Sigma$
		and construct a certain Igusa tower
		$\mathrm{Ig}_{m,n}$ over
		$Sh^{\Sigma,\mathrm{ord}}$;
		
		\item 
		construct the space
		$\mathbb{V}_{m,n}$
		of rational functions on the Igusa tower
		$\mathrm{Ig}_{m,n}$
		and relate these spaces to the space of
		classical automorphic forms
		$H^0(Sh,\mathcal{V}_\lambda)$
		of a certain automorphic sheaf
		$\mathcal{V}_\lambda$ associated to a weight
		$\lambda\in
		X^{\ast}(T)$,
		via the Hodge-Tate map;
		
		\item 
		construct a certain idempotent operator
		$e_G$
		acting on the spaces
		$\mathbb{V}_{m,n}$ and
		$H^0(Sh,\mathcal{V}_\lambda)$
		such that the subspaces
		$e_G\mathbb{V}_{m,n}$
		of $\mathbb{V}_{m,n}$
		satisfy certain finite dimensional property.
	\end{enumerate}

    In general, for a Shimura datum
    $(G,X)$,
    the ordinary locus
    $Sh^{\mathrm{ord}}$
    of the Shimura variety itself 
	$Sh$
	may be empty
	(in the PEL case, this locus is non-empty if and only if
	$E_\mathfrak{p}=\mathbb{Q}_p$,
	\textit{cf.}
	\cite[Theorem 1.6.3]{Wedhorn1999}
	).
	To carry out the first step in the above strategy,
	one needs a more general notion of ordinariness.
	The candidate we use in this article is
	$\mu$-ordinariness.
	Here $\mu$ refers to a cocharacter
	$\mu
	\colon
	\mathbb{G}_m
	\rightarrow
	G$
	associated to the Shimura datum
	$(G,X)$,
	which is, in vague terms,
	the Frobenius-twisted
	Hodge cocharacter
	induced from the points in $X$
	(see §\ref{mu-ordinary locus} for the precise formulation).
	Now we assume that our Shimura datum is of Hodge type,
	which means that there is an embedding
	$(G,X)
	\rightarrow
	(\mathrm{GSp}(V),S^\pm)$
	into some Siegel Shimura datum.
	For each geometric point
	$P$ of $Sh$,
	one can consider the specialization
	$\mathcal{A}_P$ of the abelian scheme
	$\mathcal{A}$
	(the pull-back of the universal abelian scheme from
	the Siegel Shimura variety associated to
	$(\mathrm{GSp}(V),S^\pm)$
	to
	$Sh$)
	over
	$Sh$ and one has the Frobenius morphism on the
	Dieudonn\'{e} module
	associated to
	$\mathcal{A}_P$.
	This morphism gives rise to a slope filtration on the
	Dieudonn\'{e} module and thus induces a cocharacter
	$\mu_P\in
	X_\ast(T)_\mathbb{Q}$,
	where
	$T$ is a maximal torus of $G$.
	Then the $\mu$-ordinary locus
	$Sh^\mu$
	of $Sh$
	is defined to be the reduced subscheme
	of $Sh$
	consisting of those $P$
	such that $\mu_P$ is equal to the Galois average
	$\overline{\mu}$ of
	$\mu$.
	One can show that this locus $Sh^\mu$ is in fact independent of the choice 
	of the embedding
	$(G,X)
	\rightarrow
	(\mathrm{GSp}(V),S^\pm)$
	and moreover,
	it is open and dense in
	$Sh$
	(rigorously speaking,
	one should work with the base-change
	$Sh_{\overline{\mathbb{F}}_p}$ of $Sh$
	instead of $Sh$ itself.
	We simplify the presentation in the introduction and refer the readers to
	§\ref{mu-ordinary locus and Hasse invariant}
	for more details and precise formulation).
	Using Hasse invariants,
	one can extend the $\mu$-ordinary locus
	from 
	$Sh$ to the compactification
	$Sh^\Sigma$.

	The construction of
	the Igusa tower
	over 
	$Sh^{\Sigma,\mu}$
	is similar to the PEL case:
	we consider the group
	$L_\mu$ of
	isomorphisms of the $p$-divisible group
	$\mathcal{A}_P[p^\infty]$
	associated to $\mathcal{A}_P$,
	but considering only the
	connected component
	(\textit{cf.} §\ref{p-adic modular forms}).
	From this one can define the Igusa tower
	$\mathrm{Ig}_{m,n}$
	with
	$m,n\geq1$.
	Then one has the space
	$\mathbb{V}_{m,n}
	=
	H^0(
	\mathrm{Ig}_{m,n},\mathcal{O}_{\mathrm{Ig}_{m,n}})$.
	We define the Hecke operators
	$\mathbb{T}$
	by first pulling back the Hecke correspondence
	from the Siegel Shimura variety and then
	divide the trace map associated to one projection morphism
	by a certain explicit power of $p$
	(\textit{cf.} §\ref{Hecke operators}).
	From this one can define the idempotent
	$\mu$-ordinary operator
	$e_{\widetilde{P}}$
	for the Levi subgroup
	$\widetilde{L}$ of a parabolic subgroup
	$\widetilde{P}$ of $G$
	with
	$\widetilde{L}
	\subset
	L_\mu$.
	Then one has the space of
	$\mu$-ordinary
	$p$-adic automorphic forms
	$e_{\widetilde{P}}
	\mathbb{V}^{\widetilde{P}_L^\mathrm{der}}_{m,n}$,
	invariant under the action of
	$\widetilde{P}_L^\mathrm{der}$
	(\textit{cf.} §\ref{Hida theory for Hodge Shimura}).
	Here 
	$\widetilde{P}_L^\mathrm{der}$ is
	the derived subgroup of
	$\widetilde{P}_L
	=
	\widetilde{P}\cap L$.

	Now we state the main results of this article.
	First we need some notations.
	We fix a Borel subgroup
	$B$ of $G$ and a maximal torus
	$T$ of $B$.
	For a parabolic subgroup
	$\widetilde{P}$ containing $B$ as above,
	write
	$\widetilde{T}_{\widetilde{P}}
	:=
	\widetilde{P}/\widetilde{P}^\mathrm{der}$
	for the quotient torus of $\widetilde{P}$.
	Using the natural projection
	$T\rightarrow \widetilde{T}_{\widetilde{P}}$,
	one can view the group of characters
	$X^\ast(\widetilde{T}_{\widetilde{P}})$
	as a subgroup of
	$X^\ast(T)$.
	For any character
	$\lambda
	\in
	X^\ast(\widetilde{T}_{\widetilde{P}})$,
	one can construct the space
	$H^0(Sh^{\Sigma,\mu},\mathcal{V}[\lambda])$
	of automorphic forms of weight $\lambda$,
	of level $K$
	(note that $Sh$ depends on the level group
	$K$).
	Let
	$\mathcal{O}_\mathfrak{p}$
	denote the ring of integers of
	$E_\mathfrak{p}$ and
	we write
	$\mathcal{W}_{\widetilde{P}}
	=
	\mathcal{O}_\mathfrak{p}
	[[\widetilde{T}_{\widetilde{P}}(\mathbb{Z}_p)]]$ and
	$\mathcal{W}^1_{\widetilde{P}}
	:=
	\mathcal{O}_\mathfrak{p}
	[[\mathrm{Ker}
	(\widetilde{T}_{\widetilde{P}}(\mathbb{Z}_p)
	\rightarrow
	\widetilde{T}_{\widetilde{P}}(\mathbb{F}_p))]]
	$
	for the weight space of cuspidal
	$\mu$-ordinary $p$-adic automorphic forms
	$\mathbb{M}(\widetilde{P})_\mathrm{cusp}
	=
	\mathrm{Hom}_{\mathcal{O}_\mathfrak{p}}
	(\lim\limits_{\overrightarrow{m}}
	e_{\widetilde{P}}\mathbb{V}_{\mathrm{cusp},m}^{\widetilde{P}_L^\mathrm{der}},
	E_\mathfrak{p}/\mathcal{O}_\mathfrak{p})$
	(\textit{cf.}
	(\ref{p-adic modular forms, big space})).
	
	\begin{theorem}
		Let
		$(G,X)$
		be as above such that
		$G$ is unramified at $p$,
		$G^\mathrm{ad}$
		has no factors isomorphic to
		$\mathrm{PGL}_{2/\mathbb{Q}}$.
		\begin{enumerate}
			\item 
			For any character
			$\lambda
			\in
			X^\ast(\widetilde{T}_{\widetilde{P}})$,
			the subspace
			$e_{\widetilde{P}}
			\mathbb{V}_\infty^{\widetilde{P}_L^\mathrm{der}}[\lambda^{-1}]$
			of
			$e_{\widetilde{P}}
			\mathbb{V}_\infty^{\widetilde{P}_L^\mathrm{der}}$
			on which the torus
			$\widetilde{T}_{\widetilde{P}}$
			acts by the character
			$\lambda$
			is a free $\mathcal{O}_\mathfrak{p}$-module
			of finite rank.
			Moreover,
			this rank depends only on
			the image of $\lambda$
			in the quotient
			$X^\ast(\widetilde{T}_{\widetilde{P}})/
			\mathbb{Z}N_G\lambda_G
			X^\ast(\widetilde{T}_{\widetilde{P}})$.
			Here
			$\lambda_G\in X^\ast(\widetilde{T}_{\widetilde{P}})$
			is a certain character non-trivial associated to some
			Hodge line bundle
			on the Shimura variety
			$Sh$
			(cf.(\ref{character lambda_G}))
			and
			$N_G\ge0$
			is the Hasse number
			(cf.(\ref{Hasse number N_G})).
			
			\item 
			For any $\lambda\in X^\ast(\widetilde{T}_{\widetilde{P}})$
			which is dominant as a character of $T$
			and $\langle\lambda,\alpha\rangle>0$
			for at least one positive coroot
			$\alpha$ of
			$(B,T)$,
			we have isomorphisms
			\[
			e_{\widetilde{P}}
			H^0(Sh^{\Sigma,\mu}_\infty,\mathcal{V}[\lambda^{-1}])
			\simeq
			e_{\widetilde{P}}
			\mathbb{V}_\infty^{\widetilde{P}_L^\mathrm{der}}[\lambda^{-1}].
			\]
			
			\item
			For any
			$\lambda\in
			X^\ast(\widetilde{T}_{\widetilde{P}})$
			and any integer $t\gg0$
			(depending on $\lambda$),
			we have an isomorphism between
			the spaces of cuspidal automorphic forms
			(here $C$ denotes the cusps):
			\[
			e_{\widetilde{P}}
			H^0(Sh_\infty^\mu,
			\mathcal{V}[(\lambda+t\lambda_G)^{-1}](-C))
			\simeq
			e_{\widetilde{P}}
			\mathbb{V}_{\mathrm{cusp},\infty}^{\widetilde{P}_L^\mathrm{der}}
			[(\lambda+t\lambda_G)^{-1}].
			\]
			
			\item
			The
			$\mathcal{W}_{\widetilde{P}}^1$-module
			$\mathbb{M}(\widetilde{P})_\mathrm{cusp}$
			is free of finite rank, and moreover for any
			$\lambda\in
			X^\ast(\widetilde{T}_{\widetilde{P}})$,
			one has the specialization isomorphism
			\[
			\mathbb{M}(\widetilde{P})_\mathrm{cusp}
			\otimes_{\mathcal{W}_{\widetilde{P}},\lambda}
			\mathcal{O}_\mathfrak{p}
			\simeq
			\mathrm{Hom}_{\mathcal{O}_\mathfrak{p}}
			(
			e_{\widetilde{P}}\mathbb{V}_{\mathrm{cusp},\infty}^{\widetilde{P}_L^\mathrm{der}}
			[\lambda^{-1}],
			\mathcal{O}_\mathfrak{p}
			).
			\]
		\end{enumerate}
	\end{theorem}
	The assumption
	that
	$G^\mathrm{ad}$
	has no factors isomorphic to
	$\mathrm{PGL}_{2/\mathbb{Q}}$
	is used when we want to apply
	Koecher principal
	(\textit{cf.}
	\cite[Theorem 5.2.11(5)]{Madapusi2012}).
	This is in fact not a serious restriction since the Hida theory for
	$\mathrm{PGL}_{2/\mathbb{Q}}$
	is already well-known
	(\cite{Hida1986GaloisRepresentations})
	and it is easy to combine our result with
	the case of
	$\mathrm{PGL}_{2/\mathbb{Q}}$
	to obtain Hida theory for any
	$G$
	(whose adjoint quotient may contain
	$\mathrm{PGL}_{2/\mathbb{Q}}$).	
	We do not give a uniform presentation for these two cases
	since the article is already very long and
	the inclusion of
	$\mathrm{PGL}_{2/\mathbb{Q}}$
	makes the presentation less readable 
	even though there are no essential technical difficulties involved.

	Similar constructions in the case of unitary Shimura varieties
	of PEL type can be found in
	\cite{BrascaRosso2017,EischenMantovan2017}.
	In \cite{BrascaRosso2017},
	the authors use a partial Hasse invariant to treat the problem of
	lifting automorphic forms from characteristic $p$ to characteristic $0$
	and they use a different weight space,
	which is more in the spirit of rigid geometry.
	In \cite{EischenMantovan2017},
	more results are obtained,
	including the construction of $p$-adic differential operators.

	Hida theory has important applications in modularity theorems
	of automorphic representations and Galois representations
	(\cite{Pilloni2012b,Wiles1995}).
	Our result for the spin orthogonal groups
	$\mathrm{Spin}_{n,2}$
	of signature $(n,2)$
	should also be of help to
	$p$-adic modularity results for
	abelian varieties of higher dimensions,
	just as the case of elliptic curves and abelian surfaces.
    Thus this article can be seen as the first step towards
    the $p$-adic modularity problems of abelian varieties,
    whose following steps
    we wish to carry out in the near
    future.

	Here is the organization of this article,
	where we also take the chance to indicate some of the obstacles
	that we overcome in generalizing
	\cite{Hida2002,Pilloni2012}.
	In §2 we recall the notion of Shimura varieties of Hodge type and
	fix various notations that will be used through out this article.
	In §3 we introduce the $\mu$-ordinary locus on this Shimura variety
	$Sh$ (of characteristic $p$) and then using
	the Hasse invariant to extend this locus to characteristic $0$.
	The main result is
	Proposition
	\ref{Hasse invariant is reduced},
	which is crucial to studying
	the lifting of automorphic forms from char $p$ to char $0$
	and also the control theorems.
	We reduce the problem to the case
	that $G$
	is simple of adjoint type
	and for each
	such $G$
	we give an explicit description of
	certain element
	inside
	$B(G,\mu)$
	in
	Proposition
	\ref{maximal element in B(G,mu)}.
	In §4 we construct the space of classical automorphic forms on the
	Shimura variety
	$Sh$ and the spaces of ($\mu$-ordinary) $p$-adic automorphic forms
	by introducing the Igusa tower above the
	$\mu$-ordinary locus
	$Sh^{\Sigma,\mu}$.
	The first space is an
	$L(\mathbb{Z}_p)$-torsor over
	the compactified Shimura variety
	$Sh^\Sigma$
	while the second space is
	an
	$L_\mu(\mathbb{Z}_p)$-torsor
	over
	the $\mu$-ordinary locus
	$Sh^{\Sigma,\mu}$.
	It is known that
	$L$ and $L_\mu$
	are inner twists of each other
	(
	this phenomenon seems not appear in preceding works
	like
	\cite{Hida2002,Pilloni2012,BrascaRosso2017,EischenMantovan2017}
	),
	and
	then we can embed the first space into the second
	using a natural specialization map
	(the so-called Hodge-Tate map).
	In §5 we carefully define the Hecke operators at $p$,
	the main result is
	Proposition
	\ref{Integrality of Hecke operator},
	which,
	using the Serre-Tate theory in the
	$\mu$-ordinary case,
	shows that by dividing out some explicit multiple,
	the Hecke operators
	$\mathbb{T}_{\epsilon_i}$
	preserve the integral structure of the space of
	($p$-adic) automorphic forms.
	Then we show that
	the space of
	$\mu$-ordinary automorphic forms on the Igusa tower
	descends to
	the $\mu$-ordinary locus
	$Sh^{\Sigma,\mu}$.
	In §6 we show that one can lift cuspidal automorphic forms
	of characteristic $p$ to
	characteristic $0$
	for sufficiently regular weights
	$\lambda$,
	which is one of the main ingredients in Hida theory.
	We first review the notion of toroidal compactification and
	then show that
	the mod $p$ map
	on the space of
	automorphic forms
	is surjective
	(Proposition \ref{surjectivity of mod p, mu-ordinary}).
	Then we proceed to show the finite-dimensionality of the space of
	$\mu$-ordinary automorphic forms.
	The main ingredient for this is
	Lemma
	\ref{at most one subgroup},
	which shows that
	there is at most one subgroup $H$ of
	$\mathcal{D}_x[p]$
	such that
	its height and degrees satisfy certain conditions.	
	In the last section
	§7 we summarize the results obtained in the previous sections and
	deduce the control theorems on the space of
	$p$-adic automorphic forms and the existence of
	Hida families.
	As the reader can see,
	the presentation and ideas of this article
	is heavily influenced by
	\cite{Pilloni2012}.

	\subsubsection*{Acknowledgement:}
	This work begins while
	the author was a PhD student at
	Universit\'{e} Sorbonne Paris Nord
	and
	he wants to thank
	his advisor
	Jacques Tilouine for
	suggesting this problem and
	for his constant support.
	He would also like to thank
	Andrea Conti,
	Mladen Dimitrov,
	Valentin Hernandez,
	Haruzo Hida,
	Ming-Lun Hsieh,
	Zheng Liu
	and
	Vincent Pilloni
	for useful discussions.

	\subsection*{Notations}
	\begin{enumerate}
		\item 
		We fix an odd prime number $p$.
		We fix field embeddings
		$\overline{\mathbb{Q}}
		\hookrightarrow
		\mathbb{C}$,
		$\overline{\mathbb{Q}}
		\hookrightarrow
		\overline{\mathbb{Q}}_p$
		as well as an isomorphism of fields
		$\mathbb{C}
		\simeq
		\overline{\mathbb{Q}}_p$
		compatible with the previous embeddings.
		We fix an algebraic closure
		$\overline{\mathbb{F}}_p$ of
		$\mathbb{F}_p$.
		For any field $k$
		of characteristic $p$,
		we write
		$W(k)$
		for its ring of Witt vectors
		and
		$W_n(k)$
		for its ring of Witt vectors of length $n\geq0$.
		In particular,
		we write
		$\mathbb{W}=W(\overline{\mathbb{F}}_p)$
		and
		$\mathbb{W}_n
		=W_n(\overline{\mathbb{F}}_p)$.
		Moreover,
		we identify
		the total fraction field
		of
		$\mathbb{W}$
		with
		the completion of the maximal
		unramified extension
		$\widehat{\mathbb{Q}^{\mathrm{ur}}_p}$
		of
		$\mathbb{Q}_p$.
		We write
		$\mathbb{A}=\mathbb{A}_\mathbb{Q}$
		for the ad\`{e}les of $\mathbb{Q}$
		and
		$\mathbb{A}_f$
		the finite ad\`{e}les.
		
		\item 
		Let $G$ be a connected reductive group over
		$\mathbb{Q}$,
		$B$ a Borel subgroup of $G$ defined over $\mathbb{Q}$
		and
		$T$ a maximal torus of $B$ (defined over $\mathbb{Q}$).
		We fix a root datum
		of
		$G$
		\[
		(X^\ast(T),\Phi^\ast(T),X_\ast(T),\Phi_\ast(T)),
		\]
		where
		$X^\ast(T)$
		is the group of characters of $T$,
		Let
		$\Phi^\ast(T)
		\subset X^\ast(T)\backslash0$
		be the subset of roots.
		We write
		$\Delta^\ast(T)$
		for the subset of
		$\Phi^\ast(T)$
		consisting of positive roots
		(with respect to $B$)
		and
		$\widetilde{\Delta}^\ast(T)$
		for the subset of
		$\Delta^\ast(T)$
		consisting of simple roots.
		Similarly we have
		the subset of positive coroots
		$\Delta_\ast(T)$
		and
		$\widetilde{\Delta}_\ast(T)$
		for the subset of
		$\Delta_\ast(T)$
		consisting of simple positive coroots.
		We denote by
		$X^\ast_\mathrm{dm}(T)$
		the subset of dominant characters of
		$X^\ast(T)$,
		$X^\ast_\mathrm{dd}(T)
		\subset
		X^\ast_\mathrm{dm}(T)$
		the subset consisting of
		$\lambda$
		such that
		$\langle\lambda,\mu\rangle>0$
		for at least one
		positive coroot
		$\mu\in
		\Delta_\ast(T)$.
		Similarly,
		$X_{\ast,\mathrm{dm}}(T)$ is 
		the subset of dominant cocharacters of
		$X_\ast(T)$,
		$X_{\ast,\mathrm{dd}}(T)$
		consists of dominant coroots
		$\mu$ such that
		$\langle\lambda,\mu\rangle
		>0$
		for at least one positive root
		$\lambda
		\in
		\Delta^\ast(T)$.	
		When the context is clear,
		we will omit $(T)$
		from these notations.
		Denote
		$\mathcal{W}_T$
		for the Weyl group of
		$(G,T)$.

	\end{enumerate}

	\section{Shimura varieties of Hodge type}
	\label{Shimura varieties of Hodge type}
	We recall the notion of
	Shimura varieties of Hodge type.
	The main reference is
	\cite{Milne1988}.
	
	\begin{definition}\label{Shimura varieties, definition}
		Let
		$(G,X)$
		be a Shimura datum,
		\textit{i.e.},
		$G$ is a connected reductive group over
		$\mathbb{Q}$
		and $X$
		a
		$G(\mathbb{R})$-conjugacy class
		of homomorphisms
		$\mathbb{S}:=
		\mathrm{Res}^\mathbb{C}_\mathbb{R}
		\mathbb{G}_m\to
		G_\mathbb{R}$
		satisfying the following conditions:
		\begin{enumerate}
			\item 
			there exists
			$x\in X$
			such that
			the Hodge structure
			$h_x
			\colon
			\mathbb{S}
			\to
			G_\mathbb{R}
			\to
			\mathrm{GL}(\mathfrak{g})$
			on
			the Lie algebra
			$\mathfrak{g}$ of
			$G$
			is of type
			$\{(-1,1),(0,0),(1,-1)\}$;
			
			\item 
			there exists $x\in X$
			such that
			$\mathrm{ad}h_x(i)$
			is a Cartan involution on
			$G^\mathrm{ad}_\mathbb{R}$
			where
			$G^\mathrm{ad}$
			is the adjoint quotient
			of $G$;
			
			\item 
			$G^\mathrm{ad}$
			has no factors defined
			over $\mathbb{Q}$
			whose real points
			form a compact group;
			
			\item 
			the identity component
			$Z(G)^0$
			of the centre
			$Z(G)$ of $G$
			splits over a CM-field.
			
			We shall make the following additional assumption,
			which is not part of the conditions of a Shimura datum but
			will be used in defining the $\mu$-ordinary locus of
			a toroidal compactification of
			the Shimura variety associated to $(G,X)$
			(cf.
			Definition \ref{various Hasse invariant}):

			\item 
			$G^\mathrm{ad}$
			has no factors isomorphic to
			$\mathrm{PGL}_{2/\mathbb{Q}}$.
		\end{enumerate}

	\end{definition}

	Let $K$
	be a compact open subgroup
	of
	$G(\mathbb{A}_f)$,
	then
	define
	\[
	\mathrm{Sh}_K(G,X)
	=
	G(\mathbb{Q})
	\backslash
	(X\times G(\mathbb{A}_f)/K),
	\]
	which is a finite disjoint union of
	locally symmetric spaces.
	Write
	$\mathcal{C}\subset
	G(\mathbb{A}_f)$
	for a set of representatives
	of
	the quotient
	$G(\mathbb{Q})
	\backslash
	G(\mathbb{A}_f)/K$
	and for each
	$g\in\mathcal{C}$,
	let
	$\Gamma_g'
	=gKg^{-1}\cap G(\mathbb{Q})_+$
	and
	$\Gamma_g$
	its image in
	$G^\mathrm{ad}(\mathbb{R})^+$.
	Let
	$X^+$
	be a connected component of
	$x$,
	then
	$\mathrm{Sh}_K(G,X)
	=\cup_{g\in\mathcal{C}}
	\Gamma_g\backslash X^+$.
	Write the projective limit
	\[
	\mathrm{Sh}(G,X)
	=\lim\limits_{\overleftarrow{K}}
	\mathrm{Sh}_K(G,X),
	\]
	which is a scheme over $\mathbb{C}$.

	\begin{example}
		Consider the following
		example:
		let
		$(V,\psi)$ be a
		non-degenerated
		symplectic
		vector space
		over $\mathbb{Q}$.
		The similitude symplectic
		group
		$\mathrm{GSp}(V,\psi)$
		whose $R$-points are given by
		\[
		\mathrm{GSp}(V,\psi)(R)
		=\{
		g\in\mathrm{GL}(V)(R)|
		\exists\,
		\nu\in
		R^\times,
		\psi(gv,gu)
		=\nu\psi(v,u),
		\,
		\forall\,
		v,u\in V
		\}.
		\]
		Let $S^\pm$
		be the set of
		Hodge structures of type
		$\{(-1,0),(0,-1)\}$
		on $V$
		such that
		$\pm2i\pi\psi$
		is a polarization.
		Then
		$(\mathrm{GSp}(V,\psi),S^\pm)$
		is a Shimura datum,
		called a Siegel Shimura datum.
	\end{example}

	Suppose that there is an embedding
	$\xi\colon
	(G,X)\hookrightarrow
	(\mathrm{GSp}(V,\psi),S^\pm)$
	for some symplectic space
	$(V,\psi)$,
	\textit{i.e.},
	$\mathrm{Sh}(G,X)$
	is a Shimura variety of
	\textbf{Hodge type}.
	Let
	$\mathfrak{t}=(t_\alpha)_{\alpha\in I}$
	be a family of tensors for $V$
	(\textit{i.e.}
	$t_\alpha\in
	V^\otimes
	:=
	\bigoplus_{r,s\in\mathbb{N}}
	V^{\otimes r}\otimes
	\mathrm{Hom}_\mathbb{Q}
	(V,\mathbb{Q})^{\otimes s}$)
	such that
	$G$ is the subgroup of
	$\mathrm{GSp}(V,\psi)\times\mathbb{G}_m$
	fixing all these tensors
	$(t_\alpha)$
	(see
	\cite{Kisin2010}
	for the proof of existence of these tensors).
	Recall the moduli interpretation of
	$\mathrm{Sh}(G,X)$
	(over $\mathbb{C}$):
	consider the triple
	$(A,\mathfrak{s},\eta)$
	consisting of an abelian variety
	$A$ over $\mathbb{C}$,
	a family
	$\mathfrak{s}=(s_\alpha)_{\alpha\in I}$
	of Hodge cycles on
	$A$
	(\textit{cf.}\cite[p.13]{Milne1988}) and an isomorphism
	$\eta
	\colon
	V(\mathbb{A}_f)
	\to
	H_f(A):=
	\otimes_\ell'
	H_\ell(A)$
	with
	$H_\ell(A)
	:=T_\ell(A)\otimes_\mathbb{Z}\mathbb{Q}$
	which takes
	$t_\alpha$
	to 
	$s_\alpha$
	for all $\alpha\in I$
	and
	there is an isomorphism
	$i\colon
	V
	\to
	H_B(A)=H_1(A(\mathbb{C}),\mathbb{Q})$
	sending
	$t_\alpha$
	to
	$s_\alpha$
	and
	the morphism
	$\mathbb{S}
	\to
	G_\mathbb{R},
	\,
	z\mapsto
	i^{-1}\circ
	h_M(z)\circ
	i$
	is an element in $X$.
	We say that two triples
	$(A,\mathfrak{s},\eta)$
	and
	$(A',\mathfrak{s}',\eta')$
	are equivalent
	if there is an isomorphism
	$\gamma
	\colon
	A\to
	A'$
	sending
	$s_\alpha$
	to
	$s_\alpha'$
	and
	$\gamma\circ\eta=\eta'$.
	Write
	$\mathcal{M}(G,X,\xi)$
	for the set of equivalence classes of
	these triples
	$(A,\mathfrak{s},\eta)$.
	Then we have a canonical bijection
	(\cite[Chapter II, Proposition 3.6]{Milne1988}):
	\[
	\Phi_\xi
	\colon
	\mathcal{M}(G,X,\xi)
	\to
	\mathrm{Sh}(G,X).
	\]

	Write
	$E=E(G,X)$
	for the reflex field of
	the Shimura datum
	$(G,X)$,
	then for any
	$K$ as above,
	there exists a unique
	canonical model
	$Sh_K(G,X)$
	over the reflex field $E$
	of
	$\mathrm{Sh}_K(G,X)$
	(\textit{cf.}\cite[§II.5]{Milne1988}).
	Again we write
	$Sh(G,X)$
	for the projective limit of
	$Sh_K(G,X)$
	over all possible $K$ as above,
	which is then the canonical model over
	$E$ of
	$\mathrm{Sh}(G,X)$.
	From now on,
	we consider some special
	compact open subgroups
	$K$,
	which can be written as
	$K=K^pK_p$
	with
	$K^p\subset
	G(\mathbb{A}_f^p)$
	and
	$K_p
	\subset
	G(\mathbb{Q}_p)$.
	Write
	$\mathrm{Ram}(G)$
	for the finite set of places
	of $\mathbb{Q}$
	where
	$G$ is ramified.
	We assume that
	\begin{hypothesis}\label{G is unramified at p}
		$G$ is unramified at $p$,
		\textit{i.e.}
		$p\notin\mathrm{Ram}(G)$.
	\end{hypothesis}
    With this hypothesis,
    we can and will fix a
    $\mathbb{Z}_p$-model
    for
    $G_{\mathbb{Q}_p}$.
    We denote this model again by
    $G$.
	We say that
	\begin{definition}
		$K=K^pK_p$ is
		$p$-hyperspecial
		if
		$K_p$ is hyperspecial in
		$G(\mathbb{Q}_p)$.
	\end{definition}
    The existence of
    a $p$-hyperspecial $K$ implies that
    the reflex field $E$ of the
    Shimura datum $(G,X)$
    is unramified at $p$
    (\cite[Corollary 4.7]{Milne1994}).
    In the following,
    for a place
    $\mathfrak{p}$ of $E$ over $p$,
    we will write
    \[
    \mathcal{O}_E,
    \text{ resp., }
    E_\mathfrak{p},
    \mathcal{O}_\mathfrak{p}
    =
    (\mathcal{O}_E)_\mathfrak{p}
    =
    \mathcal{O}_{E_\mathfrak{p}}
    \]
    for
    the ring of integers of $E$,
    resp.,
    the completion of $E$ at the place
    corresponding to the prime
    $\mathfrak{p}$,
    the ring of integers of $E_\mathfrak{p}$.
    Furthermore,
    we write
    $
    \kappa_\mathfrak{p}
    =
    \mathcal{O}_\mathfrak{p}/\mathfrak{p}
    $
    for the residual finite field
    of
    $\mathcal{O}_\mathfrak{p}$.

    Write
    $Sh_{K_p}(G,X)$
    for the projective limit of
    those
    $Sh_{K^pK_p}(G,X)$
    with
    $K_p$ fixed while
    $K^p$ varying.
    Let
    $\mathfrak{p}$
    be a prime of $E(G,X)$
    over $p$,
    then by
    \cite[Theorem 0]{Vasiu1999}
    and
    \cite[Theorem 1]{Kisin2010},
    $Sh_{K_p}(G,X)$
    has an integral model
    \begin{equation}\label{Sh, integral model for all level}
    Sh(G,X,K_p)
    \end{equation}
    over
    $\mathcal{O}_\mathfrak{p}$.
    This model inherits an action of
    $G(\mathbb{A}_f^p)$ and
    
    \begin{definition}\label{Sh, integral model for a fixed level}
    	We write
    	\[
    	Sh:=Sh(G,X,K)
    	\]
    	for the fixed points of
    	$Sh(G,X,K_p)$
    	by
    	$K^p$.
    	In the following,
    	we will fix one such
    	$K^p$ which is sufficiently small
    	(for example neat).
    \end{definition}

    Next we consider the toroidal
    compactifications of
    $Sh$.
    By
    \cite{Madapusi2012},
    for a certain cone decomposition
    $\Sigma$,
    one can construct
    the toroidal compactification
    $
    Sh^\Sigma
    $
    of
    $Sh$
    over $\mathcal{O}_\mathfrak{p}$.
    Moreover,
    the complement
    $
    C^\Sigma=
    Sh^\Sigma\backslash
    Sh
    $
    is a relative Cartier divisor
    over
    $\mathcal{O}_\mathfrak{p}$.
    Up to refining
    $\Sigma$,
    every complete local ring
    of
    $Sh^\Sigma$
    at a geometric point is isomorphic to
    a complete local ring
    of $Sh$.
    We can apply the Proj
    construction to a certain graded ring of
    automorphic forms on
    $Sh^\Sigma$
    and we get the minimal compactification
    of $Sh$:
    \[
    Sh^\mathrm{min}
    \]
    (independent of $\Sigma$)
    with a unique morphism
    \[
    \pi^\Sigma
    \colon
    Sh^\Sigma
    \to
    Sh^\mathrm{min}
    \]
    extending the identity morphism
    on
    $Sh$ and
    compatible with the stratifications on
    $Sh^\Sigma$
    and
    $Sh^\mathrm{min}$.
    With these integral models,
    we consider their reductions:
    for each integer $n$,
    we write
    \begin{align*}
    Sh^?_n
    &
    =Sh^?
    \otimes_{\mathcal{O}_\mathfrak{p}}
    \mathcal{O}_\mathfrak{p}/\mathfrak{p}^n,
    \quad
    ?=\emptyset,\Sigma,\mathrm{min},
    \\
    C^\Sigma_n
    &
    =
    Sh^\Sigma_n\backslash
    Sh_n
    =
    C^\Sigma
    \otimes_{\mathcal{O}_\mathfrak{p}}
    \mathcal{O}_\mathfrak{p}/\mathfrak{p}^n,
    \\
    \pi^\Sigma_n
    &
    =
    \pi^\Sigma
    \otimes_{\mathcal{O}_\mathfrak{p}}
    \mathcal{O}_\mathfrak{p}/\mathfrak{p}^n
    \colon
    Sh^\Sigma_n
    \to
    Sh^\mathrm{min}_n.
    \end{align*}
    By construction,
    we have
    $(\pi^\Sigma_n)_\ast
    (\mathcal{O}_{Sh^\Sigma_n})
    =
    \mathcal{O}_{Sh^\mathrm{min}_n}$.
    \begin{remark}
    	In the following,
    	we fix an embedding
    	$E\hookrightarrow
    	\overline{\mathbb{Q}}$
    	and thus we have a corresponding embedding
    	$\mathcal{O}_E
    	\hookrightarrow
    	\mathbb{W}$
    	since $E$ is unramified at $p$.
    	Moreover,
    	we use the same notations
    	$Sh^?,C^\Sigma,\pi^\Sigma$
    	to denote their base changes from
    	$\mathcal{O}_E$
    	to
    	$\mathbb{W}$
    	and similarly for the notations with subscript
    	$n$
    	(base changes from
    	$\mathcal{O}_E/\mathfrak{p}^n$
    	to
    	$\mathbb{W}_n$).
    \end{remark}

    For
    $(G,X)=(\mathrm{GSp}(V,\psi),S^\pm)$,
    we write
    $K_V=(K_V)_p(K_V)^p$
    for a compact open subgroup
    of
    $\mathrm{GSp}(V,\psi)(\mathbb{A}_f)$.
    We write
    $\mathcal{A}(V,\psi)$
    for the universal abelian
    scheme over
    $Sh(\mathrm{GSp}(V,\psi),S^\pm,K_V)$.
    Then write
    \[
    \pi
    \colon
    \mathcal{A}
    \to
    Sh
    \]
    for the universal abelian scheme
    over
    $Sh$
    via the pullback of
    $Sh\to
    Sh(\mathrm{GSp}(V,\psi),S^\pm,K_V)$
    (for
    $K=K_V\cap
    G(\mathbb{A}_f)$).
    Write
    $e
    \colon
    Sh
    \to
    \mathcal{A}$
    for the unit section and
    \[
    \omega_{Sh}
    :=
    \mathrm{det}(e^\ast\Omega_{\mathcal{A}/Sh}^1)
    \]
    for the Hodge line bundle on
    $Sh$.
    By
    \cite[Proposition 4.1]{KoskivirtaWedhorn2015},
    we know that
    $\omega_{Sh}$
    is ample over
    $Sh$.

    \section{$\mu$-ordinary locus and Hasse invariant}\label{mu-ordinary locus and Hasse invariant}
    \subsection{$\mu$-ordinary locus}\label{mu-ordinary locus}
    In this section,
    we review the theory of
    $\mu$-ordinary locus for
    the Shimura varieties
    $\mathrm{Sh}(G,X,K)$
    (\cite{Wedhorn1999,Wortmann2013}).
    The integral canonical model
    $Sh$
    in general does not have a moduli interpretation
    in terms of abelian schemes.
    Yet one can still define the Newton
    stratification
    over
    $Sh_1$
    and
    Ekedahl-Oort stratification
    over
    $Sh\otimes_{\mathcal{O}_\mathfrak{p}}
    \overline{\mathbb{F}}_p$.

    We suppose that $G$ is quasi-split over
    $\mathbb{Z}_p$
    and
    split over a finite \'{e}tale extension of
    $\mathbb{Z}_p$.
    Fix then a Borel subgroup
    $B$ of $G$ and a maximal torus
    $T$ of $B$
    (both defined over $\mathbb{Z}_p$).
    We can then identify the quotient
    $\mathcal{W}_T\backslash
    X^\ast(T)$
    with the set of
    conjugacy classes of co-characters
    of $G_\mathbb{C}$.
    Recall that
    for each $x\in X$,
    we have a Hodge decomposition
    $V_\mathbb{C}=V^{(-1,0)}
    \bigoplus
    V^{(0,-1)}$
    by the embedding
    $X\hookrightarrow
    S^\pm$.
    Write
    $\nu_x$
    for the co-character of
    $G_\mathbb{C}$
    such that
    $\nu_x(z)$
    acts on
    $V^{(-1,0)}$
    via multiplication by $z$ and
    on $V^{(0,-1)}$
    trivially.
    Then we write
    $[\nu]$
    for the $G(\mathbb{C})$-conjugacy class    
    of co-characters of $G_\mathbb{C}$
    containing all these 
    $\nu_x$
    and
    similarly
    $[\nu^{-1}]$
    containing all these
    $\nu_x^{-1}$.
    Then write
    \[
    \mu\in
    X^\ast(T)
    \]
    for the element such that
    $\mathrm{Fr}^{-1}(\mu)\in
    [\nu^{-1}]$
    where
    $\mathrm{Fr}$ is the Frobenius morphism.

    We fix a lattice
    $
    \Lambda\subset V
    $
    of
    $V$ and assume that
    $\mathfrak{t}
    \subset
    (\Lambda\otimes_\mathbb{Z}\mathbb{Z}_{(p)})^\otimes
    $
    and
    $\psi(\Lambda\times\Lambda)\subset
    \mathbb{Z}$.
    Write
    $\Lambda^\vee
    \subset
    V^\vee$
    for the dual lattice
    of $\Lambda$ by
    $\psi$.
    Recall the moduli interpretation of
    $Sh(\mathrm{GSp}(V,\psi),S^\pm,K_V)$
    for $(K_V)^p$
    sufficiently small:
    let
    $\mathcal{M}(\mathrm{GSp}(V,\psi),S^\pm,(K_V)^p,\xi)$
    be the moduli
    space over
    $\mathbb{Z}_{(p)}$
    parametrizing
    abelian schemes
    $A$ with a polarization of degree
    $d:=
    [\Lambda^\vee:\Lambda]$
    and a $(K_V)^p$-level structure up to
    isomorphism
    (which is representable by a quasi-projective
    scheme over $\mathbb{Z}_{(p)}$,
    denoted by the same notation).
    Thus we get an embedding
    of
    $\mathbb{Z}_{(p)}$-schemes
    \[
    Sh(\mathrm{GSp}(V,\psi),S^\pm,K_V)
    \hookrightarrow
    \mathcal{M}(\mathrm{GSp}(V,\psi),S^\pm,(K_V)^p,\xi).
    \]
    We make explicit
    this embedding over
    the $\mathbb{C}$-points:
    for any
    \[
    [h,g]\in
    Sh(\mathrm{GSp}(V,\psi),S^\pm,K_V,\xi)
    (\mathbb{C})
    =
    \mathrm{GSp}(\mathbb{Q})
    \backslash
    S^\pm\times\mathrm{GSp}(V,\psi)(\mathbb{A}_f)
    /K_V,
    \]
    we have a Hodge decomposition
    $V_\mathbb{C}
    =V^{(-1,0)}
    \bigoplus
    V^{(0,-1)}$
    given by $h$.
    Moreover, there is a unique
    lattice $\Lambda_g$
    of $V$ such that
    $(\Lambda_g)_{\widehat{\mathbb{Z}}}
    =g(\Lambda_{\widehat{\mathbb{Z}}})$
    with
    $\Lambda_{\widehat{\mathbb{Z}}}
    =\Lambda\otimes_\mathbb{Z}
    \widehat{\mathbb{Z}}$
    and a unique $\mathbb{Q}^\times$-multiple
    $\psi_{h,g}$
    of
    $\psi$
    such that
    $g(\Lambda_{\widehat{\mathbb{Z}}}^\vee)$
    is the dual lattice of
    $g(\Lambda_{\widehat{\mathbb{Z}}})$
    by
    $\psi_{h,g}$
    and such that
    the bilinear form
    $(v,w)
    \mapsto
    \psi_{h,g}(v,h(i)w)$
    is positive definite on
    $V_\mathbb{R}$.
    By Riemann's theorem,
    we can associate to
    $[h,g]$
    a triple
    $(A,\lambda,\eta)$
    with
    $A:=
    V^{(-1,0)}/\Lambda$
    a complex abelian variety
    with polarization $\lambda$
    induced by $\psi_{h,g}$
    and
    $\eta$
    the right
    $(K_V)^p$-coset of
    $\Lambda_{\widehat{\mathbb{Z}}^p}
    \xrightarrow{g^p}
    g^p(\Lambda_{\widehat{\mathbb{Z}}^p})
    =
    (\Lambda_g)_{\widehat{\mathbb{Z}}^p}
    \simeq
    H_1(A,\mathbb{Z})_{\widehat{\mathbb{Z}}^p}
    \simeq
    \prod_{\ell\neq p}
    T_\ell(A)$.

    Recall that we have a universal
    abelian scheme
    $\pi
    \colon
    \mathcal{A}
    \to
    Sh$.
    We then write
    \[
    \mathcal{V}^\circ
    :=
    H^1_\mathrm{dR}
    (\mathcal{A}/Sh),
    \quad
    \mathcal{V}
    :=
    H^1_\mathrm{dR}
    (\mathcal{A}\otimes E/\mathrm{Sh}(G,X,K)).
    \]
    Moreover, we have the Hodge filtration
    on $\mathcal{V}^\circ$:
    \[
    0=
    \mathrm{Fil}^{-1}
    \mathcal{V}^\circ
    \subset
    \mathrm{Fil}^0\mathcal{V}^\circ
    :=
    \pi_\ast
    \Omega_{\mathcal{A}/Sh}^1
    \subset
    \mathrm{Fil}^1\mathcal{V}^\circ
    =
    \mathcal{V}^\circ.
    \]
    For any
    field extension
    $E'/E(G,X)$
    embedded in $\mathbb{C}$
    and a fixed point
    $\zeta\in
    Sh(E')$,
    consider the algebraic closure
    $\overline{E'}
    \subset
    \mathbb{C}$
    and denote by
    $\overline{\zeta}$,
    resp.,
    $\zeta_\mathbb{C}$
    the $\overline{E'}$-point,
    resp.,
    $\mathbb{C}$-point corresponding to $\zeta$.
    From
    $Sh
    \hookrightarrow
    \mathcal{M}(\mathrm{GSp}(V,\psi),S^\pm,K_V,\xi)$,
    we have an isomorphism
    \[
    V\simeq
    H_1(\mathcal{A}_{\zeta_\mathbb{C}},\mathbb{Q}).
    \]
    Moreover we have the comparison isomorphisms
    \[
    H^1
    (\mathcal{A}_{\zeta_\mathbb{C}},\mathbb{Q})_\mathbb{C}
    \simeq
    H^1_\mathrm{dR}
    (\mathcal{A}_{\zeta_\mathbb{C}}/\mathbb{C}),
    \quad
    H^1
    (\mathcal{A}_{\zeta_\mathbb{C}},\mathbb{Q})_{\mathbb{Q}_\ell}
    \simeq
    H^1_\text{\'{e}t}
    (\mathcal{A}_{\zeta_\mathbb{C}},\mathbb{Q}_\ell)
    \simeq
    H^1_\text{\'{e}t}
    (\mathcal{A}_{\overline{\zeta}},\mathbb{Q}_\ell),
    \]
    The tensors
    $\mathfrak{t}
    \subset
    (V^\vee)^\otimes
    \simeq
    H^1(\mathcal{A}_{\zeta_\mathbb{C}},\mathbb{Q})^\otimes$
    correspond to
    $\mathfrak{t}_{\text{dR},\zeta}$,
    $\mathfrak{t}_{\text{\'{e}t},\ell,\zeta}$
    which form the Hodge tensors.
    We can show that there exist sections
    \[
    \mathfrak{t}_\text{dR}
    \subset
    \mathcal{V}^\otimes
    \]
    defined over $E(G,X)$
    horizontal with respect to the
    Gauss-Manin connection
    $\nabla$
    and for each
    $\zeta\in
    Sh(E')$,
    the pullback of
    $\mathfrak{t}_\text{dR}$
    to
    $\zeta$
    is
    $\mathfrak{t}_{\text{dR},\zeta}
    \subset
    H^1_\text{dR}(\mathcal{A}_{\zeta}/E')^\otimes$.
    One can then extend
    $\mathfrak{t}_\text{dR}$
    to integral tensors
    $\mathfrak{t}_\text{dR}
    ^\circ
    \subset
    (\mathcal{V}^\circ)^\otimes$.
    For any perfect field $k$
    of finite transcendental degree over
    $\mathbb{F}_p$
    and
    $L(k)=\mathrm{Frac}(W(k))$,
    consider the induced points
    $x\in
    Sh(k)$
    and
    $\widetilde{x}\in
    Sh(L(k))$.
    Write
    $\mathbb{D}_x$
    for the contravariant Dieudonn\'{e} module
    of the $p$-divisible group of
    $\mathcal{A}_x$
    (equipped with a $\mathrm{Frob}$-linear map
    $F$ and $\mathrm{Frob}^{-1}$-linear map
    $V$ such that $FV=p=VF$).
    Moreover, we have isomorphisms
    \[
    H^1_\text{dR}
    (\mathcal{A}_{\widetilde{x}}/W(k))
    \simeq
    H^1_\text{cris}
    (\mathcal{A}_x/W(k))
    \simeq
    \mathbb{D}_x,
    \]
    \[
    \Lambda_{\mathbb{Z}_p}
    \simeq
    H_1(\mathcal{A}_{\zeta_\mathbb{C}},\mathbb{Z})_{\mathbb{Z}_p}
    \simeq
    T_p(\mathcal{A}_{\zeta_\mathbb{C}})
    \simeq
    T_p(\mathcal{A}_{\overline{\zeta}}).
    \]
    Dualizing the last isomorphism gives
    $\Lambda_{\mathbb{Z}_p}^\vee
    \simeq
    T_p(\mathcal{A}_{\overline{\zeta}})^\vee(-1)
    \simeq
    H^1_{\text{\'{e}t}}
    (\mathcal{A}_{\overline{\zeta}},\mathbb{Z}_p)$,
    which sends
    $\mathfrak{t}_{\mathbb{Z}_p}$
    to
    $
    \mathfrak{t}_{\text{\'{e}t},\zeta}^\circ
    \subset
    H^1_{\text{\'{e}t}}
    (\mathcal{A}_{\overline{\zeta}},\mathbb{Z}_p)
    $
    which are invariant under
    $\mathrm{Gal}(\overline{L(k)}/L(k))$
    and
    whose base change to
    $H^1_{\text{\'{e}t}}
    (\mathcal{A}_{\overline{\zeta}},\mathbb{Q}_p)$
    are exactly
    $\mathfrak{t}_{\text{\'{e}t},p,\zeta}$.

    From the $p$-adic comparison theorem
    \[
    H^1_{\text{\'{e}t}}
    (\mathcal{A}_{\overline{\zeta}},\mathbb{Z}_p)
    \otimes_{\mathbb{Z}_p}
    B_\text{cris}
    \simeq
    H^1_\text{cris}
    (\mathcal{A}_x/W(k))
    \otimes_{W(k)}
    B_\text{cris}
    \simeq
    \mathbb{D}_x\otimes_{W(k)}
    B_\text{cris},
    \]
    one can get the images
    $
    \mathfrak{t}_{\text{dR},\widetilde{x}}^\circ
    \subset
    (\mathcal{V}_{\widetilde{x}}^\circ)^\otimes
    $
    of
    these
    $\mathfrak{t}_{\text{\'{e}t},\zeta}^\circ$.
    Moreover, we have an isomorphism
    \[
    f
    \colon
    (\Lambda_{W(k)}^\vee,\mathfrak{t}_{W(k)})
    \simeq
    (\mathcal{V}_{\widetilde{x}}^\circ,
    \mathfrak{t}_{\text{dR},\widetilde{x}}^\circ)
    \]
    and
    there is a cocharacter $\lambda$
    of $G(W(k))$
    such that
    the filtration
    $\Lambda_{W(k)}^\vee
    \supset
    f^{-1}
    (\mathrm{Fil}^1\mathcal{V}_{\overline{x}}^\circ)$
    is induced by
    $(\cdot)^\vee\circ\lambda$.
    One can also show that
    the images
    $
    \mathfrak{t}_{\text{cris},x}
    \subset
    (\mathbb{D}_x)^\otimes
    $
    of
    $\mathfrak{t}_{\text{dR},\widetilde{x}}$
    are independent of the choice of
    $\widetilde{x}$ over $x$
    and we also have an isomorphism
    \[
    \beta
    \colon
    (\Lambda_{W(k)}^\vee,\mathfrak{t}_{W(k)})
    \simeq
    (\mathbb{D}_x,\mathfrak{t}_{\text{cris},x}).
    \]
    
    To each such $\beta$,
    one can associate an element
    $g_\beta\in
    G(L(k))$
    as follows:
    $g_\beta$
    is the unique element such that the following
    diagram commutes:
    \begin{equation}\label{g_beta}
    \begin{tikzcd}
    (\Lambda_{W(k)}^\vee)^{(\sigma)}
    \arrow[d,"\beta^{(\sigma)}"]
    \arrow[r,"\sim"]
    &
    \Lambda_{W(k)}^\vee
    \arrow[r,"(g_\beta^{-1})^\mathrm{t}"]
    &
    \Lambda_{W(k)}^\vee
    \arrow[d,"\beta"]
    \\
    \mathbb{D}_x^{(\sigma)}
    \arrow[rr,"F\circ(1\otimes\sigma)^{-1}"]
    &
    &
    \mathbb{D}_x
    \end{tikzcd}
    \end{equation}
    Here the transpose inverse
    $(g_\beta^{-1})^\mathrm{t}$
    is considered in the embedding
    $g_\beta\in
    G(L(k))
    \subset
    \mathrm{GSp}(\Lambda(L(k)))$.

    We summarize the above discussion in the following lemma
    \begin{lemma}
    	For any
    	$x\in
    	Sh(k)$,
    	let
    	$\beta
    	\colon
    	(\Lambda_{W(k)}^\vee,
    	\mathfrak{t}_{W(k)})
    	\simeq
    	(\mathbb{D}_x,\mathfrak{t}_{\mathrm{cris},x})$
    	be an isomorphism.
    	Then any (other) isomorphism
    	$\beta'
    	\colon
    	(\Lambda_{W(k)}^\vee,
    	\mathfrak{t}_{W(k)})
    	\simeq
    	(\mathbb{D}_x,\mathfrak{t}_{\mathrm{cris},x})$
    	is of the form
    	$\beta'=\beta\circ h^\vee$
    	for some unique
    	$h\in
    	G(W(k))$.
    	Moreover, one has
    	\[
    	g_{\beta'}
    	=
    	h^{-1}g_\beta\sigma(h)
    	\in
    	G(W(k))\mu(p)G(W(k)).
    	\]
    \end{lemma}

    Next we define Newton stratifications on
    $Sh
    \otimes
    \kappa_\mathfrak{p}$.
    Now suppose that
    $k$ is algebraically closed.
    For two elements
    $g,g'\in G(L(k))$,
    we say $g$ is $\sigma$-conjugate to $g'$
    if $g=h^{-1}g'\sigma(h)$
    for some $h\in G(L(k))$.
    For any $g\in G(L(k))$,
    we write
    $[g]$ 
    for the $\sigma$-conjugacy class of $g$
    and
    $B(G)$
    for the set of
    all such $\sigma$-conjugacy classes
    (this set is independent of the choice of $k$).
    We have the following
    Newton map and Kottwitz map of $G$
    (\cite{Kottwitz1985,RapoportRichartz1996}):
    \[
    \nu_G
    \colon
    B(G)
    \to
    (\mathcal{W}_T
    \backslash X^\ast(T)_\mathbb{Q})^{\langle\sigma\rangle},
    \quad
    \kappa_G
    \colon
    B(G)
    \to
    \pi_1(G)_{\langle\sigma\rangle}.
    \]
    Here
    $\mathcal{W}_T$ is the Weyl group of the pair
    $(T,B)$ and
    $X^\ast(T)_\mathbb{Q}:=
    X^\ast(T)\otimes_\mathbb{Z}\mathbb{Q}$,
    the superscript
    $\langle\sigma\rangle$
    means the $\sigma$-invariant and
    the subscript $\langle\sigma\rangle$
    means the $\sigma$-coinvariant.
    For later use,
    we recall some details of the construction of
    $\nu_G$:
    for any $b\in G(L(k))$,
    there is a unique element
    $\nu_b
    \in
    X^\ast(T)_\mathbb{Q}$
    such that there exist an integer $s>0$,
    an element $c\in G(L(k))$
    and a uniformizer
    $\omega$ of
    $\mathbb{Q}_p$
    (recall that $G$ is defined over $\mathbb{Q}$),
    such that
    (among others)
    $s\nu_b
    \in
    X^\ast(T)$
    and
    $cb\sigma(b)\cdots\sigma^s(b)
    \sigma^s(c)^{-1}=c(s\nu_b)(\omega)c^{-1}$
    (\cite[§4]{Kottwitz1985}).
    Then we set
    \[
    \nu_G(b):=\nu_b.
    \]
    This is the slope homomorphism
    (cocharacter) associated to $b$.
    Moreover,
    the set
    $(\mathcal{W}_T\backslash
    X^\ast(T)_\mathbb{Q})^{\langle\sigma\rangle}$
    has a partial order
    $\preccurlyeq$
    which is a generalization of the
    notion of
    ``lying over''
    for Newton polygons
    (\cite[§2]{RapoportRichartz1996}).
    More precisely,
    recall that
    $(X^\ast(T),\Phi^\ast,X_\ast(T),\Phi_\ast)$
    is the root datum associated to $(G,T)$.
    We fix also a basis $\Delta$
    of this root datum with respect to
    $(B,T)$
    and denote by $\Delta^\vee$
    the set of coroots corresponding to $\Delta$.
    We then write
    \begin{align*}
    \overline{C}
    &
    :=
    \{
    x\in X^\ast(T)_\mathbb{R}|\langle x,\alpha\rangle\geq0,
    \forall \alpha\in\Delta
    \},
    \\
    C^\vee
    &
    :=
    \{
    x\in X^\ast(T)_\mathbb{R}|
    x=\sum_{\alpha^\vee\in\Delta^\vee}
    n_{\alpha^\vee}\alpha^\vee,
    0\leq n_{\alpha^\vee}\in\mathbb{R}
    \}
    \end{align*}
    for the Weyl chamber,
    resp.,
    obtuse Weyl chamber of
    this root datum.
    For two elements
    $x,x'\in X^\ast(T)_\mathbb{R}$,
    we write
    $x\preccurlyeq x'$ if
    $x$ lies in the convex hull of the orbit
    $\mathcal{W}_Tx'$.
    One can show that this is equivalent to the fact that for
    any $\mathbb{Q}$-rational representation
    $\rho
    \colon
    G\rightarrow
    \mathrm{GL}(W)$
    and any maximal torus
    $T'$ of $\mathrm{GL}(W)$
    containing $\rho(T)$,
    we have
    $\rho(x)\leq\rho(x')$
    in the usual sense.
    This relation induces a partial order
    $\preccurlyeq$ on
    $(\mathcal{W}_T\backslash X^\ast(T))^{\langle\sigma\rangle}$.
    Using $\nu_G$,
    we get a partial order on
    $B(G)$,
    still denoted by
    $\preccurlyeq$.   
    Recall that we have a unique dominant element
    $\mu\in
    X_\ast(T)$
    such that
    $\sigma^{-1}(\mu)
    \in
    \kappa_G(\nu^{-1})$.
    Then we set
    \[
    \overline{\mu}
    :=
    \frac{1}{r}
    \sum_{i=0}^{r-1}
    \sigma^i(\mu)
    \in
    (X_\ast(T)_\mathbb{Q})^{\langle\sigma\rangle}
    \]
    where
    $r$ is a non-zero integer
    such that $\sigma^r(\mu)=\mu$.
    We denote the image of
    $\overline{\mu}$ in
    $(\mathcal{W}_T\backslash
    X_\ast(T)_\mathbb{Q})^{\langle\sigma\rangle}$
    again by $\overline{\mu}$.
    Then we write
    $\mu^\sharp$
    for the image
    of $\mu$ under the projection
    \[
    X^\ast(T)
    \to
    \pi_1(G)_\mathbb{Q}
    :=
    (X^\ast(T)/
    \langle\alpha^\vee
    |\alpha^\vee\in\Phi^\vee\rangle)_{\langle\sigma\rangle},
    \quad
    \mu
    \mapsto
    \mu^\sharp.
    \]
    
    Then we write
    \[
    B(G,\mu)
    :=
    \{
    b\in
    B(G)|
    \kappa_G(b)=\mu^\sharp,\,
    \nu_G(b)\preceq
    \overline{\mu}
    \}.
    \]
    We know that
    $B(G,\mu)$ is exactly the image of the double coset
    $G(W(k))\mu(p)G(W(k))$
    in $B(G)$
    (\textit{cf.}\cite[Definition 5.6]{Wortmann2013}).

    For any point
    $x\in
    Sh\otimes\kappa_\mathfrak{p}$,
    write
    $k(x)$ for the residue field of $x$
    and $k$ an algebraic closure of $k(x)$
    and
    $\widehat{x}$
    the associated geometric point over $x$.
    To each isomorphism
    $\beta\colon
    (\Lambda_{W(k)}^\vee,
    \mathfrak{t}_{W(k)})
    \simeq
    (\mathbb{D}_{\widehat{x}},
    \mathfrak{t}_{\mathrm{cris},\widehat{x}})$,
    one associates an element
    $g_\beta
    \in
    G(L(k))$.
    One can show that
    $[g_\beta]$
    is independent of the choices of $\beta$,
    $k$,
    and lies in
    $B(G,\mu)$.
    Thus we get a well-defined map
    \begin{equation}\label{NT map}
    \mathcal{NT}
    \colon
    Sh\otimes\kappa_\mathfrak{p}
    \to
    B(G,\mu),
    \quad
    x\mapsto
    [g_\beta].
    \end{equation}

    \begin{definition}
    	For any element $b\in B(G,\mu)$,
    	we write
    	\[
    	\mathcal{N}^b
    	:=
    	\mathcal{NT}^{-1}(b)
    	\subset
    	Sh\otimes\kappa_\mathfrak{p}.
    	\]
    	It is called the Newton stratum of $b$.
    	The \textbf{$\mu$-ordinary locus} in
    	$Sh\otimes\kappa_\mathfrak{p}$
    	is the stratum
    	$\mathcal{N}^{b_\mathrm{max}}$
    	with
    	$b_\mathrm{max}$
    	the maximal element in $B(G,\mu)$
    	(which can be shown to exist).
    	We write
    	\[
    	Sh^\mu_1
    	:=
    	\mathcal{N}^{b_\mathrm{max}}.
    	\]
    \end{definition}

    \subsection{Hasse invariant}\label{Hasse invariant}
    In this subsection we recall the construction of
    the Hasse invariant which cuts out the $\mu$-ordinary
    locus defined above.
    Moreover,
    we shall use Hasse invariant to extend
    the $\mu$-ordinary locus from
    $Sh$ to a toroidal compactification
    $Sh^\Sigma$.

    Recall we have a cocharacter
    $\mu
    \colon
    \mathbb{G}_m
    \rightarrow
    G$
    which is defined over
    $\kappa_\mathfrak{p}$.
    We write
    $P_\pm:=
    P_\pm(\mu)$
    to be the  pair of opposite parabolic subgroups of
    $G_{\mathbb{F}_p}$
    defined by $\mu$,
    with common Levi factors
    $L$,
    the centralizer of $\mu$.
    Then
    $(G,P_+,(P_-)^{(\sigma)},\sigma\colon
    G\rightarrow
    G^{(\sigma)})$
    is an algebraic zip datum
    (\cite[§2.3]{KoskivirtaWedhorn2015}).
    For any $x\in P_+$,
    write
    $\overline{x}$
    for its image in the Levi quotient
    $P_+\rightarrow
    P_+/U_{P_+}$ and similarly
    for $y\in (P_-)^{(\sigma)}$,
    we write
    $\overline{y}$
    for its image in the Levi quotient.
    We then set the zip group to be
    $\mathbb{E}:=
    \{
    (x,y)\in P_+\times(P_-)^{(\sigma)}|
    \sigma(\overline{x})=\overline{y}
    \}$
    which acts on $G$ by restricting the action of
    $P_+\times(P_-)^{(\sigma)}$
    to $\mathbb{E}$.
    The quotient stack $[\mathbb{E}\backslash G]$
    is the stack of $G$-zips.
    More concretely,
    one has
    $\mathbb{E}=
    \{
    (ux',v\sigma(x'))|
    u\in U_{P_+},
    v\in U_{(P_-)^{(\sigma)}},
    x'\in P_+/U_{P_+}
    \}$.
    Now we write
    $S:=
    \{
    (x,y)\in E|
    x=y
    \}
    \subset
    \mathbb{E}$,
    which is the scheme-theoretic stabilizer of $1\in G$
    under the action of $E$.
    In general $S$ may not be smooth and we write
    $S_\mathrm{red}$
    to be the underlying reduced group scheme associated to $S$.
    Then $S_\mathrm{red}$ is a finite constant group scheme over
    $\overline{\mathbb{F}}_p$.
    The finite group
    $S_\mathrm{red}(\overline{\mathbb{F}}_p)$
    can be identified with another group:
    write
    $P_\pm^\circ
    =\cap_{n\geq0}
    (P_\pm)^{(\sigma^n)}$,
    which are now opposite parabolic subgroups of
    $G$ defined over $\mathbb{F}_p$
    with common Levi subgroup
    $L_{P_+^\circ}
    =
    L_{P_-^\circ}$.
    Then we have
    (\cite[Lemma 2.14]{KoskivirtaWedhorn2015})
    \[
    S_\mathrm{red}(\overline{\mathbb{F}}_p)
    =
    L_{P_\pm^\circ}(\mathbb{F}_p).
    \]
    Consider the character group
    $X^\ast(S_\mathrm{red})$,
    and we write
    $N_G$
    to be the exponent of
    the finite abelian group
    $X_\ast(S_\mathrm{red})$,
    \begin{equation}\label{Hasse number N_G}
    N_G
    =
    \min\{0<n\in\mathbb{N}|nX_\ast(S_\mathrm{red})=0\},
    \end{equation}
    called the Hasse number of
    $(G,X)$
    (\cite[Definition 4.11]{KoskivirtaWedhorn2015}).
    Moreover,
    it is easy to see that
    the parabolic subgroup
    $P_{\overline{\mu}}$
    (defined over $\mathbb{Z}_p$)
    stabilizing the filtration induced by
    $\overline{\mu}$
    (which is $\sigma$-invariant)
    gives
    $P_{\overline{\mu}}
    \otimes_{\mathbb{Z}_p}\mathbb{F}_p
    =
    P_+^\circ$.
    We thus deduce
    \begin{lemma}
    	The exponent of the
    	$\mathbb{F}_p$-points of the quotient torus
    	$T_{P_{\overline{\mu}}}
    	=
    	P_{\overline{\mu}}/
    	P_{\overline{\mu}}^\mathrm{der}$
    	is equal to
    	$N_G$.
    \end{lemma}
    On the other hand,
    suppose that the splitting field of
    $T_{P_{\overline{\mu}}}\times_{\mathbb{Z}_p}
    \mathbb{F}_p$
    is $\mathbb{F}_{p^w}$,
    then
    \begin{corollary}
    	We have the following identity:
    	\[
    	N_G
    	=p^w-1.
    	\]
    \end{corollary}

    By \cite[4.7]{KoskivirtaWedhorn2015},
    we know that
    $Sh^\mu_1$
    is open and dense in
    $Sh_1$.   
    Moreover, one can show that
    these $\mathcal{N}^b$
    are indeed locally closed
    (thus form a strata).
    Moreover,
    two points
    $x_1,x_2\in
    Sh_1$
    lie in the same Newton stratum if and only if
    we have an isomorphism
    $(\mathbb{D}_{\widehat{x_1}},
    \mathfrak{t}_{\mathrm{cris},\widehat{x_1}})
    \simeq
    (\mathbb{D}_{\widehat{x_2}},
    \mathfrak{t}_{\mathrm{cris},\widehat{x_2}})$.
    By
    \cite[Definition 4.11]{KoskivirtaWedhorn2015},
    we know that for the integer
    $N:=
    N_G$
    (the Hasse number of $(G,X)$)
    such that
    there is a section
    \[
    H
    \in
    H^0(Sh_1,\omega_{Sh}^{\otimes N})
    \]
    whose non-vanishing locus is exactly
    $Sh^\mu_1$
    (\cite[Theorem 4.12]{KoskivirtaWedhorn2015}).
    Moreover,
    by the Koecher principal
    (or apply \cite[Theorem 5.2.11(5)]{Madapusi2012}),
    $H$ extends to a unique section
    \[
    H^\Sigma
    \in
    H^0(Sh_1^\Sigma,
    \omega_{Sh^\Sigma}^{\otimes N}),
    \]
    whose non-vanishing locus is denoted by
    $Sh^{\Sigma,\mu}_1$.
    If moreover $G^\mathrm{ad}$
    has no factor isomorphic to
    $\mathrm{PGL}_{2/\mathbb{Q}}$,
    then
    $H$ extends to a unique section
    \[
    H^\mathrm{min}
    \in
    H^0(Sh^\mathrm{min}_1,
    \omega_{Sh^\mathrm{min}}^{\otimes N})
    \]
    whose non-vanishing locus is denoted by
    $Sh^{\mathrm{min},\mu}_1$
    (\cite{KoskivirtaWedhorn2015}).
    If
    $Sh^?_1$
    is projective,
    then
    $Sh^{?,\mu}_1$
    is affine
    ($?=\emptyset,\Sigma,\mathrm{min}$).
    \begin{remark}
    	We know that a certain power
    	$H^r$ of the Hasse invariant
    	lifts from
    	the special fibre
    	$Sh_1$ to
    	$Sh$,
    	denoted again by $H^r$.
    	We will write without further comment
    	the non-vanishing locus of
    	$H^r$ inside
    	$Sh$ as
    	$Sh^\mu$.
    	Similarly for
    	$Sh^{\Sigma,\mu}$
    	and
    	$Sh^{\mathrm{min},\mu}$.
    \end{remark}
    
    \begin{definition}\label{various Hasse invariant}
    	We call
    	$H,H^\Sigma,H^\mathrm{min}$
    	the Hasse invariant
    	of the Shimura variety
    	$Sh,Sh^\Sigma,Sh^\mathrm{min}$.
    \end{definition}    
    \begin{remark}
    	We work with the Hasse invariant
    	\'{a} la
    	\cite{KoskivirtaWedhorn2015}.
    	There is another notion of
    	Hasse invariant
    	studied by
    	Hernandez
    	(\cite{Hernandez2018})
    	and some others.
    	The latter invariant is a purely local notion
    	and in some sense our $H$
    	is a products of local invariants studied in
    	\cite{Hernandez2018}.
    	The existence of
    	$H$
    	is proved in
    	\cite{KoskivirtaWedhorn2015},
    	while the local Hasse invariant
    	in
    	\cite{Hernandez2018}
    	does not always exist
    	(cf.
    	Remarque 9.23 of loc.cit).
    	The relation between these notions of Hasse invariants
    	remains to be further explored.
    \end{remark}

    The rest of this subsection is devoted to the proof
    of the following proposition:
    \begin{proposition}\label{Hasse invariant is reduced}
    	The Cartier divisor
    	$V(H)$
    	associated to
    	the Hasse invariant
    	$H$ is reduced on the non $\mu$-ordinary locus
    	$Sh_1^{n-\mu}
    	:=
    	Sh_1-
    	Sh_1^\mu$.
    \end{proposition}
    Note that
    the non-vanishing locus of
    $H$
    is the $\mu$-ordinary locus
    $Sh_1^\mu$.
    The Cartier divisor
    $V(H)$
    lies only inside
    the complement
    $Sh_1^{n-\mu}$
    of
    $Sh_1^\mu$.
    
    In
    the definition of the Shimura datum
    $(G,X)$,
    we require the cocharacters
    $h_x\in X$
    to be minuscule
    (\cite[§1.2]{Deligne1977}).
    In this case we can describe the set
    $B(G,\mu)$ explicitly which will enable us
    to study the properties of the Hasse invariant.
    Let $T_0\subset T$ be the maximal split over $\mathbb{Q}_p$.
    As before,
    we write
    $(X_\ast(T),\Phi_\ast,X^\ast(T),\Phi^\ast)$
    for the absolute root datum with simple roots
    $\widetilde{\Delta}$,
    simple coroots
    $\widetilde{\Delta}_\ast$
    and
    $(X_\ast(T_0),\Phi_{\ast,0},X^\ast(T_0),\Phi_0^\ast)$
    for the relative root datum with simple roots
    $\widetilde{\Delta}_0$,
    simple coroots
    $\widetilde{\Delta}_{\ast,0}$.
    For a root $\alpha\in\Phi^\ast$,
    we denote by
    $w_\alpha\in
    X^\ast(T)_\mathbb{Q}$
    the fundamental weight associated to $\alpha$.
    Similarly,
    for a coroot
    $\beta^\vee\in\Phi_\ast$,
    we write
    $w_{\beta^\vee}$
    for the fundamental co-weight associated to $\beta^\vee$.
    For $\alpha\in\widetilde{\Delta}_0^\ast$,
    put
    \[
    \widetilde{w}_\alpha
    :=
    \sum_\beta w_\beta\in
    X^\ast(T_0)_\mathbb{Q},
    \]
    where $\beta$ runs through $\Phi^\ast$ such that
    $\beta|_{T_0}=\alpha$
    (which is equal to
    $w_\alpha$,
    the fundamental weight associated to
    $\alpha\in\Phi_0^\ast$).
    Then by
    \cite[Corollary 4.3]{ChenFarguesShen2017},
    \begin{proposition}\label{B(G,mu)}
    	The set
    	$B(G,\mu)$
    	consists
    	of dominant cocharacters
    	$\nu\in
    	X_{\ast,\mathrm{dm}}(T_0)_\mathbb{Q}$
    	such that
    	$0\leq
    	\overline{\mu}-\nu
    	\in
    	\mathbb{Q}_{\geq0}\langle\Delta_{\ast,0}\rangle$
    	and
    	for any
    	$\alpha\in
    	\widetilde{\Delta}_0^\ast$
    	with
    	$\langle\nu,\alpha\rangle\neq0$,
    	one has
    	$\langle
    	\overline{\mu}-\nu,
    	\widetilde{w}_\alpha
    	\rangle
    	\in
    	\mathbb{N}$.
    \end{proposition}

    Note that we have a decomposition
    \[
    X_\ast(T_0)_\mathbb{Q}
    =
    (\Phi_{\ast,0})_\mathbb{Q}
    \bigoplus
    (\Phi_0^\ast)_\mathbb{Q}^\perp.
    \]
    According to this decomposition we can write
    $\overline{\mu}
    =\overline{\mu}_1+\overline{\mu}_2$
    and similarly
    $\nu=\nu_1+\nu_2$
    for any $\nu\in B(G,\mu)$.
    The condition
    $\overline{\mu}-\nu\in
    (\Delta_{\ast,0})_\mathbb{Q}$
    shows that
    $\overline{\mu}_2=\nu_2$.
    Moreover,
    by definition of
    $(\Phi_0^\ast)_\mathbb{Q}^\perp$,
    we have
    $\langle\nu_2,\alpha\rangle=0$
    for any $\alpha\in\widetilde{\Delta}_0^\ast$.
    Then we can rewrite
    $B(G,\mu)$ in the following manners:
    \begin{align*}
    &
    \{
    \overline{\mu}-\nu'\in
    X_{\ast,\mathrm{dm}}(T_0)_\mathbb{Q}^+|
    0\leq\nu'
    \in
    \mathbb{Q}_{\geq0}
    \langle
    \Delta_{\ast,0}
    \rangle,
    \text{ and }
    \forall
    \alpha\in\widetilde{\Delta}_0^\ast
    \text{ with }
    \langle\overline{\mu}-\nu',\alpha\rangle\neq0,
    \langle
    \nu',\widetilde{w}_\alpha
    \rangle
    \in\mathbb{N}
    \}
    \\
    =
    &
    \{
    \nu_1+\overline{\mu}_2
    \in
    X_{\ast,\mathrm{dm}}(T_0)_\mathbb{Q}|
    0\leq\overline{\mu}_1-\nu_1\in
    \mathbb{Q}_{\geq0}
    \langle
    \Delta_{\ast,0}
    \rangle,
    \text{ and }
    \forall
    \alpha\in\widetilde{\Delta}_0^\ast
    \text{ with }
    \langle \nu_1,\alpha\rangle\neq0,
    \langle
    \overline{\mu}_1-\nu_1,w_\alpha
    \rangle
    \in
    \mathbb{N}
    \}.
    \end{align*}
    From this last description
    we see that the set $B(G,\mu)$
    depends only on the root system
    $(\mathbb{Z}
    \langle
    \Phi_{\ast,0}
    \rangle,
    \Phi_{\ast,0},
    \mathbb{Z}
    \langle
    \Phi_0^\ast
    \rangle,
    \Phi_0^\ast)$
    generated by the roots and coroots
    $\Phi_{\ast,0},\Phi_0^\ast$.

    For reasons that we will explain later,
    we will be interested in the difference
    $\overline{\mu}-\nu
    =\overline{\mu}_1-\nu_1$
    for
    $\nu\in B(G,\mu)$.
    Using the decomposition of root system
    $((\Phi_{\ast,0})_\mathbb{Z},\Phi_{\ast,0},
    (\Phi_0^\ast)_\mathbb{Z},\Phi_0^\ast)$,
    we can assume that
    the root system is one of the nine classes of
    indecomposable root systems
    $A_n,B_n,C_n,D_n,E_6,E_7,E_8,F_4,G_2$.
    
    Recall the conditions posed on the cocharacter
    $\mu$ associated to the Shimura datum
    $(G,X)$:
    the cocharacter
    $\mu\colon
    \mathbb{G}_m
    \rightarrow
    G$ is minuscule(\cite[§1.2]{Deligne1977}).
    More precisely,
    let's
    suppose that the adjoint group has a decomposition into
    simple factors
    $G^\mathrm{ad}
    =\prod_{i=1}^nG^{(i)}$
    and the images of the torus, resp. Borel subgroup,
    $T_0,B_0$
    in each factor $G^{(i)}$
    are denoted by
    $T^{(i)}_0$, resp.,
    $B^{(i)}_0$.
    Let
    $\mu^{(i)}$
    be the unique
    $B^{(i)}_0$-dominant cocharacter
    of
    $T^{(i)}_0$
    conjugate to the image of
    $\mu$ in $G^{(i)}$.
    Then $\mu$ being minuscule means that
    there exists at most one simple root
    $\alpha\in\widetilde{\Delta}^{(i),\ast}_0$
    such that
    $\langle\mu^{(i)},\alpha\rangle>0$,
    in which case
    $\langle\mu^{(i)},\alpha\rangle=1$ and
    $\alpha$ is special
    (a simple root which appears with multiplicity one in
    the highest weight of $(T^{(i)}_0,B^{(i)}_0)$).
    Suppose such a special root $\alpha$ exists,
    then we have necessarily the fundamental co-weight
    $
    \mu^{(i)}=w_{\alpha^\vee}.
    $
    Otherwise,
    we have
    $\mu^{(i)}=0$.

    \begin{proposition}\label{maximal element in B(G,mu)}
    	Suppose that
    	$(\mathbb{Z}
    	\langle
    	\Phi_{\ast,0}
    	\rangle,
    	\Phi_{\ast,0},
    	\mathbb{Z}
    	\langle
    	\Phi_0^\ast
    	\rangle,
    	\Phi_0^\ast)$
    	is an indecomposable root system
    	such that
    	$\mu=w_{\alpha^\vee}$
    	for some special root
    	$\alpha\in\widetilde{\Delta}_0^\ast$.
    	Then the maximal element in
    	$B(G,\mu)\backslash\{\overline{\mu}\}$
    	is
    	$\overline{\mu}-\frac{1}{2}\alpha^\vee$.
    \end{proposition}
    \begin{proof}
    	We proceed the proof case by case.
    	\begin{enumerate}
    		\item 
    		$G$ is of type $A_n$.
    		We may identify $G$ with $\mathrm{SL}_{n+1}$
    		and
    		$B$ is the standard subgroup of upper triangular
    		matrices.
    		Concerning the root datum,
    		for a diagonal element
    		$t=\mathrm{diag}(t_1,\cdots,t_{n+1})\in
    		\mathrm{SL}_{n+1}$,
    		let
    		$\varepsilon_i\in
    		X^\ast(T)$ denote the character
    		$\alpha_i(t)=t_i$
    		($i=1,\cdots,n+1$).
    		Write
    		$\alpha_i=\varepsilon_i-\varepsilon_{i+1}$
    		($i=1,\cdots,n$),
    		then the set of simple roots is
    		\[
    		\widetilde{\Delta}^\ast=
    		\{\alpha_1,\cdots,\alpha_n\}
    		\]
    		which is
    		a basis for
    		$X^\ast(T)$,
    		$\Phi^\ast
    		=\{
    		\varepsilon_i-\varepsilon_j|i\neq j
    		\}$.
    		Similarly,
    		let
    		$\varepsilon_i^\vee$
    		denote the cocharacter
    		$\varepsilon_i^\vee(x)=
    		\mathrm{diag}(1_{i-1},x,1_{n+1-i})\in
    		\mathrm{GL}_{n+1}$
    		and
    		then
    		$\alpha_i=\varepsilon_i^\vee-\varepsilon_{i+1}^\vee$,
    		\[
    		\widetilde{\Delta}^\ast
    		=
    		\{
    		\alpha_1^\vee,\cdots,\alpha_n^\vee
    		\},
    		\]
    		with
    		$\Phi^\ast
    		=
    		\{
    		\varepsilon_i^\vee-\varepsilon_j^\vee|i\neq j
    		\}$.
    		Moreover, it is easy to see that
    		the fundamental (co-)weights are
    		given by
    		\[
    		w_{\alpha_i}
    		=\sum_{j=1}^i
    		\varepsilon_j,
    		\quad
    		w_{\alpha_i^\vee}
    		=
    		\sum_{j=1}^i
    		\varepsilon_j^\vee
    		-\frac{i}{n+1}\sum_{j=1}^{n+1}\varepsilon_j^\vee
    		(i=1,\cdots,n).
    		\]
    		Moreover,
    		the highest weight is
    		$\varepsilon_1-\varepsilon_{n+1}
    		=\alpha_1+\cdots+\alpha_n$.
    		The special roots
    		are therefore
    		$\alpha_1,\cdots,\alpha_n$.
    		Suppose now that
    		$\overline{\mu}
    		=w_{\alpha_k^\vee}$
    		for some
    		$k=1,2,\cdots,n$.
    		For any
    		$\overline{\mu}-\delta
    		\in
    		B(G,\mu)$,
    		we write
    		\[
    		\delta=
    		a_1w_{\alpha_1^\vee}+\cdots+a_nw_{\alpha_n^\vee}
    		=
    		f_1\varepsilon_1^\vee+\cdots+
    		f_{n+1}\varepsilon_{n+1}^\vee
    		=
    		d_1\alpha_1^\vee+\cdots+d_n\alpha_n^\vee
    		\]
    		($a_i,f_j,d_l\in\mathbb{Q}$).
    		We necessarily have
    		$f_1+\cdots+f_{n+1}=0$.
    		Spelling out the condition in $B(G,\mu)$,
    		we see that the coefficients
    		$a_i,f_j$ should satisfy
    		\[
    		a_1=f_1-f_2\leq0,
    		\cdots,
    		a_{k-1}=f_{k-1}-f_k\leq0,
    		a_k=f_k-f_{k+1}\leq1,
    		\]
    		\[
    		a_{k+1}=f_{k+1}-f_{k+2}\leq0,
    		\cdots,
    		a_n=f_n-f_{n+1}\leq0,
    		\]
    		\[
    		f_1\geq0,
    		\cdots
    		f_1+\cdots+f_i\geq0,
    		\cdots,
    		f_1+\cdots+f_n\geq0
    		\]
    		and moreover
    		(1)
    		if
    		$a_i\neq0$
    		(with $i\neq k$),
    		we have
    		$f_1+\cdots+f_i\in\mathbb{N}$
    		and 
    		(2) if
    		$a_k\neq1$,
    		we have
    		$f_1+\cdots,f_k\in\mathbb{N}$.
    		
    		Suppose that the numbers
    		$a_i\neq0$ if and only if
    		$i=i_1,i_2,\cdots,i_N$
    		with
    		$i_1<i_2<\cdots<i_N$ elements in
    		$\{1,2,\cdots,n\}$.
    		\begin{enumerate}
    			\item 
    			If $k=i_j$ for some $j=1,\cdots,N$,
    			then we have
    			\[
    			0\leq
    			f_1=f_2=\cdots=f_{i_1}
    			<
    			f_{i_1+1}=\cdots=f_{i_2}
    			<\cdots
    			\]
    			\[
    			<
    			f_{i_{j-1}+1}=f_{i_{j-1}+2}=
    			\cdots
    			f_{i_j}
    			<
    			1+f_{i_j+1}=1+f_{i_j+2}
    			=\cdots
    			=
    			1+f_{i_{j+1}}
    			<
    			\cdots
    			\]
    			\[
    			<
    			1+f_{i_N+1}
    			=
    			1+f_{i_N+2}
    			=
    			\cdots
    			=1+f_{n+1}
    			<1
    			\]
    			such that
    			\[
    			d_{i_1}=\ell_1\in\mathbb{N},
    			\cdots,
    			d_{i_j}=\ell_j\in\mathbb{N},
    			d_{i_{j+1}}
    			=\ell_{j+1}\in\mathbb{N},
    			\cdots,
    			d_{i_N}
    			=\ell_N\in\mathbb{N}.
    			\]
    			One verifies that
    			$\delta=d_1\alpha_1^\vee+\cdots+d_n\alpha_n^\vee
    			>\frac{1}{2}\alpha_k^\vee$
    			unless $\delta=0$.

    			\item 
    			If $i_j<k<i_{j+1}$
    			for some $j=0,1,\cdots,N+1$
    			(with convention $i_0=0,i_{N+1}=n+1$).
    			In this case,
    			$\overline{\mu}-
    			\delta\in
    			B(G,\mu)$
    			implies that
    			\[
    			0\leq
    			f_1=\cdots=f_{i_1}
    			<
    			f_{i_1+1}=\cdots=f_{i_2}
    			<
    			\cdots
    			\]
    			\[
    			<
    			f_{i_{j-1}+1}=
    			\cdots
    			=f_{i_j}
    			<
    			f_{i_j+1}=\cdots=f_k=f_{k+1}+1
    			=\cdots=
    			f_{i_{j+1}}+1
    			<\cdots
    			\]
    			\[
    			<
    			f_{i_{j+1}+1}+1
    			=
    			f_{i_{j+1}+2}+1
    			=
    			\cdots
    			<
    			f_{i_N+1}+1=
    			f_{i_N+2}+1
    			=
    			\cdots
    			<1
    			\]
    			such that
    			\[
    			d_{i_1}=\ell_1\in\mathbb{N},
    			\cdots,
    			d_{i_j}=\ell_j\in\mathbb{N},
    			d_{i_{j+1}}=\ell_{j+1}\in\mathbb{N},
    			\cdots,
    			d_{i_N}=\ell_N\in\mathbb{N}.
    			\]
    			One verifies that
    			$\delta=d_1\alpha_1^\vee+\cdots+d_n\alpha_n^\vee
    			\geq\frac{1}{2}\alpha_k^\vee$
    			and we have
    			$\overline{\mu}
    			-\frac{1}{2}\alpha_k^\vee
    			\in
    			B(G,\mu)$.
    		\end{enumerate}

    	    Thus we conclude in the case
    	    $G$ of type $A_n$
    	    that the maximal element in
    	    $B(G,\mu)\backslash\{\overline{\mu}\}$
    	    is indeed
    	    $\overline{\mu}-\frac{1}{2}\alpha_k^\vee$.

    		\item 
    		$G$ is of type $B_n$.
    		We identify $G$ with the special orthogonal group
    		$\mathrm{SO}_{2n+1}$
    		associated to
    		the symmetric matrix
    		$\mathrm{antidiag}(1,1,\cdots,1)
    		\in
    		\mathrm{GL}_{2n+1}$.
    		Write the diagonal matrix
    		$t=\mathrm{diag}(t_1,\cdots,t_n,1,t_n^{-1},
    		\cdots,t_1^{-1})$.
    		Denote by
    		$\varepsilon_i
    		\in
    		X^\ast(T)$
    		the character
    		$\varepsilon_i(t)=t_i$
    		($i=1,\cdots,n$)
    		and set
    		$\alpha_i=\varepsilon_i-\varepsilon_{i+1}$
    		($i=1,\cdots,n-1$).
    		Then
    		\[
    		\widetilde{\Delta}^\ast
    		=
    		\{
    		\alpha_1,\cdots,\alpha_{n-1},\varepsilon_n
    		\}
    		\]
    		and
    		$\Phi^\ast
    		=\{
    		\pm(\varepsilon_i\pm\varepsilon_j)|i\neq j
    		\}
    		\cup
    		\{
    		\pm\varepsilon_i
    		\}$.
    		Similarly,
    		denote by
    		$\varepsilon_i^\vee
    		\in
    		X_\ast(T)$
    		the cocharacter
    		$\varepsilon_i^\vee(x)=
    		\mathrm{diag}(1_{i-1},x,1_{2n+3-2i},x^{-1},1_{i-1})
    		\in
    		\mathrm{GL}_{2n+1}$
    		and set
    		$\alpha_i^\vee
    		=\varepsilon_i^\vee-\varepsilon_{i+1}^\vee$.
    		Then
    		\[
    		\widetilde{\Delta}_\ast
    		=\{
    		\alpha_1^\vee,\cdots,\alpha_{n-1}^\vee,
    		2\varepsilon_n^\vee
    		\}
    		\]
    		and
    		$\Phi_\ast
    		=\{
    		\pm(\varepsilon_i^\vee\pm\varepsilon_j^\vee)|i\neq j
    		\}
    		\cup
    		\{
    		\pm2\varepsilon_i^\vee
    		\}$.
    		The fundamental (co-)weights are given by
    		\[
    		w_{\alpha_i}
    		=\sum_{j=1}^i\varepsilon_j,
    		\quad
    		w_{\varepsilon_n}
    		=\frac{1}{2}\sum_{j=1}^n\varepsilon_j;
    		\quad
    		w_{\alpha_1^\vee}
    		=
    		\sum_{j=1}^i\varepsilon_j^\vee,\quad
    		w_{2\varepsilon_n^\vee}
    		=\sum_{j=1}^n\varepsilon_j^\vee.
    		\]
    		In this case,
    		the highest weight is
    		$\varepsilon_1+\varepsilon_2
    		=\alpha_1+2\alpha_2+2\alpha_3+
    		\cdots+2\alpha_{n-1}+2\varepsilon_n$,
    		thus the special root is
    		$\alpha_1$ and therefore
    		$\overline{\mu}=w_{\alpha_1^\vee}
    		=\varepsilon_1^\vee$.
    		The computation is exactly the same
    		as above
    		(even simpler):
    		we write
    		\[
    		\delta
    		=
    		a_1w_{\alpha_1^\vee}+\cdots+a_nw_{\alpha_n^\vee}
    		=
    		d_1\alpha_1^\vee+\cdots+d_n\alpha_n^\vee
    		\]
    		with
    		$a_i,d_j\in\mathbb{Q}$.
    		Thus we have
    		the following conditions on
    		these coefficients
    		\[
    		a_1\leq1,
    		\quad
    		a_2,\cdots,a_n\leq0,
    		\quad
    		d_1,\cdots,d_n\geq0
    		\]
    		and
    		(1)
    		if
    		$a_1\neq0$ ($i\neq1$),
    		then
    		$d_i\in\mathbb{N}$;
    		(2)
    		if
    		$a_1\neq1$,
    		then
    		$d_1\in\mathbb{N}$.
    		Then one verifies that
    		$\delta=\frac{1}{2}\alpha_1^\vee$
    		satisfies all these conditions and moreover,
    		$\overline{\mu}-\delta$
    		is the maximal element in
    		$B(G,\mu)\backslash\{\overline{\mu}\}$.

    		\item 
    		$G$ is of type $C_n$
    		(the dual case of the preceding one).
    		We identify $G$ with the special
    		symplectic group
    		$\mathrm{Sp}_{2n}$
    		of the symplectic form
    		$\mathrm{antidiag}(J,-J)$
    		with
    		$J=\mathrm{antidiag}(1,1,\cdots,1)
    		\in
    		\mathrm{GL}_n$.
    		Write the diagonal matrix
    		$t=\mathrm{diag}(t_1,\cdots,t_n,t_n^{-1},\cdots,t_1^{-1})$.
    		Denote by
    		$\varepsilon_i
    		\in
    		X^\ast(T)$
    		the character
    		$\varepsilon_i(t)=t_i$
    		and
    		$\alpha_i=\varepsilon_i-\varepsilon_{i+1}$.
    		Then
    		\[
    		\widetilde{\Delta}^\ast
    		=\{
    		\alpha_1,\cdots,\alpha_{n-1},2\varepsilon_n
    		\}
    		\]
    		and
    		$\Phi^\ast
    		=
    		\{
    		\pm(\varepsilon_i\pm\varepsilon_j)|i\neq j
    		\}
    		\cup
    		\{
    		\pm2\varepsilon_i
    		\}$.
    		Similarly,
    		denote by
    		$\varepsilon_i^\vee
    		\in
    		X_\ast(T)$
    		the cocharacter
    		$\varepsilon^\vee_i(x)
    		=
    		\mathrm{diag}(1_{i-1},x,1_{2n+2-2i},x^{-1},1_{i-1})$
    		and
    		$\alpha_i^\vee
    		=\varepsilon_i^\vee-\varepsilon_{i+1}^\vee$.
    		Then
    		\[
    		\widetilde{\Delta}_\ast
    		=\{
    		\alpha_1^\vee,\cdots,\alpha_{n-1}^\vee,
    		\varepsilon_n^\vee
    		\}
    		\]
    		and
    		$\Phi_\ast
    		=\{
    		\varepsilon_i^\vee-\varepsilon_j^\vee|i\neq j
    		\}
    		\cup
    		\{
    		\pm\varepsilon_i
    		\}$.
    		The fundamental (co-)weights are given by
    		\[
    		w_{\alpha_i}
    		=
    		\sum_{j=1}^i\varepsilon_j,
    		\quad
    		w_{2\varepsilon_n}
    		=
    		\sum_{j=1}^n\varepsilon_j;
    		\quad
    		w_{\alpha_i^\vee}
    		=
    		\sum_{j=1}^i\varepsilon_j^\vee,
    		\quad
    		w_{\varepsilon_n^\vee}
    		=
    		\frac{1}{2}\sum_{j=1}^n\varepsilon_j^\vee.
    		\]
    		
    		Note that the highest weight is
    		$2\varepsilon_1
    		=2\alpha_1+\cdots+2\alpha_{n-1}+(2\varepsilon_n)$
    		and thus the special root is
    		$2\varepsilon_n$,
    		therefore we have
    		$\overline{\mu}
    		=\frac{1}{2}\sum_{j=1}^n\varepsilon_j^\vee
    		=w_{\varepsilon_n^\vee}$.
    		To simplify notations,
    		let's write
    		$\alpha_n=2\varepsilon_n$ and
    		$\alpha_n^\vee=\varepsilon_n^\vee$.
    		We write
    		\[
    		\delta
    		=
    		a_1w_{\alpha_1^\vee}+\cdots+a_nw_{\alpha_n^\vee}
    		=
    		d_1\alpha_1^\vee+\cdots+d_n\alpha_n^\vee
    		\]
    		with
    		$a_i,d_j\in\mathbb{Q}$.
    		The conditions put
    		these coefficients become
    		\[
    		a_1,\cdots,a_{n-1},a_n-1\leq0,
    		\quad
    		d_1,\cdots,d_n\geq0
    		\]
    		and
    		(1)
    		if $a_i\neq0$
    		($i\neq n$),
    		we have
    		$d_i\in\mathbb{N}$;
    		(2)
    		if $a_n-1=0$,
    		we have
    		$d_n\in\mathbb{N}$.
    		One verifies easily that
    		$\delta=\frac{1}{2}\alpha_n^\vee$
    		satisfies all these conditions
    		and
    		$\overline{\mu}-\delta
    		\in
    		B(G,\mu)$.

    		\item 
    		$G$ is of type $D_n$.
    		We identify
    		$G$ with the orthogonal group
    		$\mathrm{SO}_{2n}$ of the quadratic form
    		$\mathrm{antidiag}(1,\cdots,1)$.
    		Write the diagonal matrix
    		$t=\mathrm{diag}(t_1,\cdots,t_n,t_n^{-1},
    		\cdots,t_1^{-1})$.
    		Denote by
    		$\varepsilon_i
    		\in
    		X^\ast(T)$
    		the character
    		$\varepsilon_i(t)=t_i$ and
    		$\alpha_i=\varepsilon_i-\varepsilon_{i+1}$
    		and
    		$\alpha_{n-1}^+=\varepsilon_{n-1}+\varepsilon_n$.
    		Then
    		\[
    		\widetilde{\Delta}^\ast=
    		\{
    		\alpha_1,\cdots,\alpha_{n-1},\alpha_{n-1}^+
    		\}
    		\]
    		and
    		$\Phi^\ast
    		=\{
    		\pm(\varepsilon_i\pm\varepsilon_j)|i\neq j
    		\}$.
    		Similarly,
    		denote by
    		$\varepsilon_i^\vee
    		\in
    		X_\ast(T)$
    		the cocharacter
    		sending $x$
    		to
    		the element
    		$
    		\mathrm{diag}(1_{i-1},x,1_{2n+2-2i},x^{-1},1_{i-1})$
    		and
    		$\alpha_i^\vee=\varepsilon_i-\varepsilon_{i+1}$,
    		$(\alpha_n^+)^\vee
    		=\varepsilon_{n-1}+\varepsilon_n$.
    		Then
    		\[
    		\widetilde{\Delta}_\ast
    		=\{
    		\alpha_1^\vee,\cdots,\alpha_{n-1}^\vee,
    		(\alpha_{n-1}^+)^\vee
    		\}
    		\]
    		and
    		$\Phi_\ast
    		=\{
    		\pm(\varepsilon_i^\vee\pm\varepsilon_j^\vee)|i\neq j
    		\}$.
    		The fundamental (co-)weights are given by
    		\[
    		w_{\alpha_i}
    		=
    		\sum_{j=1}^i\varepsilon_j,\quad
    		w_{\alpha_{n-1}^+}
    		=
    		\frac{1}{2}\sum_{j=1}^n\varepsilon_j;
    		\quad
    		w_{\alpha_i^\vee}
    		=
    		\sum_{j=1}^i\varepsilon_j^\vee,\quad
    		w_{(\alpha_{n-1}^+)^\vee}
    		=
    		\frac{1}{2}\sum_{j=1}^n\varepsilon_j^\vee;
    		\]
    		
    		Note that the highest weight is
    		$\varepsilon_1+\varepsilon_2
    		=\alpha_1+2\alpha_2+\cdots+2\alpha_{n-1}+
    		\alpha_{n-1}+\alpha_{n-1}^+$,
    		thus the special roots are
    		$\alpha_1,\alpha_{n-1},\alpha_{n-1}^+$,
    		and therefore
    		we have three possibilities
    		$\overline{\mu}
    		=\varepsilon_1^\vee
    		(=w_{\alpha_1^\vee})$,
    		$\sum_{j=1}^{n-1}\varepsilon_j^\vee
    		(=w_{\alpha_{n-1}^\vee})$
    		or
    		$\frac{1}{2}\sum_{j=1}^n\varepsilon_j^\vee
    		(=w_{(\alpha_{n-1}^+)^\vee})$.
    		To simplify notations,
    		let's put
    		$\alpha_n=\alpha_{n-1}^+$ and
    		$\alpha_n^\vee=(\alpha_{n-1}^+)^\vee$.
    		We write
    		\[
    		\delta
    		=
    		a_1w_{\alpha_1^\vee}
    		+
    		\cdots
    		+a_nw_{\alpha_n^\vee}
    		=
    		d_1\alpha_1^\vee+\cdots
    		+d_n\alpha_n^\vee
    		\]
    		with
    		$a_i,d_j\in\mathbb{Q}$.
    		We discuss the three cases one by one
    		\begin{enumerate}
    			\item 
    			For $\overline{\mu}=w_{\alpha_1^\vee}$.
    			The conditions on these coefficients are
    			\[
    			a_1-1,a_2,\cdots,a_n\leq0,
    			\quad
    			d_1,\cdots,d_n\geq0
    			\]
    			and
    			(1)
    			if $a_i\neq0$
    			($i\neq1$),
    			we have
    			$d_i\in\mathbb{N}$;
    			(2)
    			if $a_1-1\neq0$,
    			we have
    			$d_1\in\mathbb{N}$.
    			Again one verifies that
    			$\delta=
    			\frac{1}{2}\alpha_1^\vee$
    			satisfies all these conditions and moreover
    			$\overline{\mu}-\delta
    			\in
    			B(G,\mu)\backslash\overline{\mu}$
    			is the maximal element.
    			
    			\item 
    			For
    			$\overline{\mu}
    			=w_{\alpha_{n-1}^\vee}$,
    			the conditions on these coefficients are
    			\[
    			a_1,\cdots,a_{n-2},a_{n-1}-1,a_n\leq0,
    			\quad
    			d_1,\cdots,d_n\geq0
    			\]
    			and
    			(1)
    			if
    			$a_i\neq0$
    			($i\neq n-1$),
    			we have
    			$d_i\in\mathbb{N}$;
    			(2)
    			if
    			$a_{n-1}-1\neq0$,
    			we have
    			$d_{n-1}\in\mathbb{N}$.
    			One verifies as above that
    			$\delta=\frac{1}{2}$
    			satisfies all these conditions and
    			$\overline{\mu}-\delta
    			\in
    			B(G,\mu)$
    			is the maximal element.
    			
    			\item 
    			For
    			$\overline{\mu}
    			=w_{\alpha_n^\vee}$,
    			the conditions on these coefficients are
    			\[
    			a_1,\cdots,a_{n-1},a_n-1\leq0,
    			\quad
    			d_1,\cdots,d_n\geq0
    			\]
    			and
    			(1)
    			if
    			$a_i\neq0$
    			($i\neq n$),
    			we have
    			$d_i\in\mathbb{N}$;
    			(2)
    			if $a_n-1\neq0$,
    			we have
    			$d_n\in\mathbb{N}$.
    			One verifies that
    			$\delta=\frac{1}{2}\alpha_n^\vee=
    			\frac{1}{2}(\alpha_{n-1}^+)^\vee$
    			satisfies all these conditions and
    			$\overline{\mu}-\delta
    			\in
    			B(G,\mu)$
    			is the maximal element.
    		\end{enumerate}

    		\item 
    		$G$ is of type $E_6$.
    		Recall the Dynkin diagram for
    		$E_6$:
    		
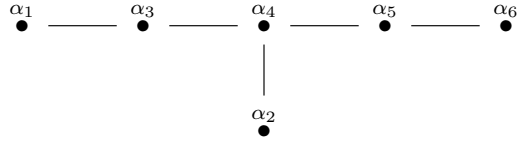
\[
    		\begin{tikzcd}[arrows=-]
    		\stackrel{\alpha_1}{\bullet}
    		\arrow[r]
    		&
    		\stackrel{\alpha_3}{\bullet}
    		\arrow[r]
    		&
    		\stackrel{\alpha_4}{\bullet}
    		\arrow[r]
    		\arrow[d]
    		&
    		\stackrel{\alpha_5}{\bullet}
    		\arrow[r]
    		&
    		\stackrel{\alpha_6}{\bullet}
    		\\
    		&
    		&
    		\stackrel{\alpha_2}{\bullet}
    		\end{tikzcd}
    		\]
    		Then one sees that
    		the highest weight is
    		$\alpha_1+2(\alpha_2+\cdots+\alpha_5)+\alpha_6$,
    		thus the special roots are
    		$\alpha_1,\alpha_6$.
    		If
    		$\overline{\mu}=w_{\alpha_1^\vee}$,
    		one can verify as above that
    		$\overline{\mu}-\frac{1}{2}\alpha_1^\vee
    		\in
    		B(G,\mu)\backslash\{\overline{\mu}\}$
    		is the maximal element.
    		Similarly,
    		if
    		$\overline{\mu}=w_{\alpha_6^\vee}$,
    		one verifies that
    		$\overline{\mu}-\frac{1}{2}\alpha_6^\vee
    		\in
    		B(G,\mu)\backslash\{\overline{\mu}\}$
    		is the maximal element.
    		
    		\item 
    		$G$ is of type $E_7$.
    		The Dynkin diagram for $E_7$ is
    		
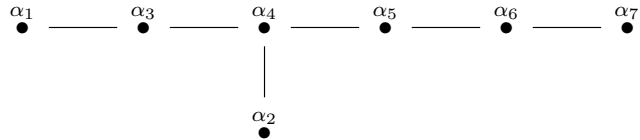
\[
    		\begin{tikzcd}[arrows=-]
    		\stackrel{\alpha_1}{\bullet}
    		\arrow[r]
    		&
    		\stackrel{\alpha_3}{\bullet}
    		\arrow[r]
    		&
    		\stackrel{\alpha_4}{\bullet}
    		\arrow[r]
    		\arrow[d]
    		&
    		\stackrel{\alpha_5}{\bullet}
    		\arrow[r]
    		&
    		\stackrel{\alpha_6}{\bullet}
    		\arrow[r]
    		&
    		\stackrel{\alpha_7}{\bullet}
    		\\
    		&
    		&
    		\stackrel{\alpha_2}{\bullet}
    		\end{tikzcd}
    		\]
    		One verifies that the special root is
    		$\alpha_1$ and we have
    		$\overline{\mu}=w_{\alpha_1}$.
    		As above,
    		we have
    		$\overline{\mu}-\frac{1}{2}\alpha_1^\vee
    		\in
    		B(G,\mu)\backslash\{\overline{\mu}\}$
    		is the maximal element.

    		\item 
    		For other types
    		$E_8,F_4,G_2$,
    		there are no special roots.   		
    		
    	\end{enumerate}
    \end{proof}

    \begin{proof}
    	(of Proposition \ref{Hasse invariant is reduced})
    	We follow the strategy in
    	\cite[A.3]{Pilloni2012}.
    	Consider a geometric point
    	$x\in
    	\mathcal{N}^b$
    	for a maximal element
    	$b\in B(G,\mu)\backslash\{\overline{\mu}\}$
    	and
    	$\mathcal{A}_x$ the
    	abelian scheme over $x$,
    	$(\mathbb{D},\mathrm{Fr})$
    	the Dieudonn\'{e} crystal with $G$-structure of
    	$\mathcal{A}_x$ evaluated at the dual numbers
    	$\overline{\mathbb{F}}_p[\varepsilon]$,
    	which is of $\overline{\mathbb{F}}_p[\varepsilon]$-rank
    	$\mathrm{dim}(V)$.
    	By the above proposition,
    	we see that in the decomposition
    	$G^\mathrm{ad}
    	=\prod_{i=1}^n
    	G^{(i)}$,
    	among
    	the components
    	$b^{(i)}$
    	of the dominant co-character
    	$b
    	\in
    	X_{\ast,\mathrm{dm}}(T_0)_\mathbb{Q}$,
    	exactly one of them
    	(say, the $i$-th component) is of the form
    	$\mu^{(i)}-\frac{1}{2}\alpha^\vee$
    	where
    	$\alpha^\vee$
    	is the notation in
    	the above proposition.
    	By the embedding
    	$(G,X)
    	\hookrightarrow
    	(\mathrm{GSp}(V,\psi),S^\pm)$,
    	we see that
    	in the copy
    	$\mathfrak{s}_\alpha$
    	of
    	$\mathfrak{sl}_2$
    	inside the Lie algebra
    	$\mathfrak{g}(\overline{\mathbb{F}}_p)$ of
    	$G(\overline{\mathbb{F}}_p)$,
    	one can choose basis
    	$\mathcal{B}=(e_1,e_2,\cdots,
    	e_n,e_1^\ast,\cdots,e_n^\ast)$
    	of
    	$\Lambda$
    	such that
    	the projection of $\mu$
    	to this copy is of the form
    	$\begin{pmatrix}
    	0 & 0 \\
    	0 & 1
    	\end{pmatrix}$
    	while the projection of
    	$b$ to this copy is of the form
    	$\begin{pmatrix}
    	0 & 1 \\
    	0 & 0
    	\end{pmatrix}$.
    	Suppose that
    	$\mathfrak{s}_\alpha$
    	acts on the subspace
    	$\overline{\Lambda}_\alpha
    	=\overline{\mathbb{F}}_p(e_n,e_n^\ast)$
    	of
    	$\Lambda_{\overline{\mathbb{F}}_p}$.
    	We identify
    	$\Lambda_{\overline{\mathbb{F}}_p}$
    	with
    	$\mathbb{D}$
    	thus
    	$\overline{\Lambda}_\alpha$
    	with a subspace of
    	$\mathbb{D}$.
    	Then the Hodge filtration
    	on
    	$\overline{\Lambda}_\alpha$
    	is given by
    	$\mathrm{Fil}^1(\overline{\Lambda}_\alpha)
    	=\overline{\mathbb{F}}_p(e_n)$.
    	Now the same argument as in
    	\cite[Th\'{e}or\`{e}me A.4]{Pilloni2012}
    	shows that
    	the Hasse invariant
    	defines a non-zero
    	linear form on
    	the $\overline{\mathbb{F}}_p$-vector space
    	$X_x(\overline{\mathbb{F}}_p[\varepsilon])$
    	of liftings from $x$ to
    	$\overline{\mathbb{F}}_p[\varepsilon]$.
    	At last note that
    	the union
    	$\cup_b\mathcal{N}^b$
    	with $b$ running through all the maximal elements
    	in
    	$B(G,\mu)\backslash\{\overline{\mu}\}$
    	is open dense in the
    	non $\mu$-ordinary locus
    	$Sh^{n-\mu}_1$
    	(\textit{cf.}
    	\cite[Theorem A]{Zhang2018}).
    \end{proof}

    \section{Automorphic forms}
    \subsection{Some representations}
    \label{Notations on the Shimura datum}
    We first fix some notations.
    Recall that we have fixed a toroidal
    compactification
    $Sh^\Sigma$
    of the integral model $Sh$ over
    $\mathbb{W}$
    and we have the Hodge line bundles
    $\omega_{Sh^\Sigma}$
    and
    $\omega_{Sh}$
    over these spaces.
    We fix a point
    $x\in X$
    and the associated cocharacter
    $\nu_x$
    of
    $G_\mathbb{C}$,
    which is actually defined over the reflex field
    $E$.
    We write
    $P_V$
    for the parabolic subgroup
    of $\mathrm{GSp}(V,\psi)_{E}$
    stabilizing
    the Hodge filtration of the Hodge structure
    $\mathrm{Ad}\circ(\xi\circ x)$,
    $L_V$ the Levi subgroup of $P_V$
    and
    $U_V$
    the unipotent radical of $P_V$.
    Note that
    $L_V$ is the centralizer of
    the cocharacter
    $\nu_x$
    inside
    $\mathrm{GSp}(V,\psi)_{E}$.

    We fix a basis for
    $V$
    such that the symplectic form
    $\psi$ on
    $V$
    is represented by
    $\begin{pmatrix}
    0 & A \\
    -A & 0
    \end{pmatrix}$
    with
    $A=\text{anti-diag}(1,1,\cdots,1)$.
    Then we write
    $T_V$
    for the maximal torus of
    $\mathrm{GSp}(V,\psi)$
    consisting of diagonal matrices,
    $B_V$
    the Borel subgroup of
    $\mathrm{GSp}(V,\psi)$
    consisting of upper triangular matrices.
    We write
    $P_V$
    for some standard parabolic subgroup
    of
    $\mathrm{GSp}(V,\psi)$
    containing
    $B_V$.
    Under this basis,
    $P_V$
    is the maximal parabolic subgroup
    of $\mathrm{GSp}(V,\psi)$
    consisting of matrices of the form
    $\begin{pmatrix}
    A & B \\
    0 & D
    \end{pmatrix}$.
    Now
    we consider a standard
    parabolic subgroup
    $\widetilde{P}_V$ of $\mathrm{GSp}(V,\varphi)$
    such that
    $B_V\subset\widetilde{P}_V\subset P_V$.
    Similarly we have a Levi decomposition
    $\widetilde{P}_V
    =
    \widetilde{L}_V\widetilde{U}_V$.
    We write
    $P_V^\circ$
    for the parabolic subgroup of
    $\mathrm{GSp}(V,\psi)$
    opposite to $P_V$
    and similarly
    $\widetilde{P}_V^\circ$,
    opposite to
    $\widetilde{P}_V$.
    We write
    $\widetilde{P}_V^\mathrm{der}$
    for the derived subgroup of
    $\widetilde{P}_V$
    and we put
    $\widetilde{T}_{\widetilde{P}_V}
    =
    \widetilde{P}_V/\widetilde{P}_V^\mathrm{der}
    $.
    Then we set
    \[
    ?
    =
    G\cap
    ?_V,
    \quad
    \text{with}
    \quad
    ?=T,B,P,L,U,
    \widetilde{P},
    \widetilde{L},
    \widetilde{U},
    \widetilde{P}^\circ,
    \widetilde{U}^\circ,
    \widetilde{P}^\mathrm{der},
    \widetilde{T}_{\widetilde{P}}
    \quad
    \text{ and}
    \quad
    \widetilde{T}_{\widetilde{P}}
    =
    G\cap
    \widetilde{T}_{\widetilde{P}_V}
    =
    \widetilde{P}/\widetilde{P}^\mathrm{der}
    \]
    for the corresponding subgroups relative to $G$.

    \begin{remark}\label{integral model for L}
    	Note in particular that
    	$L$ is the centralizer of the cocharacter
    	$\nu_x$ inside
    	$G_{\mathbb{Q}}$.
    	Recall
    	we have fixed a $\mathbb{Z}_p$-model for
    	$G_{\mathbb{Q}_p}$
    	(see the paragraph after Hypothesis \ref{G is unramified at p}).
    	We can choose a representative
    	of
    	$\nu_x$
    	defined over $\mathbb{Z}_p$
    	(which is denoted by the same letter $\nu_x$).
    	This gives rise to a
    	$\mathbb{Z}_p$-model of
    	the Levi-subgroup
    	$L$.
    	In the following
    	we will also write this $\mathbb{Z}_p$-model by
    	$L$,
    	when no confusion is possible
    	(see also \cite[2.1.5]{GoldringKoskivirta2016}).
    \end{remark}

    We also write
    $\widetilde{P}_L
    =\widetilde{P}\cap L$
    and thus
    $\widetilde{T}_{\widetilde{P}}=
    \widetilde{P}_L/
    \widetilde{P}_L^\mathrm{der}$.
    Moreover,
    the natural inclusion
    $B\subset \widetilde{P}$
    induces a quotient map
    $
    T\rightarrow \widetilde{T}_{\widetilde{P}}.
    $
    Thus we can view
    the character group
    $X^\ast(\widetilde{T}_{\widetilde{P}})$
    as a subgroup of
    $X^\ast(T)$:
    $X^\ast(\widetilde{T}_{\widetilde{P}})
    \subset
    X^\ast(T)$.
    Similarly,
    we view
    the cocharacter group
    $X_\ast(\widetilde{T}_{\widetilde{P}})$
    as a quotient group of
    $X_\ast(T)$:
    $X_\ast(\widetilde{T}_{\widetilde{P}})
    \twoheadrightarrow
    X_\ast(T)$.
    We then put
    \begin{align*}
    X^\ast_\mathrm{dm}(\widetilde{T}_{\widetilde{P}})
    &
    :=X^\ast(\widetilde{T}_{\widetilde{P}})\cap
    X^\ast_\mathrm{dm}(T),
    \\
    X_{\ast,\mathrm{dm}}(\widetilde{T}_{\widetilde{P}})
    &
    :=
    \mathrm{Im}
    (X_{\ast,\mathrm{dm}}(T)
    \rightarrow
    X_\ast(\widetilde{T}_{\widetilde{P}})).
    \end{align*}
    
    We write
    $\lambda_{\mathrm{GSp}(V,\psi)}$
    for the character of
    $T_V$
    sending
    $(t_1,t_2,\cdots,t_n,\nu/t_n,\nu/t_{n-1},\cdots,\nu/t_1)$
    to
    $t_1t_2\cdots t_N$
    and then we put
    \begin{equation}\label{character lambda_G}
    \lambda_G
    =
    \lambda_{\mathrm{GSp}(V,\psi)}|_{T}
    \in
    X^\ast(T).
    \end{equation}

    \begin{definition}
    	Write
    	$B_L=B\cap L$.
    	For any $\mathcal{O}_\mathfrak{p}$-algebra 
    	$A$
    	and any character
    	$\lambda\in
    	X^\ast(T)$,
    	we write
    	\[
    	R_A[\lambda^{-1}]
    	:=
    	\mathrm{Ind}^{L}_{B_L}(\lambda^{-1})_{/A}
    	\]
    	for the algebraic induction
    	defined as in
    	\cite[§I.3.3]{Jantzen2003},
    	which is an $L$-equivariant
    	line bundle on the flag variety
    	$L/B_L$.
    	More concretely,
    	for any $A$-algebra $A'$,
    	$R_A[\lambda^{-1}](A')$
    	is the set of rational morphisms
    	$f\colon
    	L_{/A'}
    	\rightarrow
    	\mathbb{G}_{a/A'}$
    	such that
    	$f(gh)=\lambda^{-1}(h)f(g)$
    	for any
    	$g\in L(A')$,
    	$h\in B_L(A')$.
    \end{definition}
    For a finite flat $\mathcal{O}_\mathfrak{p}$-algebra
    $A$,
    we have a canonical isomorphism
    \[
    R_A[\lambda^{-1}]\otimes_AA[1/p]
    \simeq R_{A[1/p]}[\lambda^{-1}],
    \]
    thus we have a natural inclusion
    $R_A[\lambda^{-1}]
    \hookrightarrow
    R_{A[1/p]}[\lambda^{-1}]$
    and we can view the first space
    as a lattice in the
    second space.
    
    \begin{definition}
    	For a finite flat
    	$\mathcal{O}_\mathfrak{p}$-algebra $A$,
    	we define an $A$-module
    	\[
    	R^\mathrm{top}_A[\lambda^{-1}]
    	:=
    	\mathrm{Ind}_{B_L
    		(\mathcal{O}_\mathfrak{p})}^{L
    		(\mathcal{O}_\mathfrak{p})}(\lambda^{-1},A)
    	\]
    	to be the set of continuous maps
    	$f\colon
    	L(\mathcal{O}_\mathfrak{p})
    	\rightarrow
    	A$
    	such that
    	$f(gtu)=\lambda^{-1}(t)f(g)$
    	for any
    	$g\in L(\mathcal{O}_\mathfrak{p}),
    	t\in T(\mathcal{O}_\mathfrak{p}),
    	u\in U_L(\mathcal{O}_\mathfrak{p})$.
    	Similarly
    	we define an
    	$A[1/p]$-module
    	\[
    	R^\mathrm{top}_{A[1/p]}[\lambda^{-1}]
    	:=
    	\mathrm{Ind}_{B_L(E_\mathfrak{p})}^{L(E_\mathfrak{p})}
    	(\lambda^{-1},A[1/p])
    	\]
    	to be the set of continuous maps
    	$f\colon
    	L(E_\mathfrak{p})
    	\rightarrow
    	A[1/p]$
    	such that
    	$f(gtu)=\lambda^{-1}(t)f(g)$
    	for any
    	$g\in L(E_\mathfrak{p})$,
    	$t\in T(E_\mathfrak{p})$
    	and
    	$u\in U_L(E_\mathfrak{p})$.

    	We let
    	$L$ act on
    	$R_A[\lambda^{-1}]$
    	by left translation
    	and let
    	$L(\mathcal{O}_\mathfrak{p})$
    	act on
    	$R^\mathrm{top}_A[\lambda^{-1}]$
    	by left translation.

    \end{definition}
    By the Iwasawa decomposition
    $L(E_\mathfrak{p})
    =L(\mathcal{O}_\mathfrak{p})
    B_L(E_\mathfrak{p})$,
    we see that
    an element
    $f\in R_{A[1/p]}^\mathrm{top}[\lambda^{-1}]$
    is determined by its restriction to
    $L(\mathcal{O}_\mathfrak{p})$
    and thus we have a natural inclusion
    \[
    R_A^\mathrm{top}[\lambda^{-1}]
    \hookrightarrow
    R_{A[1/p]}^\mathrm{top}[\lambda^{-1}]
    \]
    therefore
    we can view the first space as a lattice in the
    second space
    by the compactness of
    $L(\mathcal{O}_\mathfrak{p})$.

    There are several morphisms among these representations
    that will be useful later on.
    \begin{enumerate}
    	\item 
    	We have a natural map
    	\[
    	\mathrm{ev}_{\mathcal{O}_\mathfrak{p}}
    	\colon
    	R_A[\lambda^{-1}]
    	\rightarrow
    	R^\mathrm{top}_A[\lambda^{-1}],
    	\quad
    	f\mapsto
    	(
    	f|_{\mathcal{O}_\mathfrak{p}}
    	\colon
    	g\in L(\mathcal{O}_\mathfrak{p})
    	\mapsto
    	f(g)
    	),
    	\]
    	which is simply the evaluation of
    	the algebraic representation
    	$R_A[\lambda^{-1}]$
    	at $\mathcal{O}_\mathfrak{p}$
    	(we view $g\in L(\mathcal{O}_\mathfrak{p})$
    	as an element in $L(A)$ via the natural map
    	$\mathcal{O}_\mathfrak{p}\rightarrow A$).

    	\item 
    	We write $A[\lambda^{-1}]$
    	for the free $A$-module of rank $1$
    	on which $T$ acts by the character $\lambda^{-1}$.
    	For any character
    	$\lambda\in
    	X^\ast(\widetilde{T}_{\widetilde{P}})$
    	(we view it also as a character of $T$),
    	we have a restriction map,
    	which is clearly
    	$\widetilde{T}_{\widetilde{P}}$-equivariant:
    	\[
    	\mathrm{Res}_A
    	\colon
    	R_A[\lambda^{-1}]
    	\rightarrow
    	A[\lambda^{-1}],
    	\quad
    	f\mapsto
    	f|_{\widetilde{P}_L}.
    	\]
    	Denote by
    	$R_A^0[\lambda^{-1}]$
    	the kernel of this map.
    	
    	\item 
    	Similarly we have another restriction map,
    	which is
    	$\widetilde{T}_{\widetilde{P}}(\mathcal{O}_\mathfrak{p})$-equivariant:
    	\[
    	\mathrm{Res}_A^\mathrm{top}
    	\colon
    	R^\mathrm{top}_A[\lambda^{-1}]
    	\rightarrow
    	A[\lambda^{-1}],
    	\quad
    	f\mapsto
    	f|_{\widetilde{P}_L(\mathcal{O}_\mathfrak{p})},
    	\]
    	whose kernel we denote by
    	$R^{\mathrm{top},0}_A[\lambda^{-1}]$.
    \end{enumerate}

    Next we define some operators on these
    modules, which
    correspond to the Hecke operators that we will
    consider later.
    We write
    \[
    \widetilde{T}_{\widetilde{P}}^+(E_\mathfrak{p})
    \subset
    \widetilde{T}_{\widetilde{P}}(E_\mathfrak{p})
    \]
    for the sub-monoid
    generated by the elements
    $\mu(p)\in \widetilde{T}_{\widetilde{P}}(E_\mathfrak{p})$
    for all
    $\mu\in
    X_{\ast,\mathrm{dm}}(\widetilde{T}_{\widetilde{P}})$.
    We let
    $\widetilde{T}_{\widetilde{P}}^+(E_\mathfrak{p})$
    act on
    $B_L(E_\mathfrak{p})$
    by inverse conjugation.
    It is easy to see that
    $\widetilde{T}_{\widetilde{P}}^+(E_\mathfrak{p})$
    stabilizes
    the subgroup
    $B_L(\mathcal{O}_\mathfrak{p})$.

    We write the fundamental coroots
    \[
    \{
    \epsilon_1',\epsilon_2',\cdots,\epsilon_r'
    \}
    \]
    for the monoid
    $
    X_{\ast,\mathrm{dm}}
    (T)$
    of dominant cocharacters.
    Then the element
    \[
    \epsilon_i:=
    \epsilon_i'(p)
    =\mathrm{diag}(\Lambda_i,p^{s_i}\Lambda_i^{-1})
    \in
    T_V(\mathbb{Q}_p)
    \]
    is a diagonal matrix in
    $\mathrm{GSp}(V,\psi)(\mathbb{Q}_p)$
    with
    $\Lambda_i=\mathrm{diag}
    (p^{t_1},\cdots,p^{t_n})$
    such that
    $0\leq t_1\leq t_2\leq\cdots\leq t_n$
    ($2n=\mathrm{dim}_\mathbb{Q}V$)
    and
    $s_i>0$.
    We denote the images of
    $\epsilon'_i$ in the projection
    $X_\ast(T)
    \rightarrow
    X_\ast(\widetilde{T}_{\widetilde{P}})$
    again by
    $\epsilon'_i$.

    \begin{definition}
    	We fix a finite flat $\mathcal{O}_\mathfrak{p}$-algebra
    	$A$.
    	For an element
    	$\epsilon\in\widetilde{T}_{\widetilde{P}}^+(E_\mathfrak{p})$
    	and the algebraic representation
    	$R_{A[1/p]}[\lambda^{-1}]$
    	for some dominant character
    	$\lambda\in X^\ast_\mathrm{dm}(\widetilde{T}_{\widetilde{P}})$,
    	we define an operator
    	$\mathbb{T}_\epsilon$
    	on
    	$R_{E_\mathfrak{p}}[\lambda^{-1}]$
    	as follows:
    	for any element
    	$f\in R_{A[1/p]}[\lambda^{-1}]$,
    	set
    	$\mathbb{T}_\epsilon(f)
    	\in
    	R_{A[1/p]}[\lambda^{-1}]$
    	to be
    	\[
    	(\mathbb{T}_\epsilon f)(g)
    	:=
    	f(\epsilon g\epsilon^{-1}),
    	\quad
    	\forall
    	g\in
    	L(A[1/p]).
    	\]

    	We define an operator
    	$\mathbb{T}_\epsilon$ on the spaces
    	$R^\mathrm{top}_{A[1/p]}[\lambda^{-1}]$
    	and
    	$A[\frac{1}{p}][\lambda^{-1}]$
    	by the same formula.
    \end{definition}
    
    It is clear that the maps
    $\mathrm{Res}_A$
    and
    $\mathrm{Res}_A^\mathrm{top}$
    are equivariant for the operator
    $\mathbb{T}_\epsilon$.
    Moreover, for two
    $\epsilon,\epsilon'\in
    \widetilde{T}_{\widetilde{P}}^+(E_\mathfrak{p})$,
    we have
    $\mathbb{T}_\epsilon\mathbb{T}_{\epsilon'}
    =
    \mathbb{T}_{\epsilon'}\mathbb{T}_\epsilon
    =
    \mathbb{T}_{\epsilon\epsilon'}$.
    We have the following observations
    (\textit{cf.} \cite[Proposition 3.1]{Pilloni2012}):
    \begin{proposition}\label{Hecke operator preserves top rep}
    	For a finite flat $\mathcal{O}_\mathfrak{p}$-algebra
    	$A$ with field of fractions,
    	a dominant character
    	$\lambda\in
    	X^{\ast}_\mathrm{dm}(\widetilde{T}_{\widetilde{P}})$ and an element
    	$\epsilon\in\widetilde{T}_{\widetilde{P}}^+(E_\mathfrak{p})$,
    	$\mathbb{T}_\epsilon$
    	preserves the lattice
    	$R^\mathrm{top}_A[\lambda^{-1}]$
    	inside
    	$R^\mathrm{top}_{A[1/p]}[\lambda^{-1}]$.
    \end{proposition}
    \begin{proof}
    	We prove the proposition for
    	$\epsilon=\epsilon_i$.
    	The general situation follows since each
    	$\epsilon$ is non-negative
    	integral combination of these 
    	$\epsilon_i$ and
    	$\mathbb{T}_{\epsilon\epsilon'}
    	=\mathbb{T}_\epsilon\mathbb{T}_{\epsilon'}$.
    	Recall
    	$\kappa_\mathfrak{p}$
    	is the residual field of
    	$\mathcal{O}_\mathfrak{p}$.
    	Let
    	$I_{B_L^\circ}$
    	be the set of elements in
    	$L(\mathcal{O}_\mathfrak{p})$
    	whose reduction modulo $\mathfrak{p}$
    	lies in the Borel subgroup
    	$B_L(\kappa_\mathfrak{p})^\circ$
    	opposite to $B_L(\kappa_\mathfrak{p})$
    	(an Iwahori subgroup of
    	$L(\mathcal{O}_\mathfrak{p})$).
    	By the Bruhat decomposition
    	$L(\kappa_\mathfrak{p})
    	=\bigsqcup_{w}
    	B_L^\circ(\kappa_\mathfrak{p})
    	w
    	B_L(\kappa_\mathfrak{p})$
    	where
    	$w$ runs through the double quotient set
    	$\mathcal{W}_P\backslash\mathcal{W}/\mathcal{W}_p$
    	(\cite[§II.1.9]{Jantzen2003}),
    	we have a decomposition
    	\[
    	L(\mathcal{O}_\mathfrak{p})
    	=
    	\bigsqcup_wI_{B_L^\circ}
    	wB_L(\mathcal{O}_\mathfrak{p})
    	=
    	\bigsqcup_{w}I_{B_L^\circ}w
    	U_L(\mathcal{O}_\mathfrak{p})
    	\]
    	where $w$ runs through the same set
    	$\mathcal{W}_P\backslash\mathcal{W}/\mathcal{W}_P$.
    	Now take a function
    	$f\in R^\mathrm{top}_{A}[\lambda^{-1}]$
    	and an element
    	$g=iwu\in I_{B_L^\circ}w
    	B_L(\mathcal{O}_\mathfrak{p})$,
    	by definition, we have
    	\[
    	(\mathbb{T}_\epsilon f)(g)
    	=
    	f(\epsilon iwu\epsilon^{-1})
    	=
    	f(\epsilon iw\epsilon^{-1})
    	=
    	f(\epsilon i\epsilon^{-1}w
    	\cdot
    	w^{-1}\epsilon w\epsilon^{-1})
    	=
    	\lambda^{-1}
    	(w^{-1}\epsilon w\epsilon^{-1})
    	f(\epsilon i\epsilon^{-1}w).
    	\]
    	Note that for
    	$\lambda\in X^{\ast}_\mathrm{dm}(\widetilde{T}_{\widetilde{P}})$,
    	we have
    	$
    	\lambda^{-1}(w^{-1}\epsilon w\epsilon^{-1})
    	\in
    	\mathcal{O}_\mathfrak{p}.
    	$
    	On the other hand,
    	by our restriction to
    	$\epsilon=\epsilon_i$ and
    	previous assumptions on these
    	$\epsilon_i$,
    	we see that
    	$\epsilon i\epsilon^{-1}\in
    	L(\mathcal{O}_\mathfrak{p})$
    	and thus
    	$\epsilon i\epsilon^{-1}
    	w\in L(\mathcal{O}_\mathfrak{p})$.
    	We deduce that
    	$(\mathbb{T}_\epsilon f)(g)\in
    	A$.
    \end{proof}

    On the other hand,
    $R_A[\lambda^{-1}]$
    may not be preserved by $\mathbb{T}_\epsilon$
    inside
    $R_{A[1/p]}[\lambda^{-1}]$
    (unless $A$ is an unramified extension of
    $\mathcal{O}_\mathfrak{p}$).
    This is because we do not have a Bruhat decomposition
    of $L(A)$ for general $A$.
    Due to this fact,   
    we need to modify the integral structure on
    $R_A[\lambda^{-1}]$
    which will be stable under the action of
    $\mathbb{T}_\epsilon$.
    The idea is simple: recall that we have a natural map
    $\mathrm{ev}_{\mathcal{O}_{\mathfrak{p}}}
    \colon
    R_A[\lambda^{-1}]
    \rightarrow
    R_A^\mathrm{top}[\lambda^{-1}]
    $
    as well as its base change to $A[1/p]$,
    that is,
    $\mathrm{ev}_{E_{\mathfrak{p}}}
    \colon
    R_{A[1/p]}[\lambda^{-1}]
    \rightarrow
    R_{A[1/p]}^\mathrm{top}[\lambda^{-1}]
    $.
    
    \begin{definition}
    	Then we set
    	\[
    	\widetilde{R}_A[\lambda^{-1}]
    	:=
    	\mathrm{ev}_{E_\mathfrak{p}}^{-1}
    	(R_A^\mathrm{top}[\lambda^{-1}])
    	\subset
    	R_{E_\mathfrak{p}}^\mathrm{top}[\lambda^{-1}]
    	\]
    	which is a lattice in
    	$R_{A[1/p]}[\lambda^{-1}]$
    	containing
    	$R_A[\lambda^{-1}]$.
    \end{definition}
      
    Clearly the map
    $\mathrm{ev}_{E_\mathfrak{p}}$
    is equivariant for the operator
    $\mathbb{T}_\epsilon$.
    Thus we see that
    \begin{corollary}
              The operator
              $\mathbb{T}_\epsilon$
              preserves
              $\widetilde{R}_A[\lambda^{-1}]$.
    \end{corollary}

    Next we consider ordinary projectors.
    \begin{definition}
    	Consider as above
    	a character
    	$\lambda\in
    	X^{\ast}_\mathrm{dm}
    	(\widetilde{T}_{\widetilde{P}})$ and
    	we put
    	$\mathbb{T}_{\widetilde{P}}
    	=
    	\prod_{i=1}^r\mathbb{T}_{\epsilon_i}$
    	and
    	\[
    	e_{\widetilde{P}}
    	=\lim\limits_{n\rightarrow\infty}
    	(\mathbb{T}_{\widetilde{L}})^{n!}.
    	\]
    \end{definition}
    
    It is easy to see that this limit is independent of the choice of the
    set of generators.
    Then one has
    \begin{proposition}\label{ordinary algebraic form}
    	      For
    	      $\lambda\in X^{\ast}_\mathrm{dm}
    	      (\widetilde{T}_{\widetilde{P}})$
    	      and
    	      a finite flat $\mathcal{O}_\mathfrak{p}$-algebra
    	      $A$,
              we have an isomorphism of
              $A[1/p]$-modules:
              \[
              e_{\widetilde{P}}
              R_{A[1/p]}[\lambda^{-1}]
              \simeq
              A[\frac{1}{p}][\lambda^{-1}].
              \]
              Similarly,
              we have an isomorphism of
              $A$-modules,
              which is compatible with the above isomorphism:
              \[
              e_{\widetilde{P}}R_A[\lambda^{-1}]
              \simeq
              A[\lambda^{-1}].
              \]
    \end{proposition}
    \begin{proof}
    	     Consider the first isomorphism.
             For the surjectivity,
             consider the function
             $f\in
             R_{A[1/p]}[\lambda^{-1}]$
             which is $1$ on $\widetilde{P}_L$
             and $0$ otherwise.
             Since $\widetilde{P}_L\supset B_L$,
             $f$ is well-defined
             and in fact
             $f\in
             R_{A[1/p]}[\lambda^{-1}]\backslash
             R_{A[1/p]}^0[\lambda^{-1}]$.
             Moreover, the action of
             $\mathbb{T}_\epsilon$ on $f$ is trivial,
             and
             $\mathrm{Res}_{A[1/p]}(f)\neq0$,
             thus the surjectivity follows.

             For the injectivity,
             it is enough to show that
             $e_{\widetilde{P}}
             R_{A[1/p]}^0[\lambda^{-1}]=0$.
             Note that
             the big cell
             $B_L^\circ B_L$
             is dense in
             $L$
             (\cite[§II.1.9]{Jantzen2003},
             over the field
             $A[1/p]$),
             so is
             $\widetilde{P}_L^\circ\widetilde{P}_L$.
             Thus an element $f\in
             R_{A[1/p]}[\lambda^{-1}]$
             is determined by its restriction to
             $\widetilde{P}_L^\circ \widetilde{P}_L$.
             Moreover
             the conjugate action of
             $\prod_{i=1}^r\epsilon_i$
             contracts $\widetilde{P}_L^\circ$
             into
             $\widetilde{P}_L^\circ\cap\widetilde{P}_L$.
             Yet by definition,
             any
             $f\in R_{A[1/p]}^0[\lambda^{-1}]$
             satisfies
             $f|_{\widetilde{P}_L}=0$
             and thus we see that
             $e_{\widetilde{P}}f=0$.

             Now consider the second isomorphism.
             The surjectivity is the same as above.
             For the injectivity,
             any $f\in R_A^0[\lambda^{-1}]$
             is also an element in
             $R_{A[1/p]}^0[\lambda^{-1}]$
             and thus
             $e_{\widetilde{P}}f=0$.
    \end{proof}
    
    Similarly we have
    \begin{proposition}
    	For $\lambda\in X^{\ast}_\mathrm{dm}(\widetilde{T}_{\widetilde{P}})$
    	and a finite flat
    	$\mathcal{O}_\mathfrak{p}$-algebra $A$,
    	we have a commutative diagram with
    	isomorphic horizontal arrows
    	\[
    	\begin{tikzcd}
    	e_{\widetilde{P}}R_A^\mathrm{top}[\lambda]
    	\arrow[r,"\simeq"]
    	\arrow[d,hook]
    	&
    	A[\lambda^{-1}]
    	\arrow[d,hook]
    	\\
    	e_{\widetilde{P}}R_{A[1/p]}^\mathrm{top}[\lambda^{-1}]
    	\arrow[r,"\simeq"]
    	&
    	A[\frac{1}{p}][\lambda^{-1}]
    	\end{tikzcd}
    	\]
    \end{proposition}
    \begin{proof}
    	The proof is exactly the same as in
    	the above proposition.
    \end{proof}

    Combining the above propositions,
    we see
    \begin{corollary}\label{ordinary modified=ordinary}
    	For any
    	$\lambda\in X^{\ast}_\mathrm{dm}(\widetilde{T}_{\widetilde{P}})$
    	and a finite flat
    	$\mathcal{O}_\mathfrak{p}$-algebra
    	$A$,
    	we have an isomorphism of
    	$A$-modules
    	\[
    	e_{\widetilde{P}}R_A[\lambda^{-1}]
    	\simeq
    	e_{\widetilde{P}}\widetilde{R}_A[\lambda^{-1}]
    	\simeq
    	A[\lambda^{-1}].
    	\]
    \end{corollary}

    \subsection{Classical automorphic forms}
    Now we will globalize
    the above discussion to the whole
    Shimura variety
    $Sh$
    (or its the $\mu$-ordinary locus).
    We define the $G$-torsor
    \[
    \mathfrak{G}^\Sigma
    :=
    \underline{\mathrm{Isom}}
    (
    (\mathcal{O}_{Sh^\Sigma}\otimes_\mathbb{Z}
    \Lambda,
    \mathfrak{t}),
    (e^\ast H^1_\mathrm{dR}
    (\mathcal{A}^\Sigma/Sh^\Sigma),
    \mathfrak{t}_\mathrm{dR}^\circ
    )
    ),
    \]
    where
    $e\colon
    Sh^\Sigma
    \rightarrow
    \mathcal{A}^\Sigma$
    is the unit section.

    Similarly,
    the cocharacter
    $\nu_x^{-1}$
    induces a filtration
    $\mathrm{Fil}^iV_{\mathcal{O}_\mathfrak{p}}$:
    \[
    0=
    \mathrm{Fil}^{-1}V_{\mathcal{O}_\mathfrak{p}}
    \subset
    \mathrm{Fil}^0V_{\mathcal{O}_\mathfrak{p}}
    \subset
    \mathrm{Fil}^1V_{\mathcal{O}_\mathfrak{p}}
    =
    V_{\mathcal{O}_\mathfrak{p}}.
    \]
    Then $(P_V)_{\mathcal{O}_\mathfrak{p}}$
    is exactly the stabilizer in
    $\mathrm{GSp}(V,\psi)_{\mathcal{O}_\mathfrak{p}}$
    of this filtration.
    We define the
    $P$-torsor
    as the sub-torsor
    of
    $\mathfrak{G}^\Sigma$
    \[
    \mathfrak{P}^\Sigma
    \subset
    \mathfrak{G}^\Sigma
    \]
    which preserves the filtrations on
    $\mathcal{O}_{Sh^\Sigma}
    \otimes_{\mathbb{Z}}
    V(\mathbb{Z})$
    (induced from
    $V$)
    and
    on
    $H^1_\mathrm{dR}
    (\mathcal{A}^\Sigma/
    Sh^\Sigma)$
    (the Hodge filtration).
    Then we define the $L$-torsor to be the quotient
    \begin{equation}\label{L-torsor}
    \mathfrak{L}^\Sigma
    :=
    \mathfrak{P}^\Sigma/U.
    \end{equation}

    For any character
    $\lambda\in X^\ast(\widetilde{T}_{\widetilde{P}})$,
    we write
    \[
    \mathcal{V}_\lambda^\Sigma
    :=
    \mathfrak{L}^\Sigma\times^L
    R_{\mathcal{O}_\mathfrak{p}}[\lambda^{-1}]
    \]
    for the contracted product,
    which is a quasi-coherent sheaf on
    $Sh^\Sigma$.
    We then write
    $\mathcal{V}_\lambda$
    for the restriction of
    $\mathcal{V}_\lambda^\Sigma$
    from
    $Sh^\Sigma$
    to
    $Sh$
    and
    $\mathcal{V}_\lambda^{\Sigma}(-C^\Sigma)$
    for the extension of
    $\mathcal{V}_\lambda$
    by zero
    back to
    the compactification
    $Sh^\Sigma$.

    \begin{definition}
    	For any $\mathcal{O}_\mathfrak{p}$-algebra
    	$A$,
    	we call
    	$H^0(Sh^\Sigma_A,
    	\mathcal{V}_\lambda^\Sigma)$
    	the space of modular forms
    	on
    	$Sh^\Sigma$
    	of weight $\lambda$,
    	of level $K$,
    	of coefficients in $A$,
    	and
    	call
    	$H^0(Sh^\Sigma_A,
    	\mathcal{V}_\lambda^\Sigma(-C^\Sigma))$
    	the space of cuspidal
    	modular forms on
    	$Sh^\Sigma$
    	of weight $\lambda$,
    	of level $K$,
    	of coefficients in $A$.
    \end{definition}

    \begin{remark}
    	By the definition of the Hasse invariant
    	$H$,
    	we know that
    	$H\in
    	H^0(Sh_1,\omega_{Sh}^{\otimes N_G})$
    	is a modular form of weight
    	$N_G\lambda_G$.
    \end{remark}

    \subsection{$p$-adic automorphic forms}\label{p-adic modular forms}

	In this subsection, we
	define the space of $p$-adic
	automorphic forms.
	For any integer $k\geq0$
	and any
	$\mathbb{W}$-scheme
	$Y$,
	we write
	$Y_k:=
	Y\times_{\mathbb{W}}\mathbb{W}_k$
	for the base change.
	We fix a point
	$x_0
	\in
	Sh_1^\mu$
	of characteristic $p$ in the (non-compactified)
	$\mu$-ordinary locus
	and denote by
	$\mathcal{A}_{x_0}[p^\infty]$
	the $p$-divisible group
	associated to
	$\mathcal{A}_{x_0}$.

	Recall the notion of
	$p$-divisible groups with $G$-structure
	as in
	\cite[Definition 3.1]{ShankarZhou2016}:
	for a $p$-divisible group $\mathcal{D}$
	over a ring $R$,
	we write
	$\mathbb{D}(\mathcal{D})$
	for the contravariant Dieudonn\'{e}
	crystal
	associated to $\mathcal{D}$.
	Then a $p$-divisible group with $G$-structure is the data
	of a $p$-divisible group
	$\mathcal{D}$ over
	$\overline{\mathbb{F}}_p$
	and
	tensors
	$\mathfrak{s}_{0}
	=(s_{\alpha,0})
	\subset
	\mathbb{D}(\mathcal{D})(\mathbb{W})^\otimes$
	such that
	there exists a finite free
	$\mathbb{Z}_p$-module
	$V$ with a $\mathbb{W}$-linear isomorphism
	$V\otimes_{\mathbb{Z}_p}\mathbb{W}
	\simeq
	\mathbb{D}(\mathcal{D})(\mathbb{W})$
	and the stabilizer in $\mathrm{GL}(V)$
	of the images of the tensors
	$\mathfrak{s}_0$
	in $V^\otimes$
	via the above isomorphism
	is exactly the group $G$.
	By \cite[Theorem 5.5]{ShankarZhou2016},
	for any such point
	$x_0$,
	there is a unique lifting
	$\widetilde{x}_0
	\in
	Sh(\mathbb{W})$
	such that
	the action of the subgroup
	$I_{x_0}\subset
	\mathrm{Aut}_\mathbb{Q}(\mathcal{A}_{x_0})$
	on $\mathcal{A}_{x_0}$
	lifts to
	$\mathcal{A}_{\widetilde{x}_0}$.
	Here
	$I_{x_0}$ consists of those automorphisms fixing the tensors
	$\mathfrak{t}_{\text{\'{e}t},\ell,x_0}$
	for all $\ell\neq p$
	(\textit{cf.}
	\cite[§5.4]{ShankarZhou2016}).
	We call
	$\mathcal{A}_{\widetilde{x}_0}$
	the canonical lift of
	$\mathcal{A}_{x_0}$.
	For any other point
	$y_0
	\in
	Sh_0^\mu$,
	we have an isomorphism of the canonical lifts
	$\mathcal{A}_{\widetilde{x}_0}[p^\infty]
	\simeq
	\mathcal{A}_{\widetilde{y}_0}[p^\infty]$.
	Indeed,
	we have an isomorphism of $p$-divisible groups
	$\mathcal{A}_{x_0}[p^\infty]
	\simeq
	\mathcal{A}_{y_0}[p^\infty]$
	(\cite[Proposition 5.4]{ShankarZhou2016}).
	Moreover,
	by Grothendieck-Messing theory,
	lifts of
	$\mathcal{A}_{x_0}[p^\infty]$
	to
	$\mathbb{W}$
	correspond to lifts of the natural filtration
	$\overline{\mathrm{Fil}}$
	to $\mathbb{W}$
	on
	$\mathbb{D}(\mathcal{A}_{x_0}[p^\infty])
	\otimes_{\mathbb{Z}_p}\mathbb{F}_p$
	induced by the mod $p$
	cocharacter $\mu(\mathrm{mod\ }p)$
	to
	$\mathbb{D}(\mathcal{A}_{x_0}[p^\infty])$.
	Here $\mu\in X_\ast(T)$
	is the dominant representative of the
	conjugacy class of
	$\nu_x^{-1}$ in $X_\ast(T)$
	(independent of $x$).
	In particular,
	the canonical lift
	$\mathcal{A}_{\widetilde{x}_0}[p^\infty]$
	corresponds to
	the filtration on
	$\mathbb{D}(\mathcal{A}_{x_0}[p^\infty])$
	given by the cocharacter
	$\mu$.
	By comparing the filtrations on
	both sides
	$\mathbb{D}(\mathcal{A}_{x_0}[p^\infty])
	\simeq
	\mathbb{D}(\mathcal{A}_{y_0}[p^\infty])$,
	we see that
	$\mathcal{A}_{\widetilde{x}_0}[p^\infty]$
	is isomorphic to
	$\mathcal{A}_{\widetilde{y}_0}[p^\infty]$.
	We define $p$-divisible groups
	\[
	\widetilde{\mathcal{BT}}
	:=\mathcal{A}_{\widetilde{x}_0}[p^\infty],
	\quad
	\mathcal{BT}
	:=
	\mathcal{A}_{x_0}[p^\infty].
	\]

	We have a slope decomposition
	$\mathcal{BT}
	=
	\prod_{i=1}^r
	\mathcal{BT}_i$
	with each
	$\mathcal{BT}_i$
	isoclinic of slope
	$\lambda_i$
	and
	$1\geq\lambda_1>\lambda_2>\cdots>\lambda_r\geq0$.
	This slope filtration lifts to
	$\widetilde{\mathcal{BT}}$:
	\[
	0
	=
	\widetilde{\mathcal{BT}}_0
	\subset
	\widetilde{\mathcal{BT}}_1
	\subset
	\widetilde{\mathcal{BT}}_2
	\subset
	\cdots
	\subset
	\widetilde{\mathcal{BT}}_r
	=
	\widetilde{\mathcal{BT}}
	\]
	such that each graded piece
	$\mathrm{gr}_i(\widetilde{\mathcal{BT}})
	=
	\widetilde{\mathcal{BT}}_i/\widetilde{\mathcal{BT}}_{i-1}$
	lifts
	$\mathcal{BT}_i$
	from
	$\overline{\mathbb{F}}_p$
	to
	$\mathbb{W}$.

	We consider then the construction of Igusa towers.
	Set $r_0=\lfloor\frac{1+r}{2}\rfloor$.
	For any
	$x\in
	Sh_1^\mu$
	we consider
	an isomorphism
	\[
	\varphi\colon
	\prod_{i=1}^{r_0}\mathcal{A}_{x}[p^\infty]_i
	\simeq
	\prod_{i=1}^{r_0}\mathcal{A}_{x_0}[p^\infty]_i,
	\]
	which,
	if $r$ is an odd integer,
	respects the polarization at the components
	$\mathcal{A}_{x}[p^\infty]_{r_0}$
	and
	$\mathcal{A}_{x_0}[p^\infty]_{r_0}$.
	Then one can extend $\varphi$ in a unique way
	to an isomorphism
	\[
	\varphi\colon
	\prod_{i=1}^r\mathcal{A}_{x}[p^\infty]_i
	\simeq
	\prod_{i=1}^r\mathcal{A}_{x_0}[p^\infty]_i
	\]
	respecting the polarizations on both sides.
	Now for a cuspidal label representative $W$
	of rank $t$
	of a component of the
	boundary of
	$Sh^\Sigma$
	of rank $t$
	which intersects non-trivially with the closure of
	a lifting of $x$ from
	$Sh_1^\mu$
	to
	$Sh^\mu$,
    by \cite[§3.2.1]{Madapusi2012},
	we have an extension of the
	universal abelian scheme
	$\mathcal{A}_W$
	to a semi-abelian scheme
	$\mathcal{G}$
	\[
	0
	\rightarrow
	\mathbb{G}_m\otimes W
	\rightarrow
	\mathcal{G}
	\rightarrow
	\mathcal{A}_W
	\rightarrow
	0
	\]
	as well as their $p$-divisible groups
	\[
	0
	\rightarrow
	\mu_{p^\infty}^t
	\rightarrow
	\mathcal{G}[p^\infty]
	\rightarrow
	\mathcal{A}_W[p^\infty]
	\rightarrow
	0.
	\]
	By construction of the toroidal compactification for
	the integral Siegel Shimura variety
	$Sh(\mathrm{GSp}(V,\psi),S^\pm,K_V)$, we see that
	the difference between $\mathcal{A}_W[p^\infty]$
	and
	$\mathcal{A}_{x_0}[p^\infty]$
	is $\mu_{p^\infty}^t$ and
	$(\mathbb{Q}_p/\mathbb{Z}_p)^t$.
	Thus we see that
	for any $y\in W$,
	any isomorphism $\varphi$ as above
	extends in a unique way to an isomorphism
	\begin{equation}\label{extension to the boundary}
	\varphi_W
	\colon
	(\mathcal{G}_y[p^\infty])^\circ
	\simeq
	\prod_{i=1}^{r-1}\mathcal{A}_{x_0}[p^\infty]_i,
	\end{equation}
	where
	$(\mathcal{G}_y[p^\infty])^\circ$
	denotes the connected component of
	$\mathcal{G}_y[p^\infty]$.
	Translating this into the language of Dieudonn`{e} crystals,
	we have
	\begin{proposition}\label{extension to the boundary in terms of crystals}
		Fix a point
		$x\in
		Sh_1^\mu$
		and consider an isomorphism
		$\phi
		\colon
		\prod_{i=1}^{r_0}
		\mathcal{A}_x[p^\infty]_i
		\simeq
		\prod_{i=1}^{r_0}
		\mathcal{A}_{x_0}[p^\infty]_i$.
		We assume moreover that for
		$r$ odd,
		$\phi$
		respects the polarization at the components
		$\mathcal{A}_x[p^\infty]_{r_0}$
		and
		$\mathcal{A}_{x_0}[p^\infty]_{r_0}$.
		Then for any $W$ as above and any
		point
		$y\in
		W$,
		$\phi$
		extends in a unique way to an isomorphism of
		Dieudonn\'{e} crystals
		\[
		\mathbb{D}(\phi_W)(\mathbb{W})
		\colon
		\mathbb{D}
		((\mathcal{G}_y[p^\infty])^\circ)
		\simeq
		\mathbb{D}
		(\prod_{i=1}^{r-1}\mathcal{A}_{x_0}[p^\infty]_i).
		\]
		
	\end{proposition}
	
	Next we discuss the automorphism group of the $p$-divisible group
	$\mathcal{A}_{x_0}[p^\infty]$.
	We define an algebraic subgroup
	$L_\mu
	\subset
	G_{\mathbb{Q}_p}$
	over $\mathbb{Q}_p$
	as follows:
	for any
	$\mathbb{Q}_p$-algebra $R$,
	\[
	L_\mu(R):=
	\{
	g\in
	G(R\otimes_{\mathbb{Q}_p}\mathbb{W})|
	g^{-1}\sigma(\mu(p))\sigma(g)=\sigma(\mu(p))
	\}.
	\]	
	It is known that
	$L_\mu$
	has the same rank as
	$G_{\mathbb{Q}_p}$
	(\textit{cf.}
	\cite[§§3.3 and 5.4]{ShankarZhou2016}
	and
	\cite[Corollary 2.1.7]{Kisin2016}).
	By \cite[§3.3]{ShankarZhou2016},
	we know also that
	$L_\mu(\mathbb{Q}_p)$
	is the set of automorphisms of
	the $p$-divisible group
	$\mathcal{A}_{x_0}[p^\infty]$
	in the isogeny category,
	fixing the tensors
	$\mathfrak{t}$.
	Next we discuss the integral structure of
	$L_\mu$ over $\mathbb{Z}_p$.
	We know that $L_\mu$
	is an inner form
	of the centralizer
	$L_{\mathbb{Q}_p}
	=
	\mathrm{Cent}_G(\overline{\mu})$ in
	$G_{\mathbb{Q}_p}$ of the cocharacter
	$\overline{\mu}$
	(by definition of $L_\mu$ or see
	\cite[§4.3]{Kottwitz1997}).
	Since we have fixed a representative of
	$\mu$ which is defined over $\mathbb{Z}_p$,
	the $\mathbb{Z}_p$-model
	$L$ of $L_{\mathbb{Q}_p}$
	determines a $\mathbb{Z}_p$-model of
	$L_\mu$
	via the inner twist between
	$L_{\mathbb{Q}_p}$
	and
	$L_\mu$.
	We will fix such a $\mathbb{Z}_p$-model of
	$L_\mu$
	and denote this model by the same letter
	$L_\mu$,
	when no confusion is possible.
	Then one sees immediately
	(using the (resp. truncated) Dieudonn\'{e} crystals associated to
	$\mathcal{A}_{x_0}[p^\infty]$)
	that
	$L_\mu(\mathbb{Z}_p)$,
	resp.,
	$L_\mu(\mathbb{Z}_p/p^m)$,
	is the set of automorphisms of
	$\mathcal{A}_{x_0}[p^\infty]$,
	resp.,
	$\mathcal{A}_{x_0}[p^m]$
	(induced from an automorphism of
	$\mathcal{A}_{x_0}[p^\infty]$)
	fixing the tensors
	$\mathfrak{t}$.
	Lastly it is known that
	$L_\mu(\mathbb{Q}_p)$
	respects the slope filtration
	on
	$\mathbb{D}(\mathcal{A}_{x_0}[p^\infty])$
	induced by
	$\mu$
	(\textit{cf.}
	the proof of
	\cite[Theorem 3.5]{ShankarZhou2016}).

	From the above discussion we get the following:
	\begin{definition-proposition}\label{Igusa tower}
		We write
		\[
		\mathrm{Ig}
		:=
		\underline{\mathrm{Isom}}_{Sh^{\Sigma,\mu},
			\mathrm{pol}}
		((\mathcal{A}[p^\infty],\mathfrak{t}),
		(\widetilde{\mathcal{BT}},\mathfrak{t}_{\mathbb{Z}_p}))
		\]
		for the
		$L_\mu(\mathbb{Z}_p)$-torsor over
		the $\mu$-ordinary locus
		$Sh^{\Sigma,\mu}$,
		consisting of
		isomorphisms 
		\[
		\varphi\colon
		\mathcal{A}_y[p^\infty]
		\simeq
		\widetilde{\mathcal{BT}}
		\]
		as above		
		such that the induced isomorphisms on the
		$\mathbb{W}$-points of the Dieudonn\'{e} crystals
		\[
		\mathbb{D}(\varphi)(\mathbb{W})
		\colon
		\mathbb{D}(\mathcal{A}_y[p^\infty])(\mathbb{W})
		\rightarrow
		\mathbb{D}(\widetilde{\mathcal{BT}})(\mathbb{W})
		\]
		respect the Hodge tensors $\mathfrak{t}$
		and
		$\mathfrak{t}_{x_0}$
		on both sides
		(the associated map on the corresponding
		Dieudonn\'{e} crystals,
		by
		Proposition
		\ref{extension to the boundary in terms of crystals},
		extends uniquely to an isomorphism on the boundary).
		Similarly for any
		$m\geq0$,
		we have an $L_\mu(\mathbb{Z}_p)$-torsor
		\[
		\mathrm{Ig}_m
		\subset
		\underline{\mathrm{Isom}}_{Sh^{\Sigma,\mu}_m,
			\mathrm{pol}}
		((\mathcal{A}[p^\infty],\mathfrak{t}),
		(\widetilde{\mathcal{BT}},\mathfrak{t}_{\mathbb{Z}_p}))
		\]
		over
		$Sh^{\Sigma,\mu}_m$
		which is the pull-back of
		$\mathrm{Ig}$ along
		the closed embedding
		$Sh^{\Sigma,\mu}_m\hookrightarrow
		Sh^{\Sigma,\mu}$.
		Moreover
		for any integer
		$n\geq1$,
		we consider the
		$L_\mu(\mathbb{Z}_p/p^n)$-torsor
		\[
		\mathrm{Ig}_{m,n}
		\subset
		\mathrm{Isom}_{Sh^{\Sigma,\mu}_m,\mathrm{pol}}
		(\mathcal{A}[p^n],\widetilde{\mathcal{BT}}[p^n])
		\]
		consisting of
		isomorphism
		$\varphi_n
		\colon
		\prod_{i=1}^{r_0}
		\mathrm{gr}_i(\mathcal{A}_y[p^n])
		\simeq
		\prod_{i=1}^{r_0}
		\mathrm{gr}_i(\widetilde{\mathcal{BT}})$
		such that
		there is an isomorphism
		$\varphi
		\colon
		\prod_{i=1}^{r_0}
		\mathrm{gr}_i(\mathcal{A}_y[p^\infty])
		\simeq
		\prod_{i=1}^{r_0}
		\mathrm{gr}_i(\widetilde{\mathcal{BT}})$
		with reduction
		$\varphi\equiv\varphi_n(\mathrm{mod} p^n)$
		and
		$\mathbb{D}(\varphi)(\mathbb{W})$
		respects the Hodge tensors
		$\mathfrak{t}$
		and
		$\mathfrak{t}_{x_0}$
		on both sides.
		
		We write
		$\mathrm{Ig}_\infty$
		to be the formal completion of
		$\mathrm{Ig}$
		along the special fibre
		$\mathrm{Ig}_0$.
		Similarly we write
		$Sh^{\Sigma,\mu}_\infty$ to be
		the formal completion of
		$Sh^{\Sigma,\mu}$
		along
		$Sh^{\Sigma,\mu}_0$.
	\end{definition-proposition}

    Let $\widetilde{P}$
    be as above.
    
    \begin{definition}
    	\begin{enumerate}
    		\item 
    		We write
    		\begin{align*}
    		\mathbb{V}
    		&
    		:=
    		H^0(\mathrm{Ig},\mathcal{O}_{\mathrm{Ig}}),
    		&
    		\mathbb{V}_{m,n}
    		&
    		:=
    		H^0(\mathrm{Ig}_{m,n},
    		\mathcal{O}_{\mathrm{Ig}_{m,n}}),
    		\\
    		\mathbb{V}_{m}
    		&
    		:=
    		H^0(\mathrm{Ig}_m,
    		\mathcal{O}_{\mathrm{Ig}_m}),
    		&
    		\mathbb{V}_\infty
    		&
    		:=
    		\lim\limits_{\overleftarrow{m}}
    		\mathbb{V}_m.
    		\end{align*}
    		
    		\item 
    		For any $\widetilde{P}$ as above,
    		we write
    		$\mathbb{V}_\infty^{\widetilde{P}_L^\mathrm{der}}$
    		for the subspace of
    		$\mathbb{V}_\infty$
    		consisting of regular functions 
    		which are invariant under the action of
    		$\widetilde{P}_L^\mathrm{der}$.
    		This is the space of $p$-adic automorphic forms
    		for the parabolic subgroup $\widetilde{P}\subset G$
    		and of level $K$.
    		
    		\item 
    		For any character
    		$\kappa\in X^\ast(\widetilde{T}_{\widetilde{P}})$,
    		we write
    		$\mathbb{V}_\infty^{\widetilde{P}_L^\mathrm{der}}[\kappa]$
    		to be the subspace of
    		$\mathbb{V}_\infty^{\widetilde{P}_L^\mathrm{der}}$
    		on which $\widetilde{T}_{\widetilde{P}}$
    		acts by the character
    		$\kappa$.
    		This is the space of $p$-adic automorphic forms
    		for $\widetilde{P}$ of level $K$ and
    		of character $\kappa$.
    		
    		\item 
    		We write
    		$\mathbb{V}_{\mathrm{cusp},\infty}^{\widetilde{P}_L^\mathrm{der}}$
    		for the subspace of
    		$\mathbb{V}_\infty^{\widetilde{P}_L^\mathrm{der}}$
    		consisting of regular functions
    		which vanish at
    		the cusps
    		$C^\Sigma$.
    		This is the space of cuspidal $p$-adic automorphic
    		forms for $\widetilde{P}$ of level $K$.
    		Similarly
    		for any
    		$\kappa\in
    		X^\ast(\widetilde{T}_{\widetilde{P}})$
    		we have
    		the space of cuspidal
    		$p$-adic automorphic forms
    		$\mathbb{V}^{\widetilde{P}_L^\mathrm{der}}_{\mathrm{cusp},\infty}
    		[\kappa]$
    		for $\widetilde{P}$ of level $K$
    		and of character $\kappa$.
    	\end{enumerate}
    \end{definition}

	\subsection{Hodge-Tate map}
	In this subsection,
	we define the Hodge-Tate map,
	which relates the Igusa tower
	$(\mathrm{Ig}_m)_{m\geq1}$
	to the torsors
	$\mathfrak{L}^\Sigma$.
	For the torsors
	$\mathfrak{P}^\Sigma_?,
	\mathfrak{L}^\Sigma_?$,
	we write
	$\mathfrak{P}^{\Sigma,\mu}_?,
	\mathfrak{L}^{\Sigma,\mu}_?$
	for the restriction of these torsors
	from
	$Sh^\Sigma$
	to the $\mu$-ordinary locus
	$Sh^{\Sigma,\mu}_?$
	($?=0,1,\cdots,\emptyset$).
	Now for any isomorphism
	$\varphi
	\colon
	\mathcal{A}[p^\infty]^\circ
	\rightarrow
	\mathcal{A}_{x_0}[p^\infty]^\circ$
	over
	$Sh^{\Sigma,\mu}_m$ in
	$\mathrm{Ig}_m$,
	we have an induced isomorphism
	\[
	\mathrm{HT}_m(\varphi)
	\colon
	e^\ast
	H^1_\mathrm{dR}(\mathcal{A}/Sh^{\Sigma,\mu}_m)
	\rightarrow
	e^\ast
	H^1_\mathrm{dR}(\mathcal{A}_{x_0}/Sh^{\Sigma,\mu}_m)
	\]
	respecting the Hodge tensors
	$\mathfrak{t},\mathfrak{t}_\mathrm{dR}$
	and the Hodge filtrations
	on both sides.
	Then the above process
	gives us the Hodge-Tate map
	
	\begin{definition}
		We define a map
		\[
		\mathrm{HT}_m
		\colon
		\mathrm{Ig}_m
		\rightarrow
		\mathfrak{L}^{\Sigma,\mu}_m,
		\,
		\varphi
		\mapsto
		\mathrm{HT}_m(\varphi).
		\]
		For a parabolic subgroup
		$\widetilde{P}\subset P$,
		we write
		$\mathrm{HT}_m^{\widetilde{P}}$
		to be the composition
		\[
		\mathrm{HT}_m^{\widetilde{P}}
		\colon
		\mathrm{Ig}_m
		\xrightarrow{\mathrm{HT}_m}
		\mathcal{L}_m^{\Sigma,\mu}
		\rightarrow
		\mathcal{L}_m^{\Sigma,\mu}/\widetilde{P}_L^\mathrm{der}.
		\]
		Moreover,
		we write
		$\mathrm{HT}_\infty
		\colon
		\mathrm{Ig}_\infty
		\rightarrow
		\mathcal{L}_\infty^{\Sigma,\mu}$
		for the inverse limit of the maps
		$(\mathrm{HT}_m)_m$
		and similarly for
		$\mathrm{HT}_\infty^{\widetilde{P}}$.
	\end{definition}
    Note that
    \textit{a priori},
    $\mathrm{Ig}_m$
    is an
    $L_\mu(\mathbb{Z}_p/p^m)$-torsor and
    $\mathfrak{L}_m^{\Sigma,\mu}$
    is an
    $L(\mathbb{Z}_p/p^m)$-torsor.
    However by the relation between
    $L_\mu$ and $L$    
    (\textit{cf.}
    preceding
    Definition-Proposition
    \ref{Igusa tower}
    ),
    we see that
    the Hodge-Tate map
    $\mathrm{HT}_m$
    is equivariant for the actions of
    $L_\mu(\mathbb{Z}_p/p^m)$
    and
    $L(\mathbb{Z}_p/p^m)$
    on both sides.

    Using this map, we can associate to a classical
    modular form a $p$-adic modular form as follows:
    for any
    $f\in
    H^0(Sh^{\Sigma,\mu}_A,
    \mathcal{V}_\lambda^{\Sigma,\mu})$
    with $A$ an
    $\mathcal{O}_\mathfrak{p}/\mathfrak{p^m}$-algebra,
    by definition $f$ is a global section of
    $\mathcal{V}_\lambda^{\Sigma,\mu}$
    and thus $B_L$ acts on $f$ by the character
    $\lambda$.
    As a result $\mathrm{HT}_m^\ast
    (f)$
    is a global section of
    the structural sheaf
    $\mathcal{O}_{\mathrm{Ig}_m}$,
    on which $B_L$
    acts by the character same $\lambda$.
    In summary we have the following morphism
    \[
    \mathrm{HT}_m^\ast
    \colon
    H^0(Sh^\Sigma_A,\mathcal{V}_\lambda^{\Sigma,\mu})
    \rightarrow
    \mathbb{V}_m^{\widetilde{P}_L^\mathrm{der}}[\lambda^{-1}].
    \]
    Moreover, if $f$ is cuspidal,
    so is $\mathrm{HT}^\ast_m(f)$.

    Consider the quotients
    $\mathfrak{L}^{\Sigma,\mu}_m
    /\widetilde{P}_L^\mathrm{der}$
    and
    $\mathfrak{L}^{\Sigma,\mu}_m
    /\widetilde{P}_L$
    of the
    $\widetilde{P}_L$-torsor
    $\mathcal{L}^{\Sigma,\mu}$
    over
    $Sh^{\Sigma,\mu}_m$
    and the quotient
    $\mathrm{Ig}_m/\widetilde{P}_L(\mathcal{O}_\mathfrak{p})$
    of
    $\mathrm{Ig}_m$ by
    $\widetilde{P}_L(\mathcal{O}_\mathfrak{p})$.
    Write the fibre product
    \[
    \widetilde{\mathrm{Ig}}_m^{\widetilde{P}_L}
    :=
    (\mathrm{Ig}_m/\widetilde{P}_L(\mathcal{O}_\mathfrak{p}))
    \times_{\mathfrak{L}_m^{\Sigma,\mu}/\widetilde{P}_L}
    \mathfrak{L}_m^{\Sigma,\mu}/\widetilde{P}_L^\mathrm{der}
    \]
    and consider the projection to the first factor
    $\mathrm{pr}
    \colon
    \widetilde{\mathrm{Ig}}_m^{\widetilde{P}_L}
    \rightarrow
    \mathrm{Ig}_m/\widetilde{P}_L
    (\mathcal{O}_\mathfrak{p})$.
    Note that
    $\widetilde{T}_{\widetilde{P}}$
    acts on this torsor
    by its natural action on
    $\mathfrak{L}_m^{\Sigma,\mu}/\widetilde{P}_L^\mathrm{der}$.
    We have the following result which
    relates the Igusa tower and the
    $L$-torsor
    (\textit{cf.} \cite[Proposition 4.1]{Pilloni2012}):
    \begin{proposition}\label{fundamental isom}
    	For any
    	$\lambda
    	\in
    	X^\ast(\widetilde{T}_{\widetilde{P}})$,
    	the natural morphism
    	\[
    	\widetilde{\mathrm{HT}}_m^{\widetilde{P}_L}
    	\colon
    	\mathrm{Ig}_m
    	/\widetilde{P}_L^\mathrm{der}(\mathcal{O}_\mathfrak{p})
    	\rightarrow
    	\widetilde{\mathrm{Ig}}_m^{\widetilde{P}_L}
    	\]
    	induces the following isomorphism:
    	\[
    	H^0(\mathrm{Ig}_m/
    	\widetilde{P}_L(\mathcal{O}_\mathfrak{p}),
    	\mathrm{pr}_\ast
    	(\mathcal{O}_{\widetilde{\mathrm{Ig}}_m^{\widetilde{P}_L}}
    	[\lambda^{-1}]))
    	\simeq
    	\mathbb{V}_m^{\widetilde{P}_L^\mathrm{der}}[\lambda^{-1}].
    	\]
    	Here
    	$\mathcal{O}_{\widetilde{\mathrm{Ig}}_m^{\widetilde{P}_L}}
    	[\lambda^{-1}]$
    	denotes the subsheaf
    	of
    	$\mathcal{O}_{\widetilde{\mathrm{Ig}}_m^{\widetilde{P}_L}}$
    	on which $\widetilde{T}_{\widetilde{P}}$
    	acts by the character $\lambda$.
    \end{proposition}
    \begin{proof}
    	By definition,
    	we have
    	\[
    	\mathbb{V}_m^{\widetilde{P}_L^\mathrm{der}}[\lambda^{-1}]
    	=H^0
    	(\mathrm{Ig}_m
    	/\widetilde{P}_L^\mathrm{der}(\mathcal{O}_\mathfrak{p}),
    	\mathcal{O}_{\mathrm{Ig}_m/\widetilde{P}_L^\mathrm{der}
    		(\mathcal{O}_\mathfrak{p})}[\lambda^{-1}]).
    	\]
    	Note that
    	$\mathrm{Ig}_m/
    	\widetilde{P}_L^\mathrm{der}(\mathcal{O}_\mathfrak{p})$
    	is a pro-finite \'{e}tale covering of
    	$\mathrm{Ig}_m/
    	\widetilde{P}_L(\mathcal{O}_\mathfrak{p})$
    	of group
    	$\widetilde{T}_{\widetilde{P}}(\mathcal{O}_\mathfrak{p})$
    	while
    	$\widetilde{\mathrm{Ig}}_m^{\widetilde{P}_L}$
    	is a
    	$\widetilde{T}_{\widetilde{P}}$-torsor
    	over
    	$\mathrm{Ig}_m/
    	\widetilde{P}_L(\mathcal{O}_\mathfrak{p})$.
    	Thus over the base torsor
    	$\mathrm{Ig}_m/
    	\widetilde{P}_L(\mathcal{O}_\mathfrak{p})$,
    	$\widetilde{\mathrm{HT}}_m^{\widetilde{P}_L}$
    	is induced by the natural morphism
    	$\underline{\widetilde{T}_{\widetilde{P}}(\mathcal{O}_\mathfrak{p})}
    	\rightarrow
    	\widetilde{T}_{\widetilde{P}}$
    	of group schemes over
    	$\mathcal{O}_\mathfrak{p}$
    	(the first being a constant group scheme).
    	Note that $\widetilde{T}_{\widetilde{P}}$
    	is quasi-split over $\mathcal{O}_\mathfrak{p}$,
    	we see that
    	\'{e}tale locally over
    	$\mathrm{Ig}_m/
    	\widetilde{P}_L(\mathcal{O}_\mathfrak{p})$,
    	$\widetilde{\mathrm{HT}}_m^{\widetilde{P}_L}$
    	sends
    	$H^0(\widetilde{T}_{\widetilde{P}},\mathcal{O}_{\widetilde{T}_{\widetilde{P}}}
    	[\lambda^{-1}])$
    	bijectively to
    	$H^0(\underline{\widetilde{T}_{\widetilde{P}}(\mathcal{O}_\mathfrak{p})},
    	\mathcal{O}_{
    		\underline{\widetilde{T}_{\widetilde{P}}(\mathcal{O}_\mathfrak{p})}}
    	[\lambda^{-1}])$
    	(both are in fact
    	of rank $1$ over $\mathcal{O}_\mathfrak{p}$).
    \end{proof}

    \section{Hecke operators}\label{Hecke operators}
    \subsection{Parahoric Hecke algebra}
    We introduce first the abstract Hecke algebras
    associated to the parabolic subgroups
    $\widetilde{P}$ that we are going to use.
    We write
    $\mathbb{Z}[\widetilde{T}_{\widetilde{P}}^+(E_\mathfrak{p})]$
    for the commutative algebra over
    $\mathbb{Z}$ generated by the elements in
    $\widetilde{T}_{\widetilde{P}}^+(E_\mathfrak{p})$.
    For any integer $n\geq1$,
    denote by
    $I_{\widetilde{P}}(n)$
    for the subset of
    $P(\mathcal{O}_\mathfrak{p})$
    consisting of elements
    $g$ such that
    $g(\mathrm{mod}\mathfrak{p}^n)\in
    \widetilde{P}(\mathcal{O}_\mathfrak{p}/\mathfrak{p}^n)$.
    Similarly,
    denote by
    $I_{S\widetilde{P}}(n)$
    the subset of
    $P(\mathcal{O}_\mathfrak{p})$
    consisting of $g$ such that
    $g(\mathrm{mod}\,\mathfrak{p}^n)
    \in
    S\widetilde{P}(\mathcal{O}_\mathfrak{p}/\mathfrak{p}^n)$.
    Write
    \[
    \mathcal{C}(G(E_\mathfrak{p})//I_{\widetilde{P}}(n),\mathbb{Z}),
    \quad
    \text{resp.,}
    \quad
    \mathcal{C}(G(E_\mathfrak{p})//I_{S\widetilde{P}}(n),\mathbb{Z})
    \]
    for the set of compact support functions
    from
    the double coset
    $G(E_\mathfrak{p})//I_{\widetilde{P}}(n)$,
    resp.,
    $G(E_\mathfrak{p})//I_{S\widetilde{P}}(n)$,
    to $\mathbb{Z}$.
    These two sets have natural $\mathbb{Z}$-algebra structures
    whose product is given by the involution.
    We then have the canonical projections
    \[
    \pi_{\widetilde{P}}
    \colon
    G(E_\mathfrak{p})
    \rightarrow
    G(E_\mathfrak{p})//I_{\widetilde{P}}(n),
    \]
    \[
    \pi_{S\widetilde{P}}
    \colon
    G(E_\mathfrak{p})
    \rightarrow
    G(E_\mathfrak{p})//I_{S\widetilde{P}}(n).
    \]
    
    \begin{proposition}
    	The maps of $\mathbb{Z}$-algebras
    	given by
    	\[
    	\mathbb{Z}[\widetilde{T}_{\widetilde{P}}^+(E_\mathfrak{p})]
    	\rightarrow
    	\mathcal{C}
    	(G(E_\mathfrak{p})//I_{\widetilde{P}}(n),\mathbb{Z}),
    	\quad
    	\epsilon\in\widetilde{T}_{\widetilde{P}}^+(E_\mathfrak{p})
    	\mapsto
    	1_{\pi_{\widetilde{P}}(\epsilon)};
    	\]
    	\[
    	\mathbb{Z}[\widetilde{T}_{\widetilde{P}}^+(E_\mathfrak{p})]
    	\rightarrow
    	\mathcal{C}
    	(G(E_\mathfrak{p})//I_{S\widetilde{P}}(n),\mathbb{Z}),
    	\quad
    	\epsilon\in\widetilde{T}_{\widetilde{P}}^+(E_\mathfrak{p})
    	\mapsto
    	1_{\pi_{S\widetilde{P}}(\epsilon)}
    	\]
    	are well-defined and are both injective.
    	Here
    	$1_X$ denotes the characteristic function of a subset
    	$X$.
    \end{proposition}
    \begin{proof}
    	We treat the first case,
    	the second being similar.
    	We already know that the proposition is true
    	for the case
    	$G=\mathrm{GSp}(V,\psi)$
    	(\cite[§3.6]{Hida2002}).
    	The proof for the general case is quite similar
    	to this case
    	as in
    	\cite[§2]{Hida1995}.
    	More precisely,
    	write temporarily
    	$H=I_{\widetilde{P}}(n)$,
    	then for each
    	$\epsilon\in\widetilde{T}_{\widetilde{P}}^+(E_\mathfrak{p})$,
    	we can choose a set of representatives
    	$X_\epsilon\subset H$
    	such that we have a decomposition
    	\[
    	H\epsilon H=
    	\bigsqcup_{\alpha\in X_\epsilon}H\epsilon\alpha.
    	\]
    	We can choose $X_\epsilon$ as follows:
    	let $\widetilde{T}_{\widetilde{P}}^+(E_\mathfrak{p})$
    	act on
    	the parabolic subgroup
    	$\widetilde{P}(\mathcal{O}_\mathfrak{p})$
    	of
    	$G(\mathcal{O}_\mathfrak{p})$
    	by inverse conjugation,
    	thus we see that
    	$\epsilon^{-1}
    	\widetilde{P}(\mathcal{O}_\mathfrak{p})
    	\epsilon
    	\subset
    	\widetilde{P}(\mathcal{O}_\mathfrak{p})$.
    	Moreover,
    	the following natural inclusion is
    	in fact a bijection:
    	\[
    	H/(\epsilon^{-1}H\epsilon\cap H)
    	\hookrightarrow
    	\widetilde{P}(\mathcal{O}_\mathfrak{p})/
    	\epsilon^{-1}\widetilde{P}(\mathcal{O}_\mathfrak{p})\epsilon.
    	\]
    	Then we just take a set
    	$X_\epsilon$ of representatives in $\widetilde{P}(\mathcal{O}_\mathfrak{p})$
    	for the quotient
    	$H/(\epsilon^{-1}H\epsilon\cap H)$.
    	Note that
    	for any two
    	$\epsilon,\epsilon'
    	\in
    	\widetilde{T}_{\widetilde{P}}^+(E_\mathfrak{p})$,
    	\begin{align*}
    	\sharp X_{\epsilon\epsilon'}
    	&
    	=
    	[H:(\epsilon\epsilon')^{-1}H(\epsilon\epsilon')\cap H]
    	=
    	[\widetilde{P}(\mathcal{O}_\mathfrak{p})
    	:(\epsilon\epsilon')^{-1}
    	\widetilde{P}(\mathcal{O}_\mathfrak{p})(\epsilon\epsilon')]
    	\\
    	&
    	=
    	[\widetilde{P}(\mathcal{O}_\mathfrak{p})
    	:\epsilon^{-1}\widetilde{P}(\mathcal{O}_\mathfrak{p})\epsilon]
    	\cdot
    	[\epsilon^{-1}\widetilde{P}(\mathcal{O}_\mathfrak{p})\epsilon
    	:
    	(\epsilon\epsilon')^{-1}
    	\widetilde{P}(\mathcal{O}_\mathfrak{p})(\epsilon\epsilon')]
    	\\
    	&
    	=
    	\sharp X_\epsilon\cdot
    	[\widetilde{P}(\mathcal{O}_\mathfrak{p}):
    	(\epsilon')^{-1}
    	\widetilde{P}(\mathcal{O}_\mathfrak{p})\epsilon']
    	\\
    	&
    	=
    	\sharp X_\epsilon\cdot
    	\sharp X_{\epsilon'}.
    	\end{align*}
    	Combined with the inclusion
    	$H\epsilon\epsilon'H
    	\subset
    	H\epsilon H\epsilon'H$,
    	this shows that
    	\[
    	H\epsilon\epsilon' H
    	=
    	H\epsilon H\cdot H\epsilon'H
    	=
    	H\epsilon'H\cdot H\epsilon H.
    	\]
    	This shows that the map
    	sending $\epsilon$ to
    	$1_{\pi_{\widetilde{P}}(\epsilon)}$
    	is well-defined
    	and is indeed a morphism of
    	$\mathbb{Z}$-algebras.    	
    	Next we show that it is injective.
    	Note that we have inclusions
    	$G(E_\mathfrak{p})
    	\hookrightarrow
    	\mathrm{GSp}(V,\psi)(E_\mathfrak{p})$,
    	$\widetilde{T}_{\widetilde{P}}^+(E_\mathfrak{p})
    	\rightarrow
    	\widetilde{T}_{L_V}^+(E_\mathfrak{p})$,
    	the latter of which induces an inclusion
    	$\mathbb{Z}[\widetilde{T}_{\widetilde{P}}^+(E_\mathfrak{p})]
    	\hookrightarrow
    	\mathbb{Z}[\widetilde{T}_{L_V}^+(E_\mathfrak{p})]$.
    	The image
    	$\mathrm{Im}(\pi_{\widetilde{P}})$
    	is a $\mathbb{Z}$-algebra
    	generated by the characteristic functions
    	of the double cosets
    	$\pi_{\widetilde{P}}(\epsilon)$
    	for all
    	$\epsilon\in
    	\widetilde{T}_{\widetilde{P}}^+(E_\mathfrak{p})$.
    	Now consider the following map
    	\[
    	\varphi
    	\colon
    	\mathrm{Im}(\pi_{\widetilde{P}})
    	\rightarrow
    	\mathrm{Im}(\pi_{\widetilde{P}_V}),
    	\quad
    	\pi_{\widetilde{P}}(\epsilon)
    	\mapsto
    	\pi_{\widetilde{P}_V}(\epsilon)
    	,\,
    	\forall
    	\epsilon\in
    	\widetilde{T}_{\widetilde{P}}^+(E_\mathfrak{p});
    	\]
    	It is clear that $\varphi$ is injective.
    	Moreover, we have a commutative diagram
    	(the isomorphism in the bottom line is
    	\cite[Proposition 2.1]{Hida1995}):
    	\[
    	\begin{tikzcd}
    	\mathbb{Z}[\widetilde{T}_{\widetilde{P}}^+(E_\mathfrak{p})]
    	\arrow[r]
    	\arrow[d,hook]
    	&
    	\mathrm{Im}(\pi_{\widetilde{P}})
    	\arrow[d,hook]
    	\\
    	\mathbb{Z}[\widetilde{T}_{L_V}^+(E_\mathfrak{p})]
    	\arrow[r,"\simeq"]
    	&
    	\mathrm{Im}(\pi_{\widetilde{P}_V})
    	\end{tikzcd}
    	\]
    	From this we see that
    	the top line in the above diagram is
    	injective.
    \end{proof}

    We will identify
    $\mathbb{Z}[\widetilde{T}_{\widetilde{P}}^+(E_\mathfrak{p})]$
    with its images inside
    $\mathcal{C}
    (G(E_\mathfrak{p})//I_{\widetilde{P}}(n),\mathbb{Z})$
    and
    $\mathcal{C}
    (G(E_\mathfrak{p})//I_{S\widetilde{P}}(n),\mathbb{Z})$.

    \subsection{Hecke operators on the Igusa tower}
    By the discussion after
    Proposition \ref{Hecke operator preserves top rep},
    we will also need to modify the integral
    structure on the
    space of $p$-adic modular forms.
    We write
    $Sh^{\Sigma,\mu}_\mathrm{rig}$
    for the rigid space associated to the
    formal scheme
    $Sh^{\Sigma,\mu}_\infty$.
    Similarly,
    write
    $\mathrm{Ig}^?_\mathrm{rig}$
    for the rigid space associated to
    $\mathrm{Ig}_\infty^?$
    ($?=G,P,L$).
    We fix a character
    $\lambda\in
    X^{\ast}_\mathrm{dm}
    (\widetilde{T}_{\widetilde{P}})$.
    Recall we have a sheaf
    $\mathcal{V}^\Sigma_\lambda(-C^\Sigma)$
    over
    $Sh^\Sigma$.
    Then we write
    $
    \mathcal{R}_m[\lambda^{-1}]
    $
    for the restriction of
    $\mathcal{V}^\Sigma_\lambda(-C^\Sigma)$
    to
    $Sh^{\Sigma,\mu}_m$,
    and
    $
    \mathcal{R}_\mathrm{rig}[\lambda^{-1}]
    $
    for the rigid sheaf associated to the formal limit
    $
    \mathcal{R}_\infty[\lambda^{-1}]
    $
    of the projective system
    $(\mathcal{R}_m[\lambda^{-1}])_{m\geq1}$.
    Here
    $[\lambda^{-1}]$
    means the subsheaf on which
    $\widetilde{T}_{\widetilde{P}}$ acts by the character
    $\lambda$.
    Then we have the relation
    \[
    H^0(Sh^{\Sigma,\mu}_\mathrm{rig},
    \mathcal{R}_\mathrm{rig}[\lambda^{-1}])
    =
    H^0(Sh^{\Sigma,\mu}_\infty,
    \mathcal{R}_\infty[\lambda^{-1}])
    \otimes_{\mathcal{O}_\mathfrak{p}}E_\mathfrak{p}.
    \]
    On the other hand, we have a morphism
    $\varphi'
    \colon
    \mathrm{Ig}_m/\widetilde{P}_L^\mathrm{der}(\mathcal{O}_\mathfrak{p})
    \rightarrow
    Sh^{\Sigma,\mu}_m$
    and thus a sheaf
    \[
    \mathcal{R}^{\mathrm{top}}_m[\lambda^{-1}]
    :=\varphi'_\ast
    (\mathcal{O}_{\mathrm{Ig}_m/
    	\widetilde{P}_L^\mathrm{der}(\mathcal{O}_\mathfrak{p})}[\lambda^{-1}])
    \]
    over
    $Sh^{\Sigma,\mu}_m$
    and the rigid sheaf
    $
    \mathcal{R}^{\mathrm{top}}_\mathrm{rig}[\lambda^{-1}]
    $
    associated to the formal sheaf
    $
    \mathcal{R}^{\mathrm{top}}_\infty[\lambda^{-1}].
    $ 
    Again we have a similar identity as above:
    \[
    H^0(Sh^{\Sigma,\mu}_\mathrm{rig},
    \mathcal{R}^\mathrm{top}_\mathrm{rig}[\lambda^{-1}])
    =
    H^0(Sh^{\Sigma,\mu}_\infty,
    \mathcal{R}_\infty^\mathrm{top}[\lambda^{-1}])
    \otimes_{\mathcal{O}_\mathfrak{p}}E_\mathfrak{p}.
    \]
    The natural morphism
    $\mathrm{Ig}_m/\widetilde{P}_L^\mathrm{der}(\mathcal{O}_\mathfrak{p})
    \rightarrow
    \mathrm{Ig}_m/\widetilde{P}_L(\mathcal{O}_\mathfrak{p})$
    gives rise to a morphism
    $\mathcal{R}_?[\lambda^{-1}]
    \rightarrow
    \mathcal{R}^{\mathrm{top}}_?[\lambda^{-1}]$
    of sheaves over
    $Sh^{\Sigma,\mu}_?$
    ($?=m,\infty,\mathrm{rig}$).
    For an affine formal open subset
    $\mathrm{Spf}(A)$
    of
    $Sh^{\Sigma,\mu}_\infty$
    and the corresponding affinoid open subset
    $\mathrm{Sp}(A[\frac{1}{p}])$
    of
    the rigid space
    $Sh^{\Sigma,\mu}_\mathrm{rig}$,
    we have the following commutative diagram of
    functions over these open subsets
    \[
    \begin{tikzcd}
    \mathcal{R}_\infty[\lambda^{-1}](A)
    \arrow[r,"\iota_\infty"]
    \arrow[d,"\mathrm{ev}_{\mathcal{O}_\mathfrak{p}}"]
    &
    \mathcal{R}_\mathrm{rig}[\lambda^{-1}](A[\frac{1}{p}])
    \arrow[d,"\mathrm{ev}_{E_\mathfrak{p}}"]
    \\
    \mathcal{R}_\infty^\mathrm{top}[\lambda^{-1}](A)
    \arrow[r,"\iota_\infty^\mathrm{top}"]
    &
    \mathcal{R}_\mathrm{rig}^\mathrm{top}
    [\lambda^{-1}](A[\frac{1}{p}])
    \end{tikzcd}
    \]
    As in the previous section,
    we put
    \[
    \widetilde{\mathcal{R}}_\infty[\lambda^{-1}](A)
    :=
    \mathrm{ev}_{E_\mathfrak{p}}^{-1}
    (\mathrm{Im}(\iota_\infty^\mathrm{top})),
    \]
    which is a lattice in
    $\mathcal{R}_\mathrm{rig}[\lambda^{-1}](A[\frac{1}{p}])$
    containing the lattice
    $\mathcal{R}_\infty[\lambda^{-1}](A)$.
    This is analogous to the construction of
    $\widetilde{R}_B[\lambda^{-1}]$
    where $B$ is a finite flat $\mathcal{O}_\mathfrak{p}$-algebra.
    Clearly, this construction is functorial in
    $A$ and glues well,
    thus defines a new formal sheaf
    $
    \widetilde{\mathcal{R}}_\infty[\lambda^{-1}]
    $
    on the formal scheme
    $Sh^{\Sigma,\mu}_\infty$.

    Next we review the construction of the correspondence
    on the integral Siegel Shimura variety
    \[
    Sh(V)
    :=
    Sh(\mathrm{GSp}(V,\psi),S^\pm,K_V)
    \]
    (here $K_V$ is a compact open subgroup of
    $\mathrm{GSp}(V,\psi)(\mathbb{A}_f)$
    such that
    $K_V\cap G(\mathbb{A}_f)
    =K$)
    (\textit{cf.} \cite[§5.1.3]{Pilloni2012}).
    We fix an element
    $\epsilon\in
    \widetilde{T}_{\widetilde{P}}(E_\mathfrak{p})$.
    Viewed as an element in
    $\mathrm{GSp}(V,\psi)$,
    we can write
    $\epsilon
    =\mathrm{diag}(\varepsilon,\nu(\epsilon)\varepsilon^{-t})$
    with
    $\varepsilon\in
    L(E_\mathfrak{p})$.
    Then the correspondence
    $\mathrm{Ig}^{\mathrm{GSp}(V,\pi)}_{\infty,n}
    /\widetilde{P}_{L_V}(\epsilon)$
    over the $\mu$-ordinary
    locus of the (non-compactified) Shimura variety
    $Sh(\mathrm{GSp}(V,\psi),S^\pm,K_V)^{\mu}_\infty$
    parameterizes the following data:
    for any $\mathbb{Z}_p$-algebra
    $A$,
    $\mathrm{Ig}_{\infty,n}/\widetilde{P}_V(\epsilon)(A)$
    consists of quintuples
    $(\mathcal{A},\widetilde{\mathcal{A}},
    \pi,\psi_n,\widetilde{\psi}_n)$
    where
    \begin{enumerate}
    	\item 
    	$\mathcal{A}$ is a
    	$\mu$-ordinary principally polarized abelian scheme
    	over $\mathrm{Spec}(A)$
    	and a $\widetilde{P}_V
    	(\mathcal{O}_\mathfrak{p}/\mathfrak{p}^n)$-coset of
    	isomorphisms
    	$\psi_n
    	\colon
    	\mathcal{A}_{x_0}[p^n]
    	\simeq
    	\mathcal{A}[p^n]$
    	such that there exists an isomorphism
    	$\psi_\infty
    	\colon
    	\mathcal{A}_{x_0}[p^\infty]
    	\simeq
    	\mathcal{A}[p^\infty]$
    	with
    	$\psi_\infty\equiv
    	\psi_n(\mathrm{mod}\,p^n)$;
    	
    	\item 
    	$\widetilde{\mathcal{A}}$ is a $\mu$-ordinary
    	principally polarized
    	abelian scheme over
    	$\mathrm{Spec}(A)$
    	and a $\widetilde{P}_V
    	(\mathcal{O}_\mathfrak{p}/\mathfrak{p}^n)$-coset
    	of isomorphisms
    	$\widetilde{\psi}_n
    	\colon
    	\mathcal{A}_{x_0}[p^n]
    	\simeq
    	\widetilde{\mathcal{A}}[p^n]$
    	such that
    	there exists an isomorphism
    	$\widetilde{\psi}_\infty
    	\colon
    	\mathcal{A}_{x_0}[p^\infty]
    	\simeq
    	\widetilde{\mathcal{A}}[p^\infty]$
    	with
    	$\widetilde{\psi}_\infty\equiv
    	\widetilde{\psi}_n(\mathrm{mod}\,p^n)$;
    	
    	\item 
    	$\pi\colon
    	\mathcal{A}
    	\rightarrow
    	\widetilde{\mathcal{A}}$
    	is a $p$-isogeny of similitude factor
    	$\nu(\epsilon)$
    	of abelian schemes
    	such that
    	on the level of Dieudonn\'{e} crystals,
    	we have
    	$\mathbb{D}
    	((\widetilde{\psi}_\infty)^{-1}
    	\circ
    	\pi
    	\circ
    	\psi_\infty)(\mathbb{W})
    	=\epsilon
    	\colon
    	\mathbb{D}(\mathcal{A}_{x_0}[p^\infty])(\mathbb{W})
    	\rightarrow
    	\mathbb{D}(\mathcal{A}_{x_0}[p^\infty])(\mathbb{W})$.
    \end{enumerate}
    We emphasize that we parameterize only the
    $\mu$-ordinary abelian schemes.
    Now for
    $\epsilon\in
    \widetilde{T}_{\widetilde{P}}^+(E_\mathfrak{p})$,
    we define the correspondence
    \[
    \mathrm{Ig}_{\infty,n}^\mu
    /\widetilde{P}_L(\epsilon)
    \]
    to be the pull-back of
    $\mathrm{Ig}^{\mathrm{GSp}(V,\psi)}_{\infty,n}
    /\widetilde{P}_{L_V}(\epsilon)$
    along the embedding
    $Sh^\mu_\infty
    \hookrightarrow
    Sh(V)^\mu_\infty$.
    Clearly, as in the Siegel case,
    we have two natural projections
    \[
    \mathrm{pr}_1,\mathrm{pr}_2
    \colon
    \mathrm{Ig}_{\infty,n}^\mu/\widetilde{P}_L(\epsilon)
    \rightarrow
    \mathrm{Ig}_{\infty,n}^\mu/\widetilde{P}_L,
    \]
    where
    $\mathrm{pr}_1$
    takes the quintuple
    $(\mathcal{A},\widetilde{\mathcal{A}},
    \pi,\psi_n,\widetilde{\psi}_n)$
    to
    $(\mathcal{A},\psi_n)$
    and
    $\mathrm{pr}_2$ takes it to
    $(\widetilde{\mathcal{A}},\widetilde{\psi}_n)$.
    Here
    \[
    \mathrm{Ig}_{\infty,n}/\widetilde{P}_L
    :=
    \mathrm{Ig}_{\infty,n}/
    \widetilde{P}_L(\mathcal{O}_\mathfrak{p}/\mathfrak{p}^n)
    \]
    is defined in the same manner as
    the quotient
    $\mathrm{Ig}_\infty/\widetilde{P}_L(\mathcal{O}_\mathfrak{p})$.
    Similarly we have induced morphisms on
    the corresponding rigid spaces
    \[
    \mathrm{pr}_{1,\mathrm{rig}},
    \mathrm{pr}_{2,\mathrm{rig}}
    \colon
    \mathrm{Ig}_{\mathrm{rig},n}/\widetilde{P}_L(\epsilon)
    \rightarrow
    \mathrm{Ig}_{\mathrm{rig},n}/\widetilde{P}_L.
    \]
    Now over
    the correspondence
    $\mathrm{Ig}^{\mathrm{GSp}(V,\psi)}_{\infty,n}/
    \widetilde{P}(\epsilon)$,
    we have two universal abelian schemes
    $\mathcal{A}^{\mathrm{GSp}(V,\psi)}$,
    $\widetilde{\mathcal{A}}^{\mathrm{GSp}(V,\psi)}$
    and a $p$-isogeny
    \[
    \pi^{\mathrm{GSp}(V,\psi)}
    =
    \pi^{\mathrm{GSp}(V,\psi),\epsilon}
    \colon
    \mathcal{A}^{\mathrm{GSp}(V,\psi)}
    \rightarrow
    \widetilde{\mathcal{A}}^{\mathrm{GSp}(V,\psi)}
    \]
    whose trivialization via
    (lifts
    $\psi^{\mathrm{GSp}(V,\psi)}_\infty$,
    resp.,
    $\widetilde{\psi}^{\mathrm{GSp}(V,\psi)}_\infty$
    of)
    $\psi^{\mathrm{GSp}(V,\psi)}_n$,
    resp.
    $\widetilde{\psi}^{\mathrm{GSp}(V,\psi)}_n$
    is $\epsilon$.
    We denote 
    by
    $\mathcal{A}^G$,
    $\widetilde{\mathcal{A}}^G$,
    $\psi_\infty^G$,
    $\widetilde{\psi}_\infty^G$
    the corresponding pull-backs along
    $Sh^\mu_n
    \rightarrow
    Sh(V)^\mu_n$    
    and
    $\pi^G=\pi^{G,\epsilon}
    \colon
    \mathcal{A}^G
    \rightarrow
    \widetilde{\mathcal{A}}^G$.
    From this, we get a morphism
    \[
    (\pi^{G,\epsilon})^\ast
    \colon
    H^1_\mathrm{dR}
    (\widetilde{\mathcal{A}}^G/
    (\mathrm{Ig}_{\infty,n}^\mu/\widetilde{P}(\epsilon)))
    \rightarrow
    H^1_\mathrm{dR}
    (\mathcal{A}^G/
    (\mathrm{Ig}_{\infty,n}^\mu/\widetilde{P}(\epsilon))).
    \]
    Recall that we have just constructed
    $\mathcal{R}_\infty[\lambda^{-1}]$
    over
    $Sh^{\Sigma,\mu}_\infty$.
    Consider the morphism
    $\mathrm{Ig}_{m,n}^\mu/\widetilde{P}_L
    \rightarrow
    Sh^{\mu}_m$,
    and
    we denote the pull-backs along this morphism of
    the sheaf
    $\mathcal{R}_m[\lambda^{-1}]$
    by
    $
    \mathcal{R}_m^\mathrm{Ig}[\lambda^{-1}]
    $
    and similarly we have
    $
    \mathcal{R}_\infty^\mathrm{Ig}[\lambda^{-1}]$,
    $
    \mathcal{R}_\mathrm{rig}^\mathrm{Ig}[\lambda^{-1}]$
    etc..
    Using the above morphism
    $(\pi^{G,\epsilon})^\ast$,
    we can construct a morphism
    \[
    (\pi^{G,\epsilon})^\ast
    \colon
    \mathrm{pr}_{2,\mathrm{rig}}^\ast
    (\mathcal{R}_\mathrm{rig}^\mathrm{Ig}[\lambda^{-1}])
    \rightarrow
    \mathrm{pr}_{1,\mathrm{rig}}^\ast
    (\mathcal{R}_\mathrm{rig}^\mathrm{Ig})[\lambda^{-1}]
    \]as follows:
    consider an affinoid open subset
    $\mathrm{Sp}(A[\frac{1}{p}])$
    of
    $\mathrm{Ig}_{\mathrm{rig},n}^\mu
    /\widetilde{P}_L(\epsilon)$
    associated to an affine formal open subset
    $\mathrm{Spf}(A)$
    of
    $\mathrm{Ig}_{\infty,n}^\mu
    /\widetilde{P}_L(\epsilon)$.
    Then
    $\mathrm{pr}_{?,\mathrm{rig}}^\ast
    (\mathcal{R}_\mathrm{rig}^\mathrm{Ig}[\lambda^{-1}])$
    consists of rational functions on
    $G(A[\frac{1}{p}])$
    with values in $A[\frac{1}{p}]$
    invariant under the right action of the unipotent
    subgroup
    $\widetilde{U}(A[\frac{1}{p}])$
    and on which the torus
    $\widetilde{T}(A[\frac{1}{p}])$
    acts by the character
    $\lambda$
    ($?=1,2$).
    Caution that the difference between
    $\mathrm{pr}_{1,\mathrm{rig}}
    (\mathcal{R}_\mathrm{rig}^\mathrm{Ig}[\lambda^{-1}])$
    and
    $\mathrm{pr}_{2,\mathrm{rig}}
    (\mathcal{R}_\mathrm{rig}^\mathrm{Ig}[\lambda^{-1}])$
    lies in how we identify the unipotent group
    $\widetilde{U}(A[\frac{1}{p}])$
    as explained below.
    One can identify
    $G(A[\frac{1}{p}])$
    with the set of isomorphisms
    $f
    \colon
    \mathcal{A}_{x_0}[p^\infty]
    \rightarrow
    \mathcal{A}^G[p^\infty]
    $
    over (or rather, restricted to)
    $\mathrm{Sp}(A[\frac{1}{p}])$,
    preserving the Hodge tensors on both sides.
    Note that by the definition of the correspondences
    $\mathrm{Ig}_{m,n}^\mu
    /\widetilde{P}_L(\epsilon)$,
    for each such $f$,
    there is another
    $\widetilde{f}
    \colon
    \mathcal{A}_{x_0}[p^\infty]
    \rightarrow
    \widetilde{\mathcal{A}}^G[p^\infty]$
    over
    $\mathrm{Sp}(A[\frac{1}{p}])$,
    preserving the Hodge tensors and moreover
    $\mathbb{D}(\widetilde{f}^{-1}
    \circ
    \pi
    \circ
    f)(\mathbb{W})
    =\epsilon^{-1}$
    (which is necessarily unique).
    Then we define
    the morphism
    $(\pi^{G,\epsilon})^\ast$
    on the affinoid open subset
    $\mathrm{Sp}(A[\frac{1}{p}])$
    by the formula
    \[
    (\pi^{G,\epsilon})^\ast
    \colon
    \mathrm{pr}_{2,\mathrm{rig}}^\ast
    (\mathcal{R}_\mathrm{rig}^\mathrm{Ig}[\lambda^{-1}])
    (A[\frac{1}{p}])
    \rightarrow
    \mathrm{pr}_{1,\mathrm{rig}}^\ast
    (\mathcal{R}_\mathrm{rig}^\mathrm{Ig}[\lambda^{-1}])
    (A[\frac{1}{p}]),
    \]
    \[
    F
    \mapsto
    \big(
    \mathbb{D}(f)\mapsto
    F(\mathbb{D}(\widetilde{f}))
    \big).
    \]
    This process is clearly functorial in
    the affinoid open subset
    $\mathrm{Sp}(A[\frac{1}{p}])$
    and thus we can globalize the
    above process and obtain the morphism
    $(\pi^{G,\epsilon})^\ast$.
    In particular, if we fix a geometric point
    $\overline{x}
    \in
    (\mathrm{Ig}_{\infty,n}/\widetilde{P}_L(\epsilon))
    (\overline{\mathbb{F}}_p)$
    and set
    $A$
    to be the strict Henselization of
    $\mathrm{Ig}_{\infty,n}/\widetilde{P}_L(\epsilon)$
    at the point $\overline{x}$,
    then
    (\textit{cf.} \cite[Lemme 5.1]{Pilloni2012})
    \begin{proposition}\label{local-global commutativity}
    	For $\lambda\in
    	X^{\ast}_\mathrm{dm}
    	(\widetilde{T}_{\widetilde{P}})$
    	and $\overline{x}$ as above,
    	we have a commutative diagram
    	with isomorphic vertical arrows:
    	\[
    	\begin{tikzcd}
    	R_A[\lambda^{-1}]
    	\arrow[r,hook]
    	\arrow[d]
    	&
    	R_{A[\frac{1}{p}]}[\lambda^{-1}]
    	\arrow[r,"\mathbb{T}_\epsilon"]
    	\arrow[d]
    	&
    	R_{A[\frac{1}{p}]}[\lambda^{-1}]
    	\arrow[d]
    	\\
    	\mathrm{pr}_{2,\infty}^\ast
    	(\mathcal{R}_\infty^\mathrm{Ig}[\lambda^{-1}])(A)
    	\arrow[r]
    	&
    	\mathrm{pr}_{2,\mathrm{rig}}^\ast
    	(\mathcal{R}_\mathrm{rig}^\mathrm{Ig}[\lambda^{-1}])
    	(A[\frac{1}{p}])
    	\arrow[r,"(\pi^{G,\epsilon})^\ast"]
    	&
    	\mathrm{pr}_{1,\mathrm{rig}}^\ast
    	(\mathcal{R}_\mathrm{rig}^\mathrm{Ig}[\lambda^{-1}])
    	(A[\frac{1}{p}])
    	\end{tikzcd}
    	\]
    \end{proposition}
    \begin{proof}
    	By the above discussion,
    	we see easily that the vertical arrows are indeed
    	isomorphisms.
    	It remains to show that
    	the second square is commutative.
    	This follows from our definition of
    	$(\pi^{G,\epsilon})^\ast$.
    	Indeed,
    	for any
    	$f$ as in the definition of
    	$(\pi^{G,\epsilon})^\ast$
    	(we identify $f$ and its Dieudonn\'{e} crystal
    	counterpart),
    	passing to the contravariant
    	Dieudonn\'{e} crystals,
    	we have an identity
    	$\mathbb{D}(f)
    	\circ
    	\mathbb{D}(\pi)
    	\circ
    	\mathbb{D}(\widetilde{f})^{-1}
    	=\epsilon^{-1}$.
    	Moreover,
    	an element
    	$F\in
    	R_{A[\frac{1}{p}]}[\lambda^{-1}]$
    	is sent to an element
    	$F_1\in
    	\mathrm{pr}_{1,\mathrm{rig}}^\ast
    	(\mathcal{R}_\mathrm{rig}^\mathrm{Ig}[\lambda^{-1}])$,
    	resp.,
    	$F_2\in
    	\mathrm{pr}_{2,\mathrm{rig}}^\ast
    	(\mathcal{R}_\mathrm{rig}^\mathrm{Ig}[\lambda^{-1}])$
    	which is given by
    	\[
    	F_1(g\circ\mathbb{D}(\psi^G_\infty))
    	=
    	F(g),
    	\text{resp}.,
    	F_2(g\circ\mathbb{D}(\widetilde{\psi}^G_\infty))
    	=F(g),
    	\forall
    	g\in G(E_\mathfrak{p}).
    	\]
    	Moreover,
    	note that for
    	$\mathbb{D}(\phi)
    	=g\circ
    	\mathbb{D}(\psi^G_\infty)$,
    	one has
    	\[
    	\mathbb{D}(\widetilde{\phi})
    	=
    	\epsilon
    	\circ
    	g
    	\circ
    	\epsilon^{-1}
    	\circ
    	\mathbb{D}(\varphi^G_\infty).
    	\]
    	Thus
    	we have
    	\[
    	((\pi^{G,\epsilon})^\ast F_2)
    	(g\circ\mathbb{D}(\psi^G_\infty))
    	=
    	F_2
    	(
    	\epsilon
    	g
    	\epsilon^{-1}
    	\mathbb{D}(\varphi^G_\infty)
    	)
    	=
    	F(\epsilon
    	g
    	\epsilon^{-1})
    	=
    	(\mathbb{T}_\epsilon F)(g)
    	=
    	(\mathbb{T}_\epsilon F)_1
    	(g\circ
    	\mathbb{D}(\varphi^G_\infty)).
    	\]
    	From this we conclude that
    	$(\pi^{G,\epsilon})^\ast F_2
    	=(\mathbb{T}_\epsilon F)_1$,
    	thus the commutativity follows.
    \end{proof}

    Moreover, it is also easy to see 
    the above commutative diagram is
    functorial in $A$
    and thus again we can globalize the diagram and
    obtain
    the following morphisms
    \begin{align*}
    (\pi^{G,\epsilon})^\ast
    &
    \colon
    \mathrm{pr}_{2,\infty}
    (\mathcal{R}_\infty^\mathrm{top,Ig}[\lambda^{-1}])
    \rightarrow
    \mathrm{pr}_{1,\infty}
    (\mathcal{R}_\infty^\mathrm{top,Ig}[\lambda^{-1}]),
    \\
    (\pi^{G,\epsilon})^\ast
    &
    \colon
    \mathrm{pr}_{2,\mathrm{rig}}
    (\mathcal{R}_\mathrm{rig}^\mathrm{top,Ig}[\lambda^{-1}])
    \rightarrow
    \mathrm{pr}_{1,\mathrm{rig}}
    (\mathcal{R}_\mathrm{rig}^\mathrm{top,Ig}[\lambda^{-1}]),
    \\
    (\pi^{G,\epsilon})^\ast
    &
    \colon
    \mathrm{pr}_{2,\infty}
    (\widetilde{\mathcal{R}}_\infty^\mathrm{Ig}[\lambda^{-1}])
    \rightarrow
    \mathrm{pr}_{1,\infty}
    ((\widetilde{\mathcal{R}}_\infty^\mathrm{Ig}[\lambda^{-1}])).
    \end{align*}

    Then we consider the trace map
    over the formal scheme
    $\mathrm{Ig}_{\infty,n}^\mu/\widetilde{P}_L$:
    \[
    \mathrm{Tr}_{\mathrm{pr}_{1,\infty}}
    \colon
    (\mathrm{pr}_{1,\infty})_\ast
    (\mathcal{O}_{\mathrm{Ig}_{\infty,n}^\mu
    	/\widetilde{P}_L(\epsilon)})
    \rightarrow
    \mathcal{O}_{\mathrm{Ig}_{\infty,n}^\mu
    	/\widetilde{P}_L}.
    \]
    We want to study the $p$-integrality of this
    morphism.
    Recall that we have the unipotent radical
    $U_{\widetilde{P}}$
    of
    the parabolic subgroup
    $\widetilde{P}$
    of $G$.
    Note that
    $\widetilde{T}_{\widetilde{P}}^+(E_\mathfrak{p})$
    acts by inverse conjugation on
    $U_{\widetilde{P}}(\mathcal{O}_\mathfrak{p})$.

    \begin{definition}
    	For an element
    	$\epsilon
    	\in
    	\widetilde{T}_{\widetilde{P}}^+(E_\mathfrak{p})$,
    	we set
    	\[
    	m_\epsilon
    	:=
    	\sharp\frac{U_{\widetilde{P}}
    		(\mathcal{O}_\mathfrak{p})}{\epsilon
    		U_{\widetilde{P}}
    		(\mathcal{O}_\mathfrak{p})\epsilon^{-1}\cap
    		U_{\widetilde{P}}(\mathcal{O}_\mathfrak{p})}.
    	\]
    \end{definition}
    
    \begin{proposition}\label{Integrality of Hecke operator}
    	For any
    	$\epsilon\in
    	\widetilde{T}_{\widetilde{P}}^+(E_\mathfrak{p})$,
    	the morphism
    	\[
    	\mathrm{Tr}_\epsilon
    	:=
    	\frac{1}{m_\epsilon}
    	\mathrm{Tr}_{\mathrm{pr}_{1,\infty}}
    	\colon
    	(\mathrm{pr}_{1,\infty})_\ast
    	(\mathcal{O}_{\mathrm{Ig}_{\infty,n}^\mu
    		/\widetilde{P}_L(\epsilon)})
    	\rightarrow
    	\mathcal{O}_{\mathrm{Ig}_{\infty,n}^\mu
    		/\widetilde{P}_L}
    	\]
    	is well-defined
    	(in other words, the morphism
    	$\mathrm{Tr}_{\mathrm{pr}_{1,\infty}}$
    	is divisible by $m_\epsilon$).
    \end{proposition}
    \begin{proof}
    	The proof is similar to
    	the ones in
    	\cite[Lemma 6.6]{Hida2002}
    	and
    	\cite[Appendice 1]{Pilloni2012},
    	using the Serre-Tate theory developed in
    	\cite{ShankarZhou2016}.
    	First we recall some relevant results in
    	\cite{ShankarZhou2016}.
    	Let
    	$x_0
    	\in
    	Sh^\mu_1$
    	be a geometric point in the $\mu$-ordinary
    	locus
    	(as we have done since the beginning)
    	and
    	$\mathcal{BT}$
    	be the $p$-divisible group associated to
    	the abelian scheme
    	$\mathcal{A}_{x_0}$.
    	Then we write
    	$\mathcal{D}(\mathcal{BT})$
    	for the deformation space of
    	$\mathcal{BT}$ over $\mathrm{Spec}(\mathbb{W})$
    	and
    	$\mathcal{D}_G(\mathcal{BT})$
    	the subspace of
    	$\mathcal{D}(\mathcal{BT})$
    	consisting of $G$-adapted deformations
    	(\textit{cf.} \cite[§4]{ShankarZhou2016}).
    	Then by
    	\cite[Theorem 4.9]{ShankarZhou2016},
    	$\mathcal{D}_G(\mathcal{BT})$
    	is a shifted subcascade
    	and has a dense subset consisting of
    	CM points(\textit{i.e.} torsion points).
    	More precisely,
    	recall that we have a slope decomposition
    	$\mathcal{BT}
    	=\prod_{i=1}^r\mathcal{BT}_i$
    	each piece of slope
    	$1\geq\lambda_1>\lambda_2>\cdots\lambda_r\geq0$
    	and its canonical lifting
    	$\widetilde{\mathcal{BT}}
    	=\prod_{i=1}^r
    	\widetilde{\mathcal{BT}}_i$.
    	We write
    	(in the notation of \cite[Definition 4.4]{ShankarZhou2016})
    	$E_{i,j}
    	=\mathrm{Ext}
    	(\widetilde{\mathcal{BT}}_i,\widetilde{\mathcal{BT}}_j)$
    	for
    	$i<j$.
    	Let
    	$\mathcal{BT}_{i,j}$
    	denote the product
    	$\prod_{k=i}^j\mathcal{BT}_k$
    	and
    	$\mathcal{D}_G(\mathcal{BT}_{i,j})$
    	its corresponding deformation space over
    	$\mathrm{Spec}(\mathbb{W})$
    	($i\leq j$).
    	Then \cite[Proposition 4.8, Theorem 4.9]{ShankarZhou2016}
    	show that
    	\begin{enumerate}
    		\item 
    		the natural morphisms
    		$\lambda_{i,j}
    		\colon
    		\mathcal{D}_G(\mathcal{BT}_{i,j})
    		\rightarrow
    		\mathcal{D}_G(\mathcal{BT}_{i,j-1})$
    		and
    		$\rho_{i,j}
    		\colon
    		\mathcal{D}_G(\mathcal{BT}_{i,j})
    		\rightarrow
    		\mathcal{D}_G(\mathcal{BT}_{i+1,j})$
    		satisfy the commutative relation
    		$\rho_{i,j-1}\circ\lambda_{i,j}
    		=\lambda_{i+1,j}\circ\rho_{i,j}$;
    		
    		\item 
    		$\mathcal{D}_G(\mathcal{BT}_{i,j})$
    		is a bi-extension of
    		$(\mathcal{D}_G(\mathcal{BT}_{i,j-1}),
    		\mathcal{D}_G(\mathcal{BT}_{i+1,j}))$
    		by
    		$E_{i,j}\times
    		\mathcal{D}_G(\mathcal{BT}_{i+1,j-1})$.
    	\end{enumerate}
    	Recall that over the correspondence
    	$\mathrm{Ig}_{\infty,1}^\mu/\widetilde{P}_L(\epsilon)$,
    	we have two abelian schemes
    	$\mathcal{A}^G$ and
    	$\widetilde{\mathcal{A}}^G$
    	and a $p$-isogeny
    	$\pi^{G,\epsilon}
    	\colon
    	\mathcal{A}^G
    	\rightarrow
    	\widetilde{\mathcal{A}}^G$.
    	This $p$-isogeny induces an action of
    	$\epsilon$
    	on the deformation space
    	$\mathcal{D}_G(\mathcal{BT})
    	=\mathcal{D}_G(\mathcal{BT}_{1,r})$.
    	More precisely,
    	we view $\epsilon$
    	as an element in
    	$\mathrm{GSp}(V,\psi)(E_\mathfrak{p})$
    	and suppose that
    	the parabolic subgroup
    	$P_V$
    	of $\mathrm{GSp}(V,\psi)$
    	corresponds to a partition
    	$\mathrm{dim}(V)
    	=n_1+n_2+\cdots+n_r$
    	(recall $P_V$ stabilizes the filtration induced by the
    	cocharacter $\mu$)
    	and
    	the parabolic subgroup $\widetilde{P}_V$
    	(contained in $P_V$)
    	corresponds to partitions
    	$n_i=n_{i,1}+n_{i,2}+\cdots+n_{i,s_i}$
    	for
    	$i=1,\cdots,r$.
    	Then
    	$\epsilon
    	=\mathrm{diag}
    	(\epsilon_1,\cdots,\epsilon_r)$
    	with each diagonal matrix
    	$\epsilon_i
    	=\mathrm{diag}
    	(p^{t_{i,1}}1_{n_{i,1}},\cdots,
    	p^{t_{i,s_i}}1_{n_{i,s_i}})$
    	such that
    	$t_{1,1}\leq t_{1,2}\cdots
    	\leq t_{1,s_1}\leq
    	t_{2,1}\leq
    	\cdots
    	\leq
    	\cdots
    	t_{r,1}
    	\leq
    	\cdots
    	\leq
    	t_{r,s_r}$.
    	By the equivalence between the category of
    	Honda systems
    	over $\mathbb{W}$
    	and the category of $p$-divisible groups
    	over $\mathbb{W}$
    	(\textit{cf.}\cite[Appendix A]{ShankarZhou2016}),
    	we can identify
    	$E_{i,j}$
    	with
    	$\widetilde{P}(\mathcal{O}_\mathfrak{p})\cap
    	P_{i,j}(\mathcal{O}_\mathfrak{p})$,
    	here $P_{i,j}$ is the subgroup of
    	the unipotent subgroup $U_{P_V}$ of $P_V$
    	where the entries at index
    	$(i',j')$ vanish for
    	$i'<j'$ and 
    	$(i',j')\notin
    	[\sum_{k=1}^{i-1}n_k+1,\sum_{k=1}^in_k]
    	\times
    	[\sum_{k=1}^{j-1}n_k+1,\sum_{k=1}^jn_k]
    	$.
    	We see that
    	$\epsilon_i$ acts by conjugation on
    	$E_{i,j}$    	
    	via its conjugate action on
    	the Honda system
    	$\mathbb{D}
    	((\widetilde{\mathcal{BT}}_j)_{\overline{\mathbb{F}}_p})$
    	and thus we have
    	(see also \cite[p.64]{Hida2002})
    	\[
    	\sharp\frac{
    		\mathcal{D}_G(\mathcal{BT})}{\epsilon
    		\mathcal{D}_G(\mathcal{BT})
    		\epsilon^{-1}\cap
    	\mathcal{D}_G(\mathcal{BT})}
        =
        m_\epsilon.
    	\]
    	From this we deduce that
    	$\mathrm{Tr}_\epsilon$ is $p$-integral.
    \end{proof}
    
    We have also the following simple observations
    \begin{lemma}
    	We fix an element
    	$\epsilon
    	\in
    	\widetilde{T}_{\widetilde{P}}^+(E_\mathfrak{p})$.
    	\begin{enumerate}
    		\item 
    		There is a subset
    		$X_\epsilon
    		\subset\widetilde{P}(E_\mathfrak{p})$
    		such that
    		for any $n\geq1$:
    		\[
    		I_{\widetilde{P}}(n)\epsilon I_{\widetilde{P}}(n)
    		=
    		\bigsqcup_{x\in X_\epsilon}
    		I_{\widetilde{P}}(n)\epsilon x.
    		\]
    		
    		\item 
    		We have the following identity of
    		double cosets
    		for any $n\geq2$:
    		\[
    		I_{\widetilde{P}}(n)
    		\epsilon
    		I_{\widetilde{P}}(n)
    		=I_{\widetilde{P}}(n)
    		\epsilon
    		I_{\widetilde{P}}(n-1).
    		\]
    	\end{enumerate}
    \end{lemma}
    \begin{proof}
    	\begin{enumerate}
    		\item 
    		By definition,
    		$I_{\widetilde{P}}(n)
    		\epsilon
    		I_{\widetilde{P}}(n)
    		=\bigsqcup_x
    		I_{\widetilde{P}}(n)\epsilon x$
    		where
    		$x$ runs through
    		$(\epsilon^{-1}
    		I_{\widetilde{P}}(n)
    		\epsilon\cap I_{\widetilde{P}}(n))
    		\backslash
    		I_{\widetilde{P}}(n)$.
    		Since
    		$\epsilon
    		\in
    		\widetilde{T}_{\widetilde{P}}^+(E_\mathfrak{p})$,
    		we can identify the quotient with
    		$\epsilon^{-1}
    		U_{\widetilde{P}}(\mathcal{O}_\mathfrak{p})
    		\epsilon\backslash
    		U_{\widetilde{P}}(\mathcal{O}_\mathfrak{p})$,
    		which is independent of $n$.
    		
    		\item 
    		The proof is the same as
    		\cite[Proposition 5.3]{Pilloni2012}.
    		More precisely,
    		write
    		$I_{\widetilde{P}}(n)^\epsilon
    		:=
    		\epsilon^{-1}
    		I_{\widetilde{P}}(n)
    		\epsilon\cap
    		I_{\widetilde{P}}(n)$
    		and
    		$I_{\widetilde{P}}(n-1)^\epsilon
    		:=
    		\epsilon^{-1}
    		I_{\widetilde{P}}(n)
    		\epsilon
    		\cap
    		I_{\widetilde{P}}(n-1)$.
    		Then we see that both
    		$I_{\widetilde{P}}(n)\backslash
    		I_{\widetilde{P}}(n-1)$
    		and
    		$I_{\widetilde{P}}(n)^\epsilon\backslash
    		I_{\widetilde{P}}(n-1)^\epsilon$
    		are in bijection with
    		$\mathrm{Ker}
    		(U_{\widetilde{P}}^\circ
    		(\mathcal{O}_\mathfrak{p}/\mathfrak{p^n})
    		\rightarrow
    		U_{\widetilde{P}}^\circ
    		(\mathcal{O}_\mathfrak{p}/\mathfrak{p^{n-1}}))$.
    		From this we deduce that
    		the quotients
    		$I_{\widetilde{P}}(n)^\epsilon
    		\backslash
    		I_{\widetilde{P}}(n)$
    		and
    		$I_{\widetilde{P}}(n-1)^\epsilon
    		\backslash
    		I_{\widetilde{P}}(n-1)$
    		are identical using the
    		snake lemma
    		(the former one is \textit{a priori}
    		contained in the latter one).
    	\end{enumerate}
    \end{proof}

    \begin{definition}\label{Definition of Hecke operator at p}
    	For any
    	sheaf
    	$\mathcal{F}$ over
    	$\mathrm{Ig}_{\mathrm{rig},n}^\mu/\widetilde{P}_L$,
    	we define
    	$\mathbb{T}_\epsilon$
    	to be the composition of the following morphisms
    	\[
    	\begin{tikzcd}
    	H^0
    	(\mathrm{Ig}_{\mathrm{rig},n}^\mu/\widetilde{P}_L,\mathcal{F})
    	\arrow[r]
    	\arrow[d,dashed]
    	&
    	H^0
    	(\mathrm{Ig}_{\mathrm{rig},n}^\mu/\widetilde{P}_L(\epsilon),
    	\mathrm{pr}_{2,\mathrm{rig}}^\ast\mathcal{F})
    	\arrow[d,"(\pi^{G,\epsilon})^\ast"']
    	\\
    	H^0
    	(\mathrm{Ig}_{\mathrm{rig},n}^\mu/\widetilde{P}_L,\mathcal{F})
    	&
    	H^0
    	(\mathrm{Ig}_{\mathrm{rig},n}^\mu/\widetilde{P}_L(\epsilon),
    	\mathrm{pr}_{1,\mathrm{rig}}^\ast\mathcal{F})
    	\arrow[l,"\mathrm{Tr}_\epsilon"']
    	\end{tikzcd}
    	\]
    	
    	By considering the affine open formal
    	subsets of
    	$\mathrm{Ig}_{\infty,n}^\mu/\widetilde{P}_L$,
    	we define also
    	the operator
    	\[
    	\mathbb{T}_\epsilon
    	\colon
    	H^0(\mathrm{Ig}_{\infty,n}^\mu/\widetilde{P}_L,
    	\widetilde{\mathcal{R}}_\infty^\mathrm{Ig}[\lambda^{-1}])
    	\rightarrow
    	H^0(\mathrm{Ig}_{\infty,n}^\mu/\widetilde{P}_L,
    	\widetilde{\mathcal{R}}_\infty^\mathrm{Ig}[\lambda^{-1}]).
    	\]
    	By the above arguments,
    	we see that these operators are well-defined.
    \end{definition}

    \begin{proposition}
    	We have the following morphisms
    	of $\mathbb{Z}$-algebras
    	\[
    	\mathbb{Z}[\widetilde{T}_{\widetilde{P}}^+(E_\mathfrak{p})]
    	\rightarrow
    	\mathrm{End}(H^0(\mathrm{Ig}_{\infty,n}^\mu/\widetilde{P}_L,
    	\widetilde{\mathcal{R}}_\infty^\mathrm{Ig}[\lambda^{-1}])),
    	\,
    	\epsilon
    	\mapsto
    	\mathbb{T}_\epsilon;
    	\]
    	\[
    	\mathbb{Z}[\widetilde{T}_{\widetilde{P}}^+(E_\mathfrak{p})]
    	\rightarrow
    	\mathrm{End}
    	(H^0(\mathrm{Ig}_{\mathrm{rig},n}^\mu/\widetilde{P}_L,
    	\mathcal{R}_\mathrm{rig}^\mathrm{Ig}
    	[\lambda^{-1}])),
    	\,
    	\epsilon
    	\mapsto
    	\mathbb{T}_\epsilon.
    	\]
    \end{proposition}

    \begin{corollary}
    	For any
    	$\lambda\in
    	X^\ast(\widetilde{T}_{\widetilde{P}})^+$,
    	$\epsilon
    	\in
    	\widetilde{T}_{\widetilde{P}}^+(E_\mathfrak{p})$
    	and
    	$n\geq2$,
    	we have
    	\[
    	\mathbb{T}_\epsilon
    	\big(
    	H^0(\mathrm{Ig}_{\mathrm{rig},n}^\mu/\widetilde{P}_L,
    	\mathcal{R}_\mathrm{rig}^\mathrm{Ig}[\lambda^{-1}])
    	\big)
    	\subset
    	H^0(\mathrm{Ig}_{\mathrm{rig},n-1}^\mu/\widetilde{P}_L,
    	\mathcal{R}_\mathrm{rig}^\mathrm{Ig}[\lambda^{-1}]).
    	\]
    \end{corollary}
    
    Recall we have a projection
    $\mathrm{pr}
    \colon
    \widetilde{\mathrm{Ig}}_\infty^{\widetilde{P}_L}
    \rightarrow
    \mathrm{Ig}_\infty/
    \widetilde{P}_L(\mathcal{O}_\mathfrak{p})$.
    We write
    \[
    \Omega[\lambda^{-1}]
    :=
    \mathrm{pr}_\ast
    (\mathcal{O}_{\widetilde{\mathrm{Ig}}_\infty^{\widetilde{P}_L}}
    [\lambda^{-1}]).
    \]
    Now we give
    \begin{definition}
    	Fix a (finite) set of generators
    	$\epsilon_1,\cdots,\epsilon_r$
    	of the monoid
    	$\widetilde{T}_{\widetilde{P}}^+(E_\mathfrak{p})$,
    	we define
    	the $\widetilde{P}$-ordinary projector:
    	\[
    	e_{\widetilde{P}}
    	:=
    	\lim\limits_{n\rightarrow\infty}
    	\left(
    	\prod_{i=1}^r\mathbb{T}_{\epsilon_i}
    	\right)^{n!}
    	\in
    	\mathrm{End}_{\mathcal{O}_\mathfrak{p}}
    	(
    	H^0(\mathrm{Ig}_{\infty}^{L,\mu}/\widetilde{P}_L,
    	\Omega[\lambda^{-1}])
    	).
    	\]
    	
    	Using the isomorphism in
    	\ref{fundamental isom}
    	and the Koecher principal
    	(by Definition 
    	\ref{Shimura varieties, definition}, Condition 5
    	and
    	\cite[Remark 10.3]{Lan2016}),
    	we write
    	\[
    	e_{\widetilde{P}}
    	\mathbb{V}_\infty^{\widetilde{P}_L^\mathrm{der}}[\lambda^{-1}]
    	\]
    	for the space of
    	$\widetilde{P}$-ordinary $p$-adic modular forms
    	of weight $\lambda$, of level
    	$K$.
    	Using the Hodge-Tate map
    	$\mathrm{HT}_\infty^\ast$,
    	we write
    	\[
    	e_{\widetilde{P}}
    	H^0(Sh^\Sigma,\mathcal{V}_\lambda^\Sigma),
    	\text{resp., }
    	e_{\widetilde{P}}
    	H^0(Sh^\Sigma,\mathcal{V}_\lambda^\Sigma(-C^\Sigma))
    	\]
    	for the images of
    	$H^0(Sh^\Sigma,\mathcal{V}_\lambda^\Sigma)$,
    	resp.,
    	$H^0(Sh^\Sigma,
    	\mathcal{V}_\lambda^\Sigma(-C^\Sigma))$
    	in
    	$e_{\widetilde{P}}
    	\mathbb{V}^{\widetilde{P}_L^\mathrm{der}}_\infty[\lambda^{-1}]$.
    \end{definition}

    Now we can globalize the result in
    Proposition
    \ref{ordinary algebraic form}
    as follows
    \begin{proposition}\label{comparison of forms on Igusa and p-adic}
    	For any
    	$\lambda\in X^\ast_\mathrm{dm}(\widetilde{T}_{\widetilde{P}})$,
    	we have a commutative diagram
    	with horizontal isomorphisms
    	\[
    	\begin{tikzcd}
    	e_{\widetilde{P}}
    	H^0(\mathrm{Ig}_{\infty,1}^\mu/\widetilde{P}_L,
    	\widetilde{\mathcal{R}}_\infty^\mathrm{Ig}
    	[\lambda^{-1}])
    	\arrow[r,"\simeq"]
    	\arrow[d,hook]
    	&
    	e_{\widetilde{P}}
    	\mathbb{V}^{\widetilde{P}_L^\mathrm{der}}_\infty[\lambda^{-1}]
    	\arrow[d,hook]
    	\\
    	e_{\widetilde{P}}
    	H^0
    	(\mathrm{Ig}_{\mathrm{rig},1}^\mu/\widetilde{P}_L,
    	\mathcal{R}_\mathrm{rig}^\mathrm{Ig}
    	[\lambda^{-1}])
    	\arrow[r,"\simeq"]
    	&
    	e_{\widetilde{P}}
    	\mathbb{V}^{\widetilde{P}_L^\mathrm{der}}_\infty[\lambda^{-1}]
    	[\frac{1}{p}]
    	\end{tikzcd}
    	\]
    \end{proposition}
    \begin{proof}
    	Using the preceding corollary,
    	the bottom isomorphism comes from the
    	isomorphism
    	\[
    	e_{\widetilde{P}}
    	H^0
    	(\mathrm{Ig}_{\mathrm{rig},\infty}^\mu/\widetilde{P}_L,
    	\mathcal{R}_\mathrm{rig}^\mathrm{Ig}
    	[\lambda^{-1}])
    	\simeq
    	e_{\widetilde{P}}
    	\mathbb{V}_\infty^{\widetilde{P}_L^\mathrm{der}}[\lambda^{-1}]
    	[\frac{1}{p}],
    	\]
    	which again comes from the $p$-integral version
    	by considering stalks of
    	the corresponding sheaves on
    	$\mathrm{Ig}_\infty^\mu/\widetilde{P}_L$:
    	\[
    	e_{\widetilde{P}}
    	H^0(\mathrm{Ig}_{\infty}^\mu/\widetilde{P}_L,
    	\widetilde{\mathcal{R}}_\infty^\mathrm{Ig}
    	[\lambda^{-1}])
    	\simeq
    	e_{\widetilde{P}}
    	\mathbb{V}_\infty^{\widetilde{P}_L^\mathrm{der}}[\lambda^{-1}].
    	\]
    	To show this last isomorphism,
    	again for a point
    	$x\in\mathrm{Ig}_\infty^\mu/\widetilde{P}_L$,
    	consider the stalk
    	$\widetilde{\mathcal{R}}_\infty^\mathrm{Ig}
    	[\lambda^{-1}]_x$
    	of the sheaf
    	$\widetilde{\mathcal{R}}_\infty[\lambda^{-1}]$
    	at the point $x$.
    	Now for any affine open subset
    	$\mathrm{Spf}(A)$
    	containing $x$,
    	$\Omega_\infty[\lambda^{-1}](A)$
    	is given exactly by the module
    	$A[\lambda^{-1}]$.
    	On the other hand,
    	$\widetilde{\mathcal{R}}_\infty^\mathrm{Ig}
    	[\lambda^{-1}](A)$
    	is just the modified representation
    	$\widetilde{R}_A[\lambda^{-1}]$.
    	Thus by Proposition \ref{ordinary algebraic form},
    	we have an isomorphism
    	$e_{\widetilde{P}}\widetilde{R}_A[\lambda^{-1}]
    	\simeq
    	A[\lambda^{-1}]
    	$.
    	Write
    	$\widetilde{\mathcal{R}}_\infty^\mathrm{Ig}
    	[\lambda^{-1}]^0$
    	for the sheaf on
    	$\mathrm{Ig}_\infty^\mu/\widetilde{P}_L$
    	which is given on affine open subset
    	$\mathrm{Spf}(A)$
    	by the kernel
    	$\mathrm{Ker}
    	(\widetilde{R}_A[\lambda^{-1}]
    	\rightarrow
    	A[\lambda^{-1}])$.
    	Then it suffices to show that
    	$e_{\widetilde{P}}
    	((\widetilde{\mathcal{R}}_\infty^\mathrm{Ig}
    	[\lambda^{-1}]^0)_x)=0$
    	for any point $x$.
    	Clearly the operator
    	$e_{\widetilde{P}}$
    	commutes with the inductive limit
    	and this last one follows from the fact that
    	$e_{\widetilde{P}}
    	\mathrm{Ker}
    	(\widetilde{R}_A[\lambda^{-1}]
    	\rightarrow
    	A[\lambda^{-1}])=0$.
    	The top isomorphism in the proposition
    	follows also from this isomorphism.
    \end{proof}

    We want to refine the top isomorphism in
    the above proposition:
    more precisely,
    it is desirable to replace
    $\mathrm{Ig}_{\infty,1}^\mu/\widetilde{P}_L$
    by the $\mu$-ordinary locus
    $Sh^{\Sigma,\mu}_\infty$.
    We set
    \[
    X^\ast_\mathrm{dd}(\widetilde{T}_{\widetilde{P}})
    :=
    X^\ast_\mathrm{dd}(T)\bigcap
    X^\ast_\mathrm{dm}(\widetilde{T}_{\widetilde{P}}).
    \]
    
    \begin{proposition}
    	For any
    	$\lambda\in
    	X^\ast_\mathrm{dd}(\widetilde{T}_{\widetilde{P}})$,
    	the following natural map is an isomorphism:
    	\[
    	e_{\widetilde{P}}
    	H^0
    	(Sh^{\Sigma,\mu}_\infty,
    	\mathcal{R}_\infty[\lambda^{-1}])
    	\rightarrow
    	e_{\widetilde{P}}
    	H^0
    	(\mathrm{Ig}_{\infty,1}^\mu/\widetilde{P}_L,
    	\mathcal{R}_\infty^\mathrm{Ig}[\lambda^{-1}]).
    	\]
    \end{proposition}
    \begin{proof}
    	We follow the strategy in
    	\cite[§5.2.3]{Pilloni2012}.
    	The idea is to construct a new Hecke operator
    	on a modified correspondence such that
    	this new Hecke operator takes
    	functions on
    	$\mathrm{Ig}_{\infty,1}^\mu/\widetilde{P}_L$
    	to functions on
    	$Sh^\mu_\infty$,
    	\textit{i.e.},the image
    	of this Hecke operator
    	are functions invariant under the action
    	$G(\mathcal{O}_\mathfrak{p}/\mathfrak{p})$.
    	More precisely, we consider the following
    	correspondence
    	$\mathrm{Ig}_{\infty,n}^{\mathrm{GSp}(V,\psi),
    		\triangle}/\widetilde{P}_L(\epsilon)$
    	over
    	$Sh(V)^\infty_m$,
    	whose $A$-points
    	($A$ is a $\mathbb{Z}_p$-algebra) consists
    	of quintuples
    	$(\mathcal{A},\widetilde{\mathcal{A}},
    	\pi,\psi,\widetilde{\psi}_n)$
    	where
    	\begin{enumerate}
    		\item 
    		$\mathcal{A}$ is a $\mu$-ordinary
    		principally polarized abelian scheme
    		over $\mathrm{Spec}(A)$
    		and a
    		$\widetilde{P}_V(\mathcal{O}_\mathfrak{p}
    		/\mathfrak{p}^n)$-coset of isomorphisms
    		$\psi_n
    		\colon
    		\mathcal{A}_{x_0}[p^n]
    		\simeq
    		\mathcal{A}[p^n]$;
    		
    		\item
    		$\widetilde{\mathcal{A}}$
    		is a $\mu$-ordinary principally polarized
    		abelian scheme over
    		$\mathrm{Spec}(A)$
    		and a
    		$\widetilde{P}_V(\mathcal{O}_\mathfrak{p}
    		/\mathfrak{p}^n)$-coset of
    		isomorphisms
    		$\widetilde{\psi}_n
    		\colon
    		\mathcal{A}_{x_0}[p^n]
    		\simeq
    		\widetilde{\mathcal{A}}[p^n]$;
    		
    		\item 
    		$\pi\colon
    		\mathcal{A}
    		\rightarrow
    		\widetilde{\mathcal{A}}$
    		is a $p$-isogeny of similitude factor
    		$\nu(\epsilon)$
    		of abelian schemes such that
    		there are isomorphisms
    		$\psi_\infty
    		\colon
    		\mathcal{A}_{x_0}[p^\infty]
    		\simeq
    		\mathcal{A}[p^\infty]$
    		and
    		$\widetilde{\psi}_\infty\colon
    		\mathcal{A}_{x_0}[p^\infty]
    		\simeq
    		\widetilde{\mathcal{A}}[p^\infty]$
    		with
    		$\widetilde{\psi}_\infty
    		\equiv
    		\widetilde{\psi}_n
    		(\mathrm{mod}\,p^n)$ and
    		on the level of
    		Dieudonn\'{e} crystals,
    		we have
    		$\mathbb{D}((\widetilde{\psi}_\infty)^{-1}
    		\circ
    		\pi\circ
    		\psi_\infty)(\mathbb{W})
    		=\epsilon^{-1}
    		\colon
    		\mathbb{D}(\mathcal{A}_{x_0}[p^\infty])
    		(\mathbb{W})
    		\rightarrow
    		\mathbb{D}(\mathcal{A}_{x_0}[p^\infty])
    		(\mathbb{W})$.
    	\end{enumerate}
        Then as before,
        we write
        $\mathrm{Ig}_{\infty,n}^{\mu,\triangle}/
        \widetilde{P}_L(\epsilon)$
        to be the pull-back of
        $\mathrm{Ig}_{\infty,n}^{\mathrm{GSp}(V,\psi)}
        /\widetilde{P}_L(\epsilon)$
        along the embedding
        $Sh^\mu_\infty
        \rightarrow
        Sh(V)^\mu_\infty$.
        We then write
        \[
        \mathrm{pr}_1^\triangle,\mathrm{pr}_2^\triangle
        \colon
        \mathrm{Ig}^{\mu,\triangle}_{\infty,n}
        /\widetilde{P}_L(\epsilon)
        \rightarrow
        \mathrm{Ig}_{\infty,n}^\mu/\widetilde{P}_L
        \]
        for the two natural projections
        which takes
        $(\mathcal{A},\widetilde{\mathcal{A}},\pi,
        \psi_n,\widetilde{\psi}_n)$
        to
        $(\mathcal{A},\psi_n)$
        resp.,
        $(\widetilde{\mathcal{A}},
        \widetilde{\psi}_n)$.

        Note that in the definition above,
        the isomorphism
        $\psi_\infty$
        may not be congruent to
        $\psi_n$ modulo $p^n$.
        Write
        $\mathcal{W}_{\widetilde{P}_V}$
        to be the Weyl group scheme
        of
        $\widetilde{P}_V$
        (over
        $\mathrm{Spec}(\mathcal{O}_\mathfrak{p})$),
        which is finite \'{e}tale over
        $\mathrm{Spec}(\mathcal{O}_\mathfrak{p})$
        (\cite[Proposition 3.2.8]{Conrad2011}).
        We know that the set of
        $\widetilde{P}_V(\mathcal{O}_\mathfrak{p}/
        \mathfrak{p}^n)$-cosets of isomorphisms
        $\psi_n$
        is parameterized by the Weyl group
        $\mathcal{W}_{\widetilde{P}_V}
        	(\mathcal{O}_\mathfrak{p}/\mathfrak{p}^n)$.
        We claim that
        there is an isomorphism
        $\psi_\infty^0
        \colon
        \mathcal{A}_{x_0}[p^\infty]
        \rightarrow
        \mathcal{A}[p^\infty]$
        such that
        $\psi_\infty^0
        \equiv
        \psi_n
        (\mathrm{mod}\,p^n)$
        and
        $\mathbb{D}((\widetilde{\psi}_\infty)^{-1}
        \circ
        \pi
        \circ
        \psi_\infty^0)
        =\epsilon^{-1}\circ w$
        for some
        $w_n\in\mathcal{W}_{\widetilde{P}_V}
        	(\mathcal{O}_\mathfrak{p}/\mathfrak{p}^n)$.
        Indeed,
        we know that there exists some
        $w_\infty
        \in\mathcal{W}_{\widetilde{P}}(\mathcal{O}_\mathfrak{p})$
        such that
        $\psi_\infty\circ w_\infty
        \equiv
        \psi_n(\mathrm{mod}\,p)$.
        Thus we put
        $\psi_\infty^0
        =\psi_\infty\circ w_\infty$,
        then one gets the desired claim
        by setting $w_n$ to be the image of
        $w_\infty$ in
        $\mathcal{W}_{\widetilde{P}}
        (\mathcal{O}_\mathfrak{p}/\mathfrak{p}^n)$.
        Moreover,
        it is clear that
        $w_\infty\in\mathcal{W}_{\widetilde{L}}$
        if and only if the point
        $x\in
        \mathrm{Ig}_{\infty,n}^{\mu,\triangle}/
        \widetilde{P}_L(\epsilon)$
        corresponding to the above quintuple lies in
        $\mathrm{Ig}_{\infty,n}^\mu
        /\widetilde{P}_L(\epsilon)$.

        Write
        $(\mathcal{A}^{G,\triangle},
        \widetilde{\mathcal{A}}^{G,\triangle},
        \pi^{G,\triangle,\epsilon},
        \psi_n^{G,\triangle},
        \widetilde{\psi}_n^{G,\triangle})$
        for the universal quintuple over
        $\mathrm{Ig}^{\mu,\triangle}_{\infty,n}
        /\widetilde{P}_L$.
        As in the case of
        $\mathrm{Ig}_{\infty,n}^\mu
        /\widetilde{P}_L(\epsilon)$,
        one can define a morphism
        \[
        (\pi^{G,\triangle,\epsilon})^\ast
        \colon
        (\mathrm{pr}_{2,\mathrm{rig}}^\triangle)^\ast
        (\mathcal{R}_\mathrm{rig}^\mathrm{Ig}[\lambda^{-1}])
        \rightarrow
        (\mathrm{pr}_{1,\mathrm{rig}}^\triangle)^\ast
        (\mathcal{R}_\mathrm{rig}^\mathrm{Ig}[\lambda^{-1}]),
        \quad
        F
        \mapsto
        (\mathbb{D}(f)
        \mapsto
        F(\mathbb{D}(\widetilde{f}))).
        \]
        
        For a finite flat $\mathcal{O}_\mathfrak{p}$-algebra
        $A$,
        we define a modified Hecke operator
        $\mathbb{T}_\epsilon^\triangle$
        on
        $R_{A[\frac{1}{p}]}[\lambda^{-1}]$
        by sending
        an element $F$ to
        $(\mathbb{T}_\epsilon^\triangle F)(g)
        =
        F(\epsilon w_\infty g\epsilon^{-1})
        $
        for any $g\in G$.
        We claim that
        for
        $F\in
        e_{\widetilde{P}}
        R_A[\lambda^{-1}]$,
        $w_\infty\notin\mathcal{W}_{\widetilde{L}}$,
        $\epsilon_{\widetilde{P}}
        =\prod_{i=1}^r\epsilon_i$
        and
        $\lambda\in
        X^\ast_\mathrm{dd}(\widetilde{T}_{\widetilde{P}})$,
        we have
        $\mathbb{T}_{\epsilon_{\widetilde{P}}}^\triangle F\in
        p\widetilde{R}_A[\lambda^{-1}]$.
        Since $F$ is supported on $\widetilde{P}$,
        it suffices to show that
        $(\mathbb{T}_{\epsilon_{\widetilde{P}}}^\triangle F)(g)$
        is divisible by $p$ for
        $g\in\widetilde{P}$.
        In this case,
        we have
        \[
        (\mathbb{T}_{\epsilon_{\widetilde{P}}}^\triangle F)(g)
        =
        F(\epsilon_{\widetilde{P}}w_\infty
        g\epsilon_{\widetilde{P}}^{-1})
        =
        F(\epsilon_{\widetilde{P}}w_\infty\epsilon_{\widetilde{P}}^{-1}
        \cdot
        \epsilon_{\widetilde{P}}g\epsilon_{\widetilde{P}}^{-1})
        =
        F(\epsilon_{\widetilde{P}}w_\infty e_{\widetilde{P}}^{-1})
        =
        \lambda^{-1}(w_\infty^{-1}\epsilon_{\widetilde{P}}
        w_\infty\epsilon_{\widetilde{P}}^{-1})
        F(1).
        \]
        Note that
        $\lambda^{-1}(w_\infty^{-1}\epsilon_{\widetilde{P}}
        w_\infty\epsilon_{\widetilde{P}}^{-1})
        \in
        pA$,
        we deduce that
        $(\mathbb{T}_{\epsilon_{\widetilde{P}}}^\triangle F)(g)$
        is indeed
        divisible by $p$ for any $g$.

        Next we can globalize the above argument.
        The same reasoning as in
        Proposition \ref{local-global commutativity}
        shows that
        for any affine open subset
        $\mathrm{Spf}(A)$ of
        $\mathrm{Ig}_{\mathrm{rig},n}^{\mu,\triangle}/
        \widetilde{P}_L$,
        we have the following commutative diagram
        \[
        \begin{tikzcd}
        R_A[\lambda^{-1}]
        \arrow[r]
        \arrow[d,hook]
        &
        (\mathrm{pr}_{2,\mathrm{rig}}^{\triangle})^\ast
        (\mathcal{R}_\mathrm{rig}^{\mathrm{Ig}}
        [\lambda^{-1}])
        (A[\frac{1}{p}])
        \arrow[d]
        \\
        R_{A[\frac{1}{p}]}[\lambda^{-1}]
        \arrow[r]
        \arrow[d,"\mathbb{T}_\epsilon^\triangle"]
        &
        (\mathrm{pr}_{2,\mathrm{rig}}^{\triangle})^\ast
        (\mathcal{R}_\mathrm{rig}^{\mathrm{Ig}}
        [\lambda^{-1}])
        (A[\frac{1}{p}])
        \arrow[d,"(\pi^{G,\triangle,\epsilon})^\ast"]
        \\
        R_{A[\frac{1}{p}]}[\lambda^{-1}]
        \arrow[r]
        &
        (\mathrm{pr}_{1,\mathrm{rig}}^{\triangle})^\ast
        (\mathcal{R}_\mathrm{rig}^{\mathrm{Ig}}
        [\lambda^{-1}])
        (A[\frac{1}{p}])
        \end{tikzcd}
        \]

        We write
        \[
        \mathrm{Tr}_\epsilon^\triangle
        =
        \frac{1}{m_\epsilon}
        \mathrm{Tr}_{\mathrm{pr}_{1,\mathrm{rig}}^\triangle}.
        \]
        Then for any
        $\mathcal{O}_{\mathrm{Ig}_{\mathrm{rig},n}^{\mu,\triangle}
        /\widetilde{P}_L}$-sheaf $\mathcal{F}$,
        we 
        write
        $\mathbb{T}_\epsilon^\triangle$
        for the composition of morphisms
        (see the remark below):

        \[
        \begin{tikzcd}
        H^0(\mathrm{Ig}_{\mathrm{rig},n}^\mu
        /\widetilde{P}_L,\mathcal{F})
        \arrow[r]
        \arrow[d,dashed]
        &
        H^0(\mathrm{Ig}_{\mathrm{rig},n}^{\mu,\triangle}
        /\widetilde{P}_L(\epsilon),
        (\mathrm{pr}_{2,\mathrm{rig}}^\triangle)^\ast\mathcal{F})
        \arrow[d,"(\pi^{G,\triangle,\epsilon})^\ast"']
        \\
        H^0(\mathrm{Ig}_{\mathrm{rig},n}^\mu
        /\widetilde{P}_L,\mathcal{F})
        &
        H^0(\mathrm{Ig}_{\mathrm{rig},n}^{\mu,\triangle}
        /\widetilde{P}_L(\epsilon),
        (\mathrm{pr}_{1,\mathrm{rig}}^\triangle)^\ast\mathcal{F})
        \arrow[l,"\mathrm{Tr}_\epsilon^\triangle"']
        \end{tikzcd}
        \]

        On one hand,
        for
        $\lambda\in
        X^\ast_\mathrm{dd}(\widetilde{T}_{\widetilde{P}})$
        and
        $F\in
        e_{\widetilde{P}}
        H^0(\mathrm{Ig}_{\infty,n}^\mu/
        \widetilde{P}_L,
        \mathcal{R}^\mathrm{Ig}_\infty[\lambda^{-1}])$,
        it is easy to see that
        $\mathbb{T}_\epsilon F
        -\mathbb{T}_\epsilon^\triangle F$
        is divisible by $p$:
        indeed,
        for any point
        $x\in
        \mathrm{Ig}_{\infty,n}^{\mu,\triangle}/
        \widetilde{P}_L(\epsilon)$,
        if this point is in
        $\mathrm{Ig}_{\infty,n}^\mu/
        \widetilde{P}_L(\epsilon)$,
        then
        $\mathbb{T}_\epsilon$ and
        $\mathbb{T}_\epsilon^\triangle$ coincide
        and the difference is $0$;
        otherwise,
        by definition
        $(\mathbb{T}_\epsilon F)_x=0$
        while by the above argument
        $(\mathbb{T}_\epsilon^\triangle F)_x$ is divisible by
        $p$.
        Thus one deduces that
        \[
        \mathbb{T}_\epsilon F
        -
        \mathbb{T}_\epsilon^\triangle F
        \in
        pH^0(\mathrm{Ig}_{\infty,n}^\mu/
        \widetilde{P}_L,
        \widetilde{\mathcal{R}}_\infty^\mathrm{Ig}[\lambda^{-1}])
        \]

        On the other hand,
        by the definition of the correspondence
        $\mathrm{Ig}_{\infty,n}^{\mu,\triangle}/
        \widetilde{P}_L(\epsilon)$,
        one sees that
        $\mathbb{T}_\epsilon^\triangle$
        sends a section
        $F\in
        H^0(\mathrm{Ig}_{\infty,n}^\mu/
        \widetilde{P}_L,
        \mathcal{R}_\infty^\mathrm{Ig}[\lambda^{-1}])$
        to a section
        that is in fact invariant under
        the whole
        $G(\mathcal{O}_\mathfrak{p}/\mathfrak{p}^n)$,
        and thus
        $\mathbb{T}_\epsilon^\triangle F$
        is a section in
        $H^0
        (Sh_\infty^\mu,
        \mathcal{R}_\infty[\lambda^{-1}])$.

        Now using the fact that
        $\mathbb{T}_\epsilon$ is invertible on
        $e_{\widetilde{P}}
        H^0
        (\mathrm{Ig}_{\infty,n}^\mu
        /\widetilde{P}_L,
        \mathcal{R}_\infty^\mathrm{Ig}[\lambda^{-1}])$
        and Nakayama's Lemma,
        we see that the natural inclusion
        \[
        e_{\widetilde{P}}
        H^0
        (Sh_\infty^{\Sigma,\mu},\mathcal{R}_\infty[\lambda^{-1}])
        \hookrightarrow
        e_{\widetilde{P}}
        H^0
        (\mathrm{Ig}_{\infty,n}^\mu/
        \widetilde{P}_L,
        \mathcal{R}_\infty^\mathrm{Ig}[\lambda^{-1}])
        \]
        is indeed an isomorphism.       
    \end{proof}
    \begin{remark}
    	Here we have implicitly used
    	Koecher's principle
    	(which is valid by our assumption
    	Definition \ref{Shimura varieties, definition}
    	and
    	\cite{Lan2016})
    	in the definition of
    	$\mathbb{T}_\epsilon^\triangle$:
    	indeed,
    	write
    	$\widetilde{\mathrm{Ig}}_\infty^{\mu}/\widetilde{P}_L$
    	for
    	the pull-back of
    	$Sh_{\infty}^\mu
    	\rightarrow
    	Sh_\infty^{\Sigma,\mu}$
    	along the projections
    	$\mathrm{Ig}_\infty^\mu/\widetilde{P}_L$.
    	Then    	
    	\textit{a priori},
    	$\mathbb{T}_\epsilon^\triangle$
    	is defined only on
    	$H^0
    	(\widetilde{\mathrm{Ig}}_\infty^\mu
    	/\widetilde{P}_L,
    	\mathcal{F})$.
    	Then we apply
    	Koecher's principle
    	to extend to
    	$H^0(\mathrm{Ig}_\infty^\mu/\widetilde{P}_L)$.
    \end{remark}
    Applying Corollary \ref{ordinary modified=ordinary},
    we obtain
    \begin{corollary}\label{descent from Igusa to Sh}
    	For any
    	$\lambda
    	\in
    	X_\mathrm{dd}^\ast(\widetilde{T}_{\widetilde{P}})$,
    	we have an isomorphism
    	\[
    	e_{\widetilde{P}}
    	H^0
    	(Sh^{\Sigma,\mu}_\infty,
    	\mathcal{R}_\infty[\lambda^{-1}])
    	\rightarrow
    	e_{\widetilde{P}}
    	H^0
    	(\mathrm{Ig}_{\infty,1}^\mu/\widetilde{P}_L,
    	\widetilde{\mathcal{R}}_\infty^\mathrm{Ig}
    	[\lambda^{-1}]).
    	\]
    \end{corollary}

    \section{Lifting modular forms}
    In this section,
    we will deal with the problem of
    lifting modular forms with values in
    $A/pA$ to modular forms with values in $A$
    for a finite flat extension $A$ of
    $\mathcal{O}_\mathfrak{p}$.

    \subsection{Review of toroidal compactification}
    We first review some results in
    \cite{Madapusi2012}
    and also introduce some notations.

    Let
    $(G,X)$ be our Shimura data of Hodge type
    and decompose the adjoint group
    $G^\mathrm{ad}=
    \prod_{i=1}^mG_i$
    into simple factors.
    Then a parabolic subgroup
    $P$ of $G$ is said to be admissible
    if $P_i:=
    \mathrm{Im}(P\hookrightarrow G\rightarrow G^\mathrm{ad}
    \rightarrow G_i)$
    is a proper maximal parabolic subgroup of $G_i$ or
    the whole $G_i$
    for all $i=1,\cdots,m$.
    Such an admissible parabolic subgroup corresponds to a
    rational boundary component for
    $(G,X)$
    (\cite[§2.1.3]{Madapusi2012}).
    Recall $U_P$ denotes the unipotent radical of $P$,
    we then write
    $W_P=Z(U_P)$
    for the centre of $U_P$.
    Denote by $Q_P$ the minimal
    normal subgroup pf $P$ such that
    the morphism
    $h_x'
    \colon
    \mathbb{S}
    \rightarrow
    G_\mathbb{R}$
    factors through
    $(Q_P)_\mathbb{R}$
    (here $x\in X$).
    Then
    $Q_P(\mathbb{R})W_P(\mathbb{C})$
    acts on
    the set $\pi_0(X)$
    of connected components of $X$
    via the natural map
    $\pi_0
    (Q_P(\mathbb{R})W_P(\mathbb{C}))
    \rightarrow
    \pi_0(Q_P(\mathbb{R}))
    \rightarrow
    \pi_0(G(\mathbb{R}))$.
    For a connected component $X^+$ of $X$,
    we then write
    $F_{P,X^+}$
    for the
    $Q_P(\mathbb{R})W_P(\mathbb{C})$-orbit
    of
    $(X^+,h_x')$
    inside
    $\pi_0(X)\times
    \mathrm{Hom}(\mathbb{S}_\mathbb{C},
    (Q_P)_\mathbb{C})$
    (\cite[2.1.5]{Madapusi2012}).
    A cusp label representative
    (clr) for
    $(G,X)$ is a triple
    $\Phi=(P_\Phi,X_\Phi^+,g_\Phi)$
    where
    $P_\Phi$ is an admissible parabolic subgroup of $G$,
    $X_\Phi^+$ is a connected component of $X$
    and
    $g_\Phi$ is an element in
    $G(\mathbb{A}_\mathrm{f})$
    (\cite[§2.1.7]{Madapusi2012}).
    We put
    $Q_\Phi=P_\Phi$,
    $U_\Phi=U_{P_\Phi}$,
    $W_\Phi=W_{P_\Phi}$,
    $\overline{Q}_\Phi
    =Q_\Phi/W_\Phi$,
    $V_\Phi:=
    U_\Phi/W_\Phi$
    the unipotent radical of $\overline{Q}_\Phi$,
    $G_{\Phi,h}=Q_\Phi/U_\Phi$
    the Levi component of
    $Q_\Phi$,
    $D_\Phi:=
    F_{P_\Phi,X_\Phi^+}$,
    $\overline{D}_\Phi
    =W_\Phi(\mathbb{C})\backslash D_\Phi$,
    $D_{\Phi,h}=
    V_\Phi(\mathbb{R})\backslash\overline{D}_\Phi$.
    Then
    $(G_{\Phi,h},D_{\Phi,h})$
    is a pure Shimura datum with reflex field
    $E$
    and
    $(Q_\Phi,D_\Phi)$,
    $(\overline{Q}_\Phi,\overline{D}_\Phi)$
    mixed Shimura data with reflex field
    $E$
    (\cite[2.1.7]{Madapusi2012}).
    Given the Shimura variety
    $Sh_K(G,X)$,
    we write
    $Sh_{K_\Phi}(Q_\Phi,D_\Phi)$
    for the (complex) Shimura variety associated to
    $(Q_\Phi,D_\Phi,K_\Phi)$
    where
    $K_\Phi=Q_\Phi(\mathbb{A}_\mathrm{f})
    \cap K_{\Phi,P}$
    and
    $K_{\Phi,P}
    =P_{\Phi}(\mathbb{A}_\mathrm{f})
    \cap
    g_\Phi Kg_\Phi^{-1}$.
    Write
    $K_{\Phi,U}$
    to be the subset of
    $U_\Phi(\mathbb{A}_\mathrm{f})$
    which is the image of the set
    $\{(z,u)\in
    Z_\Phi(\mathbb{Q})\times
    U_\Phi(\mathbb{A}_\mathrm{f})|
    zu\in K_\Phi\}$
    under the canonical projection to the second factor
    $U_\Phi(\mathbb{A}_\mathrm{f})$.
    Then let
    $K_{\Phi,V}\subset
    V_{P_\Phi}(\mathbb{A}_\mathrm{f})$
    denote the image of
    $K_{\Phi,U}$ under the natural quotient map.
    Thus we obtain a canonical smooth family of
    abelian varieties
    (\cite[§2.1.10]{Madapusi2012})
    \[
    \mathcal{A}_K(\Phi)(\mathbb{C})
    \rightarrow
    Sh_{K_{\Phi,h}}
    (G_{\Phi,h},D_{\Phi,h})(\mathbb{C}).
    \]

    Concerning the interaction between our Shimura variety
    of Hodge type and the Siegel Shimura variety,
    We change a little bit of notations and
    write
    \[
    \iota
    \colon
    (G,X)
    \rightarrow
    (\mathrm{GSp}(V,\psi),S^\pm)
    =:
    (G^\ddagger,X^\ddagger)
    \]
    for the embedding of our Shimura data into the
    Siegel Shimura data fixed from the beginning.
    We fix also an element
    $g_V$
    in
    $G^\ddagger(\mathbb{A}_\mathrm{f})$.
    Then the clr
    $\Phi$ for $(G,X)$
    induces a clr 
    $\Phi^\ddagger
    =(P_{\Phi^\ddagger},X_{\Phi^\ddagger}^+),g_{\Phi^\ddagger}$
    for
    $(G^\ddagger,X^\ddagger)$:
    $P_{\Phi^\ddagger}
    =\iota_\ast(P_\Phi)$,
    $X_{\Phi^\ddagger}^+$
    is the unique connected component of
    $X^\ddagger$ containing
    $\iota(X_\Phi)$
    and
    $g_{\Phi^\ddagger}=g_\Phi g^\ddagger$
    (\cite[2.1.28]{Madapusi2012}).

	Now we continue our discussion.
	We write
	$\mathbf{B}_K(\Phi):=
	(W_\Phi(\mathbb{Q})\cap K_{\Phi,W})(-1)$
	and let
	$\mathbf{E}_K(\Phi)$
	denote the torus split over
	$\mathbb{Z}$ with character group
	$\mathbf{S}_K(\Phi)
	:=\mathbf{B}_K(\Phi)^\vee
	=\mathrm{Hom}_\mathbb{Z}
	(\mathbf{B}_K(\Phi),\mathbb{Z})
	$(\cite[§2.1.11]{Madapusi2012}).
	From this we get a canonical isomorphism of families of
	complex groups
	\[
	P_\Phi(0)(\mathbb{Q})\backslash
	D_\Phi(0)
	\times
	P_\Phi(0)(\mathbb{A}_\mathrm{f})/
	K_\Phi(0)
	\xrightarrow{\sim}
	\mathbf{E}_K(\Phi)(\mathbb{C})
	\times
	Sh_{\nu_\Phi(K_{\Phi,h})}
	(\mathbb{G}_m,S^\pm(0))(\mathbb{C}).
	\]
	Here
	we fix a surjective map of Shimura data
	$\nu_\Phi
	\colon
	(G_{\Phi,h},D_{\Phi,h})
	\rightarrow
	(\mathbb{G}_m,S^\pm(0))$,
	$P_\Phi(0)
	:=
	W_\Phi\rtimes\mathbb{G}_m$,
	$D_\Phi(0)
	:=
	\{
	(\omega,\lambda)\in
	\mathrm{Hom}(\mathbb{S}_\mathbb{C},
	P_\Phi(0)_\mathbb{C}\times S^\pm(0))|
	(\pi\circ h)(x,y)=xy
	\}$,
	and
	$K_\Phi(0)
	:=
	K_{\Phi,W}\rtimes
	\nu_\Phi(K_{\Phi,h})
	\subset
	P_\Phi(0)(\mathbb{A}_\mathrm{f})$.

	The relation among different choices of
	clr $\Phi$ for
	$(G,X)$
	is as given
	(\textit{cf.}\cite[§4.1.3]{Madapusi2012}).
	To state the relevant results,
	we need first introduce some more notations.
	Recall that for a finite dimensional vector space
	$\mathbf{V}$ over $\mathbb{Q}$,
	a rational polyhedral cone
	(rpc)
	$\sigma\subset
	\mathbf{V}\otimes_\mathbb{Q}\mathbb{R}$
	is a subset given by
	$\sigma
	=
	\{
	v\in\mathbf{V}\otimes\mathbb{R}|
	f_i(v)\geq0,\forall i=1,2,\cdots,r
	\}$
	for a finite set
	$f_1,f_2,\cdots,f_r$
	of dual vectors in
	$\mathbf{V}^\vee$.
	Fix a $\mathbb{Z}$-lattice $\mathbf{X}$ of
	$\mathbf{V}$,
	then $\sigma$ is admissible
	(with respect to $\mathbf{X}$)
	if the above $f_1,\cdots,f_r$
	can be chosen to be part of a basis for
	$\mathbf{X}$.
	On the other hand,for two admissible parabolic subgroups
	$P_1,P_2$ of $G$
	and an element
	$\gamma\in G(\mathbb{Q})$,
	we write
	$P_1\xrightarrow{\gamma}P_2$
	if
	$\gamma Q_{P_1}\gamma^{-1}\subset Q_{P_2}$.
	For two clr
	$\Phi_1,\Phi_2$ for
	$(G,X)$ and elements
	$\gamma\in G(\mathbb{Q})$,
	$q_2\in Q_{\Phi_2}(\mathbb{A}_\mathrm{f})$,
	we write
	$\Phi_1\xrightarrow{(\gamma,q_2)_K}\Phi_2$
	if
	$P_{\Phi_1}
	\xrightarrow{\gamma}
	P_{\Phi_2}$,
	$\gamma\cdot X_{\Phi_1}^+
	\in
	\pi_0(X)$
	lies inside the
	$Q_{\Phi_2}(\mathbb{Q})$-orbit of
	$X_{\Phi_2}^+$
	and moreover
	$\gamma g_{\Phi_1}\in
	q_2g_{\Phi_2}K$.
	Now for a clr
	$\Phi$,
	we write
	$\mathbf{H}^\ast(\Phi)
	\subset
	W_\Phi(\mathbb{R})(-1)$
	to be the union of the images of the cones
	$\mathrm{Int}(\gamma^{-1})(\mathbf{H}(\Phi'))$
	for all
	$
	\Phi'
    \xrightarrow{(\gamma,q)_K}
	\Phi$(
	\cite[§2.1.14]{Madapusi2012}).
	Here
	$\mathbf{H}(\Phi')
	:=\mathbf{H}_{P_{\Phi'},X_{\Phi'}^+}$
	is given in
	\cite[§2.1.6]{Madapusi2012},
	which is an open non-degenerate self-adjoint convex cone
	inside
	$W_{P_{\Phi'}}(\mathbb{R})(-1)$
	corresponding to
	$X_{\Phi'}^+$.
	A rational polyhedral cone decomposition
	(rpcd)
	for $\mathbf{H}^\ast(\Phi)$
	is a set
	$\Sigma(\Phi)$
	of rpc $\sigma\subset W_\Phi(\mathbb{R})(-1)$
	such that
	(1)
	$\sigma\subset \mathbf{H}^\ast(\Phi)$ for all
	$\sigma\in\Sigma(\Phi)$;
	(2)for any $\sigma\in\Sigma(\Phi)$,
	any face of $\sigma$ is also in $\Sigma(\Phi)$;
	(3)for any $\sigma_1,\sigma_2\in\Sigma(\Phi)$,
	$\sigma_1\cap\sigma_2$ is a common face of
	$\sigma_1$ and $\sigma_2$
	(\cite[§2.1.22]{Madapusi2012}).
	An admissible rpcd
	for $(G,X,K)$
	is an assignment to each clr $\Phi$ for
	$(G,X)$
	a rpcd
	$\Sigma(\Phi)$
	for
	$\mathbf{H}^\ast(\Phi)$
	and satisfying compatibility conditions:
	for any
	$\Phi_1\xrightarrow{(\gamma,q_2)_K}\Phi_2$,
	we have
	$\Sigma(\Phi_2)
	=
	\{
	\sigma\subset
	W_{\Phi_1}(\mathbb{Q})(-1)|
	\mathrm{Int}(\gamma^{-1})(\sigma)\in
	\Sigma(\Phi_1)
	\}$.

	We fix an admissible
	rpcd
	(rpcd)
	$\Sigma^\ddagger$
	for
	$(G^\ddagger,X^\ddagger,K^\ddagger)$
	and
	$\Sigma$
	the induced
	rpcd for
	$(G,X,K)$,
	then we have the corresponding morphism of
	toroidal compactifications
	of the integral models of the
	Shimura varieties:
	\[
	Sh^\Sigma
	\rightarrow
	Sh(G^\ddagger,X^\ddagger,K^\ddagger)^{\Sigma^\ddagger},
	\]
	which, by construction,
	is the normalization of
	$Sh$ inside
	$Sh(G^\ddagger,X^\ddagger,K^\ddagger)^{\Sigma^\ddagger}$
	(\cite[§4.1.4]{Madapusi2012}).
	We consider pairs
	$(\Phi,\sigma)$
	where $\Phi$ is a clr and
	$\sigma\in\Sigma^\circ(\Phi)$,
	\textit{i.e.}
	$\sigma^\circ
	\subset
	\mathbf{H}(\Phi)$.
	We say two such pairs
	$(\Phi_1,\sigma_1)$ and
	$(\Phi_2,\sigma_2)$ are equivalent
	if there is
	$\Phi_1\xrightarrow{(\gamma,q_2)_K}
	\Phi_2$ such that
	$\gamma P_{\Phi_1}=P_{\Phi_2}$ and
	$\mathrm{Int}(\gamma)(\sigma_1)=\sigma_2$.
	Write
	$\mathrm{Cusp}^\Sigma_K(G,X)$
	for the set of such equivalence classes
	and there is a natural partial order
	$\preccurlyeq$ on this set:
	two equivalence classes
	$[(\Phi_1,\sigma_1)]
	\preccurlyeq
	[(\Phi_2,\sigma_2)]$
	if
	$(\Phi_1,\sigma_1)
	\xrightarrow{(\gamma,q_2)_K}
	(\Phi_2,\sigma_2)$
	for some $\gamma\in G(\mathbb{Q})$
	and
	$q_2\in Q_{\Phi_2}(\mathbb{A}_\mathrm{f})$
	(\cite[2.1.26]{Madapusi2012}).

	Now for any rpc
	$\sigma^\ddagger
	\subset
	\mathbf{B}_{K^\ddagger}(\Phi^\ddagger)\otimes\mathbb{R}$,
	we have the twisted torus embedding
	(\cite[4.1.2]{Madapusi2012},
	note that our notations are a bit different from
	\textit{loc.cit}):
	\[
	Sh(Q_{\Phi^\ddagger},D_{\Phi^\ddagger},
	K_{\Phi^\ddagger}^\ddagger)
	\hookrightarrow
	Sh(Q_{\Phi^\ddagger},D_{\Phi^\ddagger},
	K_{\Phi^\ddagger}^\ddagger,
	\sigma^\ddagger).
	\]
	Let
	$\sigma=\sigma^\ddagger\cap W_\Phi(\mathbb{R})(-1)$
	and
	$Sh(Q_\Phi,D_\Phi,K_\Phi,\sigma)$
	be the normalization of
	$Sh(Q_{\Phi^\ddagger},D_{\Phi^\ddagger},
	K_{\Phi^\ddagger}^\ddagger,
	\sigma^\ddagger)$
	in
	$Sh(Q_\Phi,D_\Phi,K_\Phi)$.
	Then
	we write
	$\widetilde{Z}
	(Q_\Phi,D_\Phi,K_\Phi,\sigma)$
	for the closed stratum in
	$Sh(Q_\Phi,D_\Phi,K_\Phi,\sigma)
	\otimes_{\mathcal{O}_{E}}E$
	and
	$Z(Q_\Phi,D_\Phi,K_\Phi,\sigma)$
	for the normalization of
	$Z(Q_\Phi,D_\Phi,K_\Phi,\sigma)$ in
	$\widetilde{Z}(Q_\Phi,D_\Phi,K_\Phi,\sigma)$.
	Now we choose
	$\Upsilon
	=
	[(\Phi,\sigma)]
	\in
	\mathrm{Cusp}_K^\Sigma(G,X)$
	and set
	$\Upsilon^\ddagger
	=\iota_\ast\Upsilon
	=
	[(\Phi^\ddagger,\sigma^\ddagger)]$.
	Then we have the following commutative diagram
	(\cite[§4.1.4]{Madapusi2012})
	\[
	\begin{tikzcd}
	\widetilde{Z}
	(Q_\Phi,D_\Phi,K_\Phi,\sigma)
	\arrow[r]
	\arrow[d,"\simeq"]
	&
	\widetilde{Z}
	(Q_{\Phi^\ddagger},D_{\Phi^\ddagger},
	K_{\Phi^\ddagger}^\ddagger,\sigma^\ddagger)
	\arrow[r]
	\arrow[d,"\simeq"]
	&
	Z
	(Q_{\Phi^\ddagger},D_{\Phi^\ddagger},
	K_{\Phi^\ddagger}^\ddagger,\sigma^\ddagger)
	\arrow[d,"\simeq"]
	\\
	\widetilde{Z}_K(\Upsilon)
	\arrow[r]
	&
	\widetilde{Z}_{K^\ddagger}(\Upsilon^\ddagger)
	\arrow[r]
	&
	Z_{K^\ddagger}(\Upsilon)
	\end{tikzcd}
	\]
	We then write
	$Z_K(\Upsilon)$ for the normalization of
	$Z_{K^\ddagger}(\Upsilon^\ddagger)$
	inside
	$\widetilde{Z}_K(\Upsilon)$.
	As such one obtains the structure of the
	toroidal compactification
	$Sh^\Sigma$.
	We refer the reader to
	\cite[Theorem 4.1.5]{Madapusi2012}
	for the precise statements.
	By construction,
	the stratification on
	$Sh^\Sigma$
	is parameterized by
	$\Upsilon\in
	\mathrm{Cusp}_K^\Sigma(G,X)$.

	Next we review the construction of minimal
	compactification
	$Sh^\mathrm{min}$ of
	$Sh$.
	We write
	$\mathrm{Cusp}_K(G,X)$
	for the set of equivalence classes
	of clr for
	$(G,X)$
	for the relation
	$\Phi_1\sim\Phi_2$
	given by some
	$\Phi_1\xrightarrow{(\gamma,q_2)_K}
	\Phi_2$
	such that
	$\gamma\cdot P_{\Phi_1}=P_{\Phi_2}$.
	It also has a natural partial order
	$\preccurlyeq$ given by:
	$[\Phi_1]\preccurlyeq[\Phi_2]$ if
	there exists some
	$\Phi_1
	\xrightarrow{(\gamma,q_2)_K}
	\Phi_2$.
	By \cite[§5.2.11]{Madapusi2012},
	one has
	\[
	Sh^\mathrm{min}
	=
	\bigsqcup_{[\Phi]\in\mathrm{Cusp}_K(G,X)}
	Z_K([\Phi]).
	\]
	
	Recall that the Hodge line bundle
	$\omega_{K^\ddagger}(\Sigma^\ddagger)$
	over
	$Sh(G^\ddagger,X^\ddagger,K^\ddagger)^{\Sigma^\ddagger}$
	is given by
	(\cite[Definition 5.1.1]{Madapusi2012}):
	\[
	\omega_{K^\ddagger}(\Sigma^\ddagger)
	=
	\mathrm{det}
	(
	H_\mathrm{dR}(\Sigma^\ddagger)/F^0H_\mathrm{dR}(\Sigma^\ddagger)
	)^{-1}.
	\]
	We write
	$\omega_K(\Sigma)$
	for the restriction of
	$\omega_{K^\ddagger}(\Sigma^\ddagger)$
	to
	$Sh^\Sigma$
	(independent of the choice of
	$K^\ddagger,\Sigma^\ddagger$).
	We know that
	$\omega_K(\Sigma)$
	is an ample line bundle over
	$Sh^\Sigma$
	and the graded 
	$\mathcal{O}_\mathfrak{p}$-ring
	$\oplus_{n\geq0}
	H^0(Sh^\Sigma,\omega_K(\Sigma)^{\otimes n})$
	is finitely generated
	(\cite[§5.1.3]{Madapusi2012}).
	Let
	$H_\mathrm{dR}(\Phi,\sigma)$
	be the restriction of
	$H_\mathrm{dR}(\Phi^\ddagger,\sigma^\ddagger)$
	to
	$Sh(Q_\Phi,D_\Phi,K_\Phi,\sigma)$,
	then we have
	\[
	\mathrm{det}
	(H_\mathrm{dR}(\Phi,\sigma)/F^0H_\mathrm{dR}(\Phi,\sigma))^{-1}
	\simeq
	\omega_K^{\text{\'{e}t}}(\Phi)
	\otimes_\mathbb{Z}
	\omega_K^{\text{ab}}(\Phi)=
	\omega_K(\Phi).
	\]
	Here
	$\omega_K^\text{ab}(\Phi)$
	is the Hodge line bundle associated with the universal
	abelian scheme
	$\mathcal{A}
	\rightarrow
	Sh(G_{\Phi,h},D_{\Phi,h},K_{\Phi,h})$
	(and also its pull-back to
	$Sh(Q_\Phi,D_\Phi,K_\Phi,\sigma)$)
	and
	$\omega_K^\text{\'{e}t}(\Phi)
	:=
	\mathrm{det}(\mathrm{gr}^WH^g(\mathbb{Z}))^{-1}$
	(\cite[§5.1.4]{Madapusi2012}).
	Moreover,
	since
	$Sh(Q_\Phi,D_\Phi,K_\Phi)$ is an
	$\mathbf{E}_K(\Phi)$-torsor over
	$Sh(\overline{Q}_\Phi,\overline{D}_\Phi,
	\overline{K}_\Phi)$,
	we have
	(\cite[p.32]{Pilloni2012})
	\[
	Sh(Q_\Phi,D_\Phi,K_\Phi)
	=
	\mathrm{Spec}
	(\oplus_{l\in\mathbf{S}_K(\Phi)}
	\Psi_K^{(l)}(\Phi))
	\]
	where
	$\Psi_K^{(l)}(\Phi)$
	is a line bundle over
	$Sh(\overline{Q}_\Phi,\overline{D}_\Phi,
	\overline{K}_\Phi)$.
	We then write the projective morphism
	\[
	\pi'
	\colon
	Sh(\overline{Q}_\Phi,\overline{D}_\Phi,
	\overline{K}_\Phi)
	\rightarrow
	Sh(G_{\Phi,h},D_{\Phi,h},K_{\Phi,h}).
	\]
	Now for each $l\in\mathbf{S}_K(\Phi)$,
	we define a coherent sheaf
	over
	$Sh(G_{\Phi,h},D_{\Phi,h},K_{\Phi,h})$:
	\[
	\mathrm{FJ}_K^{(l)}
	(\Phi)
	:=
	\pi'_\ast
	(\Psi_K^{(l)}(\Phi))
	\]
	For a
	$\Upsilon=[(\Phi,\sigma)]$ and an integer $n\geq0$,
	we have an evaluation map
	$\mathrm{FJ}_{(\Phi,\sigma)}$
	which is the composition of the following morphisms
	(\cite[§5.1.5]{Madapusi2012}):	
	\[
	\begin{tikzcd}
	H^0
	(Sh^\Sigma,\omega_K(\Sigma)^{\otimes n})
	\arrow[r]
	\arrow[rdd,bend right=10,dashed]
	&
	H^0
	(Sh_{Z_K(\Upsilon)}^{\Sigma,\wedge},
	\omega_K(\Phi)^{\otimes n})
	\arrow[d,"\sim"]
	\\
	&
	H^0
	(Sh(Q_\Phi,D_\Phi,K_\Phi,\sigma)^\wedge,
	\omega_K(\Phi)^{\otimes n})
	\arrow[d,hookrightarrow]
	\\
	&
	\prod_{l\in\mathbf{S}_K(\Phi)}
	H^0
	(Sh(G_{\Phi,h},D_{\Phi,h},K_{\Phi,h}),
	\mathrm{FJ}_K^{(l)}(\Phi)\otimes
	\omega_K(\Phi)^{\otimes n}).
	\end{tikzcd}
	\]
	One can show that
	these morphisms
	$\mathrm{FJ}_{(\Phi,\sigma)}$
	are compatible with each other
	(\cite[(5.1.6.2)]{Madapusi2012}).
	
	Now let
	$P_\Phi(\mathbb{Q})_\triangledown
	\subset
	P_\Phi(\mathbb{Q})$
	be the stabilizer in
	$P_\Phi(Q)$ of
	the
	$Q_\Phi(\mathbb{R})$-orbit of
	$X_\Phi^+$.
	We know that there is an element
	$q\in Q_\Phi(\mathbb{A})$
	such that
	$\Phi\xrightarrow{(\gamma,q)_K}\Phi$.
	This induces an isomorphism
	on
	$Sh(Q_\Phi,D_\Phi,K_\Phi)$
	which does not depend on the choice of
	the finite component
	$q_\mathrm{f}$ of $q$.
	Now we define
	\[
	\triangle_K(\Phi):=
	P_\Phi(Q)_\triangledown\bigcap
	(Q_\Phi(\mathbb{A}_\mathrm{f})g_\Phi Kg_\Phi^{-1})/
	Q_\Phi(\mathbb{Q}).
	\]
	Then
	$\triangle_K(\Phi)$ acts on the tower of Shimura varieties
	\[
	Sh(Q_\Phi,D_\Phi,K_\Phi)
	\rightarrow
	Sh(\overline{Q}_\Phi,\overline{D}_\Phi,\overline{K}_\Phi)
	\rightarrow
	Sh(G_{\Phi,h},D_{\Phi,h},K_{\Phi,h}).
	\]
	Moreover,
	the action of
	$\triangle_K(\Phi)$ on this last Shimura variety
	factors through a finite quotient
	$\triangle_K^\mathrm{fin}(\Phi)$.
	If
	$L_\Phi:=
	P_\Phi/U_\Phi
	\rightarrow
	G_{\Phi,l}:=
	P_\Phi/Q_\Phi$
	has a section
	(such that $L_\Phi=G_{\Phi,h}\times G_{\Phi,l}$),
	then
	$\triangle_K(\Phi)$
	acts trivially on
	$Sh(G_{\Phi,h},D_{\Phi,h},K_{\Phi,h})$
	(\cite[(2.1.16.2)]{Madapusi2012}).
	Now write
	$\mathbf{P}_K(\Phi)
	\subset
	\mathbf{S}_K(\Phi)$
	for the sub-monoid consisting of elements that are
	non-negative on
	$\mathbf{H}(\Phi)$.
	Then
	$\mathrm{FJ}_\Phi
	:=
	\mathrm{FJ}_{(\Phi,\sigma)}$
	does not depend on the choice of
	$\sigma\in\Sigma(\Phi)$
	and its image is in the invariant subspace
	$\prod_{l\in\mathbf{S}_K(\Phi)}
	H^0
	(Sh(G_{\Phi,h},D_{\Phi,h},K_{\Phi,h}),
	\mathrm{FJ}_K^{(l)}
	\otimes\omega_K(\Phi)^{\otimes n})^{\triangle_K(\Phi)}$
	(\cite[(5.18)]{Madapusi2012}).
	Now the minimal compactification of
	$Sh$ is given by
	(\cite[5.2.1]{Madapusi2012})
	\[
	Sh^\mathrm{min}
	:=
	\mathrm{Proj}
	\left(
	\oplus_{n\geq0}
	H^0(Sh^\Sigma,
	\omega_K(\Sigma)^{\otimes n})
	\right)
	\]
	together with a canonical map
	\[
	\oint^\Sigma
	\colon
	Sh^\Sigma
	\rightarrow
	Sh^\mathrm{min}.
	\]
	
	Now let $x\in Sh^\mathrm{min}(\overline{\mathbb{F}}_p)$
	be a geometric point lifting to a point
	$y\in Z_K([\Phi])$ and write
	$\mathrm{FJ}_K^{(l)}(\Phi)_y^\wedge$
	for the completion of
	$\mathrm{FJ}_K^{(l)}(\Phi)$
	at the point $y$.
	Then we have canonical isomorphisms of
	complete local rings
	(\cite[5.2.8]{Madapusi2012}):
	\[
	\mathcal{O}_{Sh^\mathrm{min},x}^\wedge
	\simeq
	\bigg(
	\prod_{l\in\mathbf{S}_K(\Phi)}
	\mathrm{FJ}_K^{(l)}(\Phi)_y^\wedge
	\bigg)^{\triangle_K(\Phi)},
	\]
	\begin{align*}
	&
	H^0
	\left(
	\mathrm{Spec}
	(\mathcal{O}_{Sh^\mathrm{min},x}^\wedge),
	(\oint^\Sigma)_\ast(\omega_K(\Phi)^{\otimes n})
	\right)
	\\
	\simeq
	&
	\bigg(
	\prod_{l\in\mathbf{S}_K(\Phi)}
	H^0(Sh(G_{\Phi,h},D_{\Phi,h},K_{\Phi,h})_y,
	\mathrm{FJ}_K^{(l)}(\Phi)\otimes
	\omega_K(\Phi)^{\otimes n})
	\bigg)^{\triangle_K(\Phi)}.
	\end{align*}

	\subsection{Base change of modular forms}
	We continue to use the notations in the
	previous subsection.
	Now we can show
	
	\begin{proposition}
		For $y\in
		Z_K(\Phi)$
		as in the last subsection,
		$\lambda\in
		X^\ast_\mathrm{dm}(\widetilde{T}_{\widetilde{P}})$,
		$l\in
		\mathbf{P}_K(\Phi)$
		positive definite
		and
		$n\geq0$,
		we have the following isomorphism
		\begin{align*}
		&
		H^0
		(Sh(G_{\Phi,h},D_{\Phi,h},K_{\Phi,h})_y
		\otimes_{\mathcal{O}_\mathfrak{p}}
		\mathcal{O}_\mathfrak{p}/\mathfrak{p}^n,
		\mathrm{FJ}_K^{(l)}\otimes
		\omega_K(\Phi)[\lambda^{-1}])
		\\
		\simeq
		&
		H^0
		(Sh(G_{\Phi,h},D_{\Phi,h},K_{\Phi,h})_y,
		\mathrm{FJ}_K^{(l)}\otimes
		\omega_K(\Phi)[\lambda^{-1}])
		\otimes
		\mathcal{O}_\mathfrak{p}/\mathfrak{p}^n.
		\end{align*}
		Moreover,
		if $x\in
		Sh^{\mathrm{min},\mu}
		(\overline{\mathbb{F}}_p)$
		is in the $\mu$-ordinary locus,
		then we have
		(write the morphism
		$\widetilde{\pi}\colon
		\mathrm{Ig}_{\infty}^\mu/\widetilde{P}_L
		\rightarrow
		Sh_\infty^\mu$):
		\begin{align*}
		&
		H^0
		(
		Sh_y
		\otimes
		\mathcal{O}_\mathfrak{p}/\mathfrak{p}^n,
		\mathrm{FJ}_K^{(l)}(\Phi)
		\otimes
		\widetilde{\pi}_\ast(\mathcal{O}_{\mathrm{Ig}_\infty^\mu/
			\widetilde{P}_L}[\lambda^{-1}])
		)
		\\
		\simeq
		&
		H^0
		(
		Sh_y,
		\mathrm{FJ}_K^{(l)}(\Phi)
		\otimes
		\widetilde{\pi}_\ast(\mathcal{O}_{\mathrm{Ig}_\infty^\mu/
			\widetilde{P}_L}[\lambda^{-1}])
		)
		\otimes
		\mathcal{O}_\mathfrak{p}/\mathfrak{p}^n.
		\end{align*}
	\end{proposition}
    \begin{proof}
    	First introduce some notations.
    	Recall that we have fixed a
    	$\mathbb{Z}$-lattice
    	$\Lambda$ of $V$
    	which is unimodular with respect to the
    	symplectic form
    	$\langle\cdot,\cdot\rangle$.
    	Let
    	$\mathfrak{C}$
    	be the set of totally isotropic submodules
    	of $\Lambda$.
    	For any $\Lambda'\in\mathfrak{C}$,
    	denote by
    	$\Lambda^{'\perp}$
    	for the submodule of $\Lambda$ orthogonal to
    	$\Lambda'$
    	and
    	by $C(\Lambda/\Lambda^{'\perp})$
    	the cone of positive definite
    	symmetric bilinear forms on the vector space
    	$\Lambda/\Lambda^{'\perp}\otimes_\mathbb{Z}\mathbb{R}$
    	whose radical is defined over $\mathbb{Q}$.
    	It is clear that for
    	$\Lambda''\subset\Lambda'$ in
    	$\mathfrak{C}$,
    	one has
    	$C(V/V^{''\perp})
    	\subset
    	C(V/V^{'\perp})$.
    	Consider the equivalence relation
    	on the set of
    	$C(V/V^{'\perp})$
    	for all $\Lambda\in\mathfrak{C}$
    	generated by this
    	inclusion relation.
    	We write the quotient set of the above set by this
    	equivalence relation by
    	$\mathcal{C}$.
    	Note that for any $\Lambda'\in\mathfrak{C}$,
    	there is a natural filtration
    	$W_\ast\Lambda$ on $\Lambda$
    	\[
    	W_{-3}\Lambda=0
    	\subset
    	W_{-2}\Lambda=\Lambda'
    	\subset
    	W_{-1}\Lambda=\Lambda^{'\perp}
    	\subset
    	W_0\Lambda=\Lambda.
    	\]
    	Then we write $P_{\Lambda'}\subset G$
    	for the parabolic subgroup stabilizing this filtration
    	$W_\ast\Lambda$
    	(\cite[§2.2.1]{Madapusi2012}).
    	Conversely,
    	for any admissible parabolic subgroup $P\subset G$,
    	there is a canonical filtration
    	$W_\ast\mathfrak{g}$ on $\mathfrak{g}$:
    	\[
    	W_{-3}\mathfrak{g}=0
    	\subset
    	W_{-2}\mathfrak{g}=\mathrm{Lie}(W_P)
    	\subset
    	W_{-1}\mathfrak{g}=\mathrm{Lie}(U_P)
    	\subset
    	W_0\mathfrak{g}=\mathrm{Lie}(P)
    	\subset
    	W_1\mathfrak{g}=\mathfrak{g}
    	\]
    	which then induces a filtration on
    	$\Lambda$
    	Now we choose a $\Lambda'\in\mathfrak{C}$
    	such that
    	the corresponding parabolic  subgroup
    	$P_{\Lambda'}$
    	is the same as $P_\Phi$.  	
    	Note that
    	$\Psi_K^{(l)}(\Phi)$
    	is the pull-back of the Poincar\'{e} line bundle
    	on
    	$\mathcal{A}\times\mathcal{A}^t$
    	(recall
    	$\mathcal{A}$ is the universal abelian scheme over
    	$Sh$)
    	along the natural map
    	$c_l
    	\colon
    	\mathrm{Hom}_\mathbb{Z}
    	(\Lambda/\Lambda^{'\perp},\mathcal{A}^t)
    	\rightarrow
    	\mathcal{A}\times\mathcal{A}^t$.
    	Moreover,
    	$Sh(\overline{Q}_\Phi,
    	\overline{D}_\Phi,\overline{K}_\Phi)$
    	is in fact a
    	$\mathrm{Hom}_\mathbb{Z}(\Lambda/\Lambda^{'\perp},
    	\mathcal{A}^t)$-torsor over the Shimura variety
    	$Sh(G_{\Phi,h},D_{\Phi,h},K_{\Phi,h})$.
    	Thus by Mumford's vanishing theorem
    	(\cite[§III.16]{Mumford1970}),
    	we get the first isomorphism.
    	
    	As for the second isomorphism,
    	we define the following torsor
    	over
    	$Sh(\overline{Q}_\Phi,\overline{D}_\Phi,
    	\overline{K}_\Phi)_m$:
    	\[
    	\pi_{m,n}^\Phi
    	\colon
    	\mathrm{Ig}^\Phi_{m,n}
    	:=\underline{\mathrm{Isom}}_{
    		Sh(\overline{Q}_\Phi,\overline{D}_\Phi,
    		\overline{K}_\Phi)_m}
    	(\mathcal{A}_{x_0}[p^n]^\circ,
    	\mathcal{A}[p^ny]^\circ)
    	\rightarrow
    	Sh(\overline{Q}_\Phi,\overline{D}_\Phi,
    	\overline{K}_\Phi)_m.
    	\]
    	Since
    	the morphism
    	$
    	\pi_{m,n}^\Phi/\widetilde{P}_L^\mathrm{der}
    	\colon
    	\mathrm{Ig}_{m,n}^\Phi/\widetilde{P}_L^\mathrm{der}
    	\rightarrow
    	Sh(\overline{Q}_\Phi,\overline{D}_\Phi,
    	\overline{K}_\Phi)$
    	is finite,
    	thus
    	$(\pi_{m,n}^\Phi/\widetilde{P}_L^\mathrm{der})
    	^\ast
    	(\mathrm{FJ}_K^{(l)})$
    	is an ample line bundle
    	(\cite[II.5.1.12]{EGA1961}),
    	and thus is affine.
    	From this we get the following isomorphism
    	\[
    	H^0
    	((\mathrm{Ig}_{m,n}^\Phi/\widetilde{P}_L^\mathrm{der})_x
    	\otimes\mathcal{O}_\mathfrak{p}/\mathfrak{p}^n,
    	(\pi_{m,n}^\Phi/\widetilde{P}_L^\mathrm{der})^\ast
    	\mathrm{FJ}_K^{(l)})
    	\simeq
    	H^0
    	((\mathrm{Ig}_{m,n}^\Phi/\widetilde{P}_L^\mathrm{der})_x,
    	(\pi_{m,n}^\Phi/\widetilde{P}_L^\mathrm{der})^\ast
    	\mathrm{FJ}_K^{(l)})
    	\otimes\mathcal{O}_\mathfrak{p}/\mathfrak{p}^n.
    	\]
    	This isomorphism corresponds to the second isomorphism
    	in the proposition in the case
    	$\lambda=0$.
    	In general,
    	since
    	$\mathcal{O}_{\mathrm{Ig}_{m,n}^\Phi/\widetilde{P}_L^\mathrm{der}}
    	[\lambda^{-1}]$
    	and
    	$\mathcal{O}_{\mathrm{Ig}_{m,n}^\Phi/\widetilde{P}_L^\mathrm{der}}
    	[0]$
    	are equivalent as
    	$\mathcal{O}_{\mathrm{Ig}_{m,n}^\Phi/\widetilde{P}_L^\mathrm{der}}$-modules,
    	thus we deduce the second isomorphism
    	for any $\lambda\in
    	X^\ast_\mathrm{dm}(\widetilde{T}_{\widetilde{P}})$.
    \end{proof}

    We can deduce the base change of cuspidal
    modular forms on
    the minimal compactification:
    \begin{proposition}
    	Assume that $K$ is a neat compact open subgroup of
    	$G(\mathbb{A}_\mathrm{f})$.
    	Then we have the following isomorphism
    	\begin{align*}
    	&
    	H^0
    	\bigg(
    	\mathrm{Spec}((Sh^{\mathrm{min}})_x^\wedge)
    	\otimes
    	\mathcal{O}_\mathfrak{p}/\mathfrak{p}^n,
    	\oint^\Sigma_\ast
    	\mathcal{R}_\infty^\mathrm{Ig}[\lambda^{-1}](-C^\Sigma)
    	\bigg)
    	\\
    	\simeq
    	&
    	H^0
    	\bigg(
    	\mathrm{Spec}((Sh^{\mathrm{min}})_x^\wedge),
    	\oint^\Sigma_\ast
    	\mathcal{R}_\infty^\mathrm{Ig}[\lambda^{-1}](-C^\Sigma)
    	\bigg)
    	\otimes
    	\mathcal{O}_\mathfrak{p}/\mathfrak{p}^n.
    	\end{align*}
    	Moreover,
    	if $x$ is in the $\mu$-ordinary locus,
    	then the following spaces are isomorphic
    	\begin{align*}
    	&
    	H^0
    	\bigg(
    	\mathrm{Spec}((Sh^{\mathrm{min}})_x^\wedge)
    	\otimes
    	\mathcal{O}_\mathfrak{p}/\mathfrak{p}^n,
    	\oint^\Sigma_\ast
    	\mathcal{R}_\infty^\mathrm{top,Ig}[\lambda^{-1}](-C^\Sigma)
    	\bigg)
    	\\
    	\simeq
    	&
    	H^0
    	\bigg(
    	\mathrm{Spec}((Sh^{\mathrm{min}})_x^\wedge),
    	\oint^\Sigma_\ast
    	\mathcal{R}_\infty^\mathrm{top,Ig}[\lambda^{-1}](-C^\Sigma)
    	\bigg)
    	\otimes
    	\mathcal{O}_\mathfrak{p}/\mathfrak{p}^n.
    	\end{align*}
    \end{proposition}
    \begin{proof}
    	As in the proof of
    	\cite[Proposition 6.2]{Pilloni2012},
    	it suffices to show that
    	the action of
    	the group
    	$\triangle_K(\Phi)$ on these spaces of cuspidal
    	forms is trivial.
    	Indeed,
    	for $l$ positive definite,
    	its stabilizer in $\triangle_K(\Phi)$
    	is trivial by the assumption that
    	$K$ is neat.
    \end{proof}

    For a scheme $Y$ over
    $\mathbb{W}$,
    we write
    $i_m
    \colon
    Y_m
    \hookrightarrow
    Y$
    for the closed immersion of the reduction modulo
    $\mathfrak{p}^m$ of $Y$ into $Y$.
    Then we have the following commutative diagram
    \[
    \begin{tikzcd}
    Sh^\Sigma_m
    \arrow[r,"i_m"]
    \arrow[d,"\pi"]
    &
    Sh^\Sigma
    \arrow[d,"\pi"]
    \\
    Sh^\mathrm{min}_m
    \arrow[r,"i_m"]
    &
    Sh^\mathrm{min}
    \end{tikzcd}
    \]
    We have a similar commutative diagram if we take the
    $\mu$-ordinary locus of each space in the diagram.
    We write
    $i_m^\mu,\pi^\mu$
    for the induced morphisms.
    From the above proposition,
    we get
    \begin{corollary}
    	The following natural morphisms of automorphic sheaves
    	are isomorphisms
    	\[
    	i_m^\ast
    	\pi_\ast
    	\oint_\ast^\Sigma
    	\mathcal{R}_\infty^\mathrm{Ig}[\lambda^{-1}](-C^\Sigma)
    	\simeq
    	\pi_\ast
    	i_m^\ast
    	\oint_\ast^\Sigma
    	\mathcal{R}_\infty^\mathrm{Ig}[\lambda^{-1}](-C^\Sigma),
    	\]
    	\[
    	(i_m^\mu)^\ast
    	(\pi^\mu)_\ast
    	\oint_\ast^\Sigma
    	\mathcal{R}_\infty^\mathrm{top,Ig}[\lambda^{-1}](-C^\Sigma)
    	\simeq
    	(\pi^\mu)_\ast
    	(i_m^\mu)^\ast
    	\oint_\ast^\Sigma
    	\mathcal{R}_\infty^\mathrm{top,Ig}[\lambda^{-1}](-C^\Sigma).
    	\]
    \end{corollary}

    From these isomorphisms,
    one can deduce
    \begin{proposition}\label{surjectivity of mod p, mu-ordinary}
    	For any
    	$\lambda\in X^\ast_\mathrm{dm}(\widetilde{T}_{\widetilde{P}})$
    	and any integer $m\geq0$
    	the following natural morphisms of reduction modulo
    	$\mathfrak{p}^m$ are surjective
    	\begin{align*}
    	H^0(\mathrm{Ig}_\infty/\widetilde{P}_L,
    	\mathcal{R}_\infty^\mathrm{Ig}[\lambda^{-1}]
    	(-C^\Sigma))
    	&
    	\rightarrow
    	H^0(\mathrm{Ig}_m/\widetilde{P}_L,
    	\mathcal{R}_m^\mathrm{Ig}[\lambda^{-1}]
    	(-C^\Sigma)),
    	\\
    	H^0(Sh^{\Sigma,\mu}_\infty,\mathcal{V}_\lambda
    	(-C^\Sigma))
    	&
    	\rightarrow
    	H^0(Sh^{\Sigma,\mu}_m,\mathcal{V}_\lambda
    	(-C^\Sigma)),
    	\\
    	\mathbb{V}_{\mathrm{cusp},\infty}^{\widetilde{P}_L^\mathrm{der}}
    	[\lambda^{-1}]
    	&
    	\rightarrow
    	\mathbb{V}_{\mathrm{cusp},m}^{\widetilde{P}_L^\mathrm{der}}[\lambda^{-1}].
    	\end{align*}
    \end{proposition}
    \begin{proof}
    	This follows from the fact that
    	the $\mu$-ordinary locus on the minimal compactification
    	$Sh^\mathrm{min}$
    	is affine
    	because
    	the latter is projective by construction.
    \end{proof}

    Recall the Hodge line bundle
    $\omega$
    we used to construct the Hasse invariant $H$
    in §\ref{Hasse invariant}
    (using its power
    $\omega^{\otimes N_G}$).
    Similar to the above proposition,
    one has
    \begin{proposition}\label{surjectivity of mod p, without mu-ord}
    	For each
    	$\lambda\in X^\ast_\mathrm{dm}(\widetilde{T}_{\widetilde{P}})$,
    	there is a positive integer
    	$N(\lambda)$
    	such that 
    	for any $t\ge N(\lambda)$, the following reduction map is
    	surjective
    	\[
    	H^0(Sh^\Sigma_\infty,
    	\mathcal{V}_{\lambda+tN_G\lambda_G}(-C^\Sigma))
    	\rightarrow
    	H^0(Sh^\Sigma_m,
    	\mathcal{V}_{\lambda+tN_G\lambda_G}(-C^\Sigma)).
    	\]
    \end{proposition}
    The proof is the same as above by noting that
    $\omega$ is an ample line bundle.

    Recall that we have a morphism
    $\widehat{\mathrm{HT}}^{\widetilde{P}_L}
    \colon
    \mathrm{Ig}/\widetilde{P}_L^\mathrm{der}(\mathcal{O}_\mathfrak{p})
    \xrightarrow{\widetilde{\mathrm{HT}}^{\widetilde{P}_L}}
    \widetilde{\mathrm{Ig}}^{\widetilde{P}_L}
    \rightarrow
    Sh^{\Sigma,\mu}$.
    We deduce the following important corollary
    \begin{corollary}
    	The module
    	\[
    	\bigg(
    	\bigoplus_{\lambda\in X^\ast_\mathrm{dm}(\widetilde{T}_{\widetilde{P}})}
    	(\widehat{\mathrm{HT}}^{\widetilde{P}_L})^\ast
    	H^0(Sh^{\Sigma,\mu},
    	\mathcal{V}^\Sigma[\lambda^{-1}])[\frac{1}{p}]
    	\bigg)
    	\bigcap
    	\mathbb{V}_{\mathrm{cusp},\infty}^{\widetilde{P}_L^\mathrm{der}}
    	\]
    	is dense in
    	$\mathbb{V}_{\mathrm{cusp},\infty}^{\widetilde{P}_L^\mathrm{der}}$.
    \end{corollary}

    \subsection{Finite dimensionality}
    In this subsection,
    we will show that the space of $\mu$-ordinary
    $p$-adic modular forms is bounded in certain sense.
    The main strategy is to use
    Proposition \ref{Hasse invariant is reduced}
    (\textit{cf.}\cite[Appendice A.3]{Pilloni2012}).

	First note that for a $\mu$-ordinary modular form
	$f$,
	the multiplication by Hasse invariant
	$H(f)$ is again $\mu$-ordinary:
	\begin{lemma}\label{Hasse invariant commutes with Hecke operators}
		We fix a character
		$\lambda\in
		X^\ast_\mathrm{dm}(\widetilde{T}_{\widetilde{P}})$.
		For any cuspidal
		automorphic form
		$f$ in the space
		$H^0(Sh^\Sigma_1,
		\mathcal{V}_\lambda^\Sigma(-C^\Sigma))$
		and any Hecke operator
		$\mathbb{T}_\epsilon$
		with
		$\epsilon\in
		\widetilde{T}_{\widetilde{P}}^+(E_\mathfrak{p})$,
		we have
		\[
		H(\mathbb{T}_\epsilon f)
		=
		\mathbb{T}_\epsilon(Hf).
		\]
	\end{lemma}
    \begin{proof}
    	Consider a $\mu$-ordinary point
    	$x\in
    	Sh^\mu$,
    	$\mathcal{A}_x$
    	the abelian scheme over $x$
    	and
    	$(\mathbb{D}_x,\mathrm{Fr})$
    	the $F$-crystal with $G$-structure associated to
    	the $p$-divisible group
    	$\mathcal{A}_x[p^\infty]$ of
    	$\mathcal{A}_x$.
    	One can choose a basis
    	\[
    	\mathcal{B}=\{e_1,\cdots,e_n,e_1^\ast,\cdots,e_n^\ast\}
    	\]
    	of $\mathbb{D}_x$ under which
    	$\mathrm{Fr}^w$ becomes
    	$\mathrm{diag}(1_n,p1_n)$.
    	Here
    	$2n=\mathrm{dim}(V)$
    	and
    	$\mathbb{F}_{p^w}$ is the splitting field of
    	the torus 
    	$(T_{p_{\overline{\mu}}})_{\mathbb{F}_p}$.
    	Then
    	$Hf(x)$
    	is the multiplication of
    	$f(x)$ by the determinant
    	$\mathrm{det}(\overline{\mathrm{Fr}^w})$
    	where
    	$\overline{\mathrm{Fr}^w}$
    	is the induced map on the quotient
    	$\mathbb{D}_x/\mathrm{Fil}^1\mathbb{D}_x$.
    	Applying the Hecke operator
    	$\mathbb{T}_\epsilon$
    	to the form
    	$f$
    	amounts,
    	in terms of Dieudonn\'{e} crystals,
    	to consider subcrystals
    	$\mathbb{D}_x'$ of
    	$\mathbb{D}_x$
    	such that under the basis $\mathcal{B}$
    	one has
    	$\epsilon\mathbb{D}_x=\mathbb{D}_x'$.
    	Since the Hodge filtration on
    	$\mathbb{D}_x'$ is induced from
    	$\mathbb{D}_x$,
    	we see that
    	the determinant
    	$\mathrm{det}(\overline{\mathrm{Fr}^w}')$
    	of the Frobenius on the quotient
    	$\mathbb{D}_x'/\mathrm{Fil}(\mathbb{D}_x')$
    	is the same as
    	$\mathrm{det}(\overline{\mathrm{Fr}^w})$.
    	Thus we conclude that
    	$H(\mathbb{T}_\epsilon f)
    	=
    	\mathbb{T}_\epsilon(Hf)$
    	for any $f$.
    \end{proof}

    Next we turn to canonical subgroups of
    the $p$-divisible group
    $\mathcal{A}[p^\infty]$
    of the abelian scheme
    $\mathcal{A}$ in a strict neighbourhood of the
    $\mu$-ordinary locus
    $Sh^\mu$.
    We first recall some definitions from
    \cite{Bijakowski2016}
    (see also\cite{Fargues2010}).
    Fix a finite extension $L$ of
    $\mathbb{Q}_p$
    and also a valuation
    $v$ on $L$
    such that
    $v(p)=1$.
    Let $H$ be a finite flat group scheme of
    $p$-power order over
    the ring of integers
    $\mathcal{O}_L$ of $L$.
    Write
    $\omega_H$
    for the co-normal module along the unit section of
    $H$.
    The degree of $H$ is defined to be 
    \[
    \mathrm{deg}(H)
    :=
    v(\mathrm{Fitt}(\omega_H)),
    \]
    where
    $\mathrm{Fitt}$
    is the Fitting ideal.
    For two finite flat subgroups
    $H_1,H_2$ of $H$,
    one has
    (\textit{cf.}\cite[Corollary 1.16]{Bijakowski2016}):
    \[
    \mathrm{deg}(H_1)+\mathrm{deg}(H_2)
    \leq
    \mathrm{deg}(H_1\cap H_2)+
    \mathrm{deg}(H_1+H_2).
    \]
    Write
    $\mathrm{ht}(H)$
    for the height of $H$.
    Then the slope of $H$ is related to
    its height and degree by the formula:
    \[
    \lambda(H)
    =
    \frac{\mathrm{deg}(H)}{\mathrm{ht}(H)}.
    \]
    Recall that the slopes of the abelian scheme
    $\mathcal{A}_{x_0}$
    over a point
    $x_0\in
    Sh^\mu$
    in the $\mu$-ordinary locus are
    $1\geq\lambda_1>\lambda_2\cdots>\lambda_r\geq0$.
    Write
    \[
    \delta
    =
    \frac{1}{4}
    \min_{0<i<r}
    \lambda_i-\lambda_{i+1}>0.
    \]
    Let
    $\mathcal{A}$
    be the abelian scheme
    over
    $Sh$.
    For a point $x\in
    Sh^\mu$
    in the $\mu$-ordinary locus,
    we have a slope filtration
    of the $p$-divisible group
    $\mathcal{D}_x$
    associated to the abelian scheme
    $\mathcal{A}_x$ over the point $x$:
    \[
    0
    \subset
    (\mathcal{D}_x)_1
    \subset
    (\mathcal{D}_x)_2
    \subset
    \cdots
    \subset
    (\mathcal{D}_x)_r
    \subset
    \mathcal{D}_x.
    \]
    Correspondingly we have a filtration for the
    $\mathfrak{p}$-torsion subgroups of 
    the above subgroups
    $0\subset
    (\mathcal{D}_x)_i[\mathfrak{p}]
    \subset
    \cdots
    \subset
    (\mathcal{D}_x)_r[\mathfrak{p}]
    \subset
    \mathcal{D}_x[\mathfrak{p}]$.
    We write
    $h_i$
    for the height of
    $(\mathcal{D}_x)_i[\mathfrak{p}]$
    ($i=1,\cdots,r$).
    Then
    one has
    \[
    \mathrm{deg}
    \big(
    (\mathcal{D}_x)_i[\mathfrak{p}]
    \big)
    =
    h_1\lambda_1+(h_2-h_1)\lambda_2+
    \cdots
    +(h_i-h_{i-1})\lambda_i
    =:
    d_i.
    \]

    Consider a maximal element
    $b\in
    B(G,X)\backslash\{\overline{\mu}\}$
    and a geometric point
    $x\in \mathcal{N}^b$
    in the Shimura scheme.
    Then the slopes of
    the abelian scheme
    $\mathcal{A}_x$
    are
    \begin{align*}
    &
    1\geq\lambda_1>\cdots>\lambda_{i-1}>
    (\lambda_i)>\lambda_i'>
    (\lambda_{i+1})
    >\lambda_{i+2}
    >
    \cdots
    \\
    >
    &
    \lambda_{r-i-1}
    >(\lambda_{r-i})>
    \lambda_{r-i}'
    >
    (\lambda_{r+1-i})
    >\lambda_{r+2-i}
    \cdots
    >
    \lambda_r
    \geq0
    \end{align*}
    where
    $(\lambda_i)$
    means that the slope
    $\lambda_i$ may not exist in the
    above sequence
    and
    $\lambda_i'=\frac{1}{2}(\lambda_i+\lambda_{i+1})$
    and
    similarly
    $\lambda_{r-i}'
    =\frac{1}{2}(\lambda_{r-i}+\lambda_{r+1-i})$.

    Now we have the following observation
    (\textit{cf.}
    \cite[Proposition 1.24]{Bijakowski2016}):
    \begin{lemma}\label{at most one subgroup}
    	Let
    	$b$ be 
    	$\overline{\mu}$
    	or a maximal element in
    	$
    	B(G,X)\backslash\{\overline{\mu}\}$
    	and consider a geometric point
    	$x\in \mathcal{N}^b$.
    	Then for any $i=1,\cdots,r$,
    	there is at most one subgroup
    	$H$ of
    	the finite flat group scheme
    	$\mathcal{D}_x[\mathfrak{p}]$
    	such that
    	$\mathrm{ht}(H)=h_i$
    	and
    	$\mathrm{deg}(H)
    	>
    	d_i-\delta$.
    \end{lemma}
    \begin{proof}
    	Suppose that there are two such subgroups
    	$H_1,H_2$ of
    	$H:=(\mathcal{D}_x)[\mathfrak{p}]$.
    	Write
    	$h$ for the height of
    	$H_1\cap H_2$.
    	Then the height of
    	$H_1+H_2$ is
    	$2h_i-h$.

    	First we assume that
    	$b=\overline{\mu}$.
    	Then the degree of
    	$H_1\cap H_2$
    	is bounded as follows:
    	if
    	$h_k\le h\le h_{k+1}$
    	(set $h_0=0$),
    	then
    	$\mathrm{deg}(H_1\cap H_2)
    	\le
    	h_1\lambda_1+
    	(h_2-h_1)\lambda_2
    	+
    	\cdots
    	+
    	(h_k-h_{k-1})\lambda_k
    	+
    	(h-h_k)\lambda_{k+1}$.
    	On the other hand,
    	if
    	$h_l\le 2h_i-h\le h_{l+1}$,
    	then similarly
    	$\mathrm{deg}(H_1+H_2)
    	\le
    	h_1\lambda_1+(h_2-h_1)\lambda_2
    	+
    	\cdots
    	+
    	(h_l-h_{l-1})\lambda_l
    	+
    	(2h_i-h-h_l)\lambda_{l+1}$.
    	Suppose that
    	$h<h_i$,
    	then we have
    	$h_{k+1}\le h_i\le h_l$.
    	Write
    	\begin{align*}
    	\mathrm{deg}_{i;k,l}
    	=
    	&
    	h_1\lambda_1+
    	(h_2-h_1)\lambda_2
    	+
    	\cdots
    	+
    	(h_k-h_{k-1})\lambda_k
    	+
    	(h-h_k)\lambda_{k+1}
    	\\
    	&
    	+
    	h_1\lambda_1+(h_2-h_1)\lambda_2
    	+
    	\cdots
    	+
    	(h_l-h_{l-1})\lambda_l
    	+
    	(2h_i-h-h_l)\lambda_{l+1}.
    	\end{align*}
    	From the fact
    	$\mathrm{deg}(H_1)+\mathrm{deg}(H_2)
    	\leq
    	\mathrm{deg}(H_1\cap H_2)+
    	\mathrm{deg}(H_1+H_2)$,
    	and the assumption
    	$\mathrm{deg}(H_1),\mathrm{deg}(H_2)
    	>d_i-\delta$,
    	one gets
    	by a simple computation
    	the following inequality:
    	\begin{align*}
    	&
    	(h_{k+1}-h_k-(h-h_k))
    	(\lambda_{k+1}-\lambda_{k+2})
    	+
    	\cdots
    	+
    	(h_i-h_k-(h-h_k))
    	(\lambda_i-\lambda_{i+1})
    	\\
    	+
    	&
    	(h_i-h_k-(h-h_k+h_{i+1}-h_i))
    	(\lambda_{i+1}-\lambda_{i+2})
    	+
    	\cdots
    	+
    	(h_i-h_k-(h-h_k+h_l-h_i))(\lambda_l-\lambda_{l+1})
    	\\
    	+
    	&
    	(h_i-h)\lambda_{l+1}
    	<
    	2\delta
    	\end{align*}
    	which is impossible by the definition of
    	$\delta$.
    	Thus we get
    	$h=h_i$
    	and therefore
    	$H_1=H_2$.
    	
    	Next consider
    	$b\ne\overline{\mu}$.
    	Suppose that
    	the slopes of
    	$\mathcal{A}_x$
    	are
    	\[
    	\lambda_1>\cdots>\lambda_{i_0}>\lambda_{i_0}'>\lambda_{i_0+1}
    	\cdots
    	>
    	\lambda_{r-i_0-1}>\lambda_{r-i_0}
    	>\lambda_{r-i_0}'
    	>\lambda_{r-i_0+1}
    	>
    	\cdots
    	>
    	\lambda_r,
    	\]
    	and the corresponding heights are
    	\begin{align*}
    	&
    	h_1<\cdots<h_{i_0-1}<h_{i_0}-\Delta h<h_{i_0}+\Delta h<h_{i_0+1}
    	<
    	\cdots
    	\\
    	<
    	&
    	h_{r-i_0-1}<
    	h_{r-i_0}-\Delta h<
    	h_{r-i_0}+\Delta h<
    	h_{r-i_0+1}<
    	\cdots<
    	h_r.
    	\end{align*}
    	We then put
    	\begin{align*}
    	s_1
    	&
    	=
    	d_1,
    	\cdots,
    	s_{i_0-1}
    	=
    	d_{i_0-1},
    	\\
    	s_{i_0}
    	&
    	=
    	d_{i_0}
    	-\Delta h\lambda_{i_0},
    	s_{i_0}'
    	=
    	d_{i_0}
    	+\Delta h\lambda_{i_0+1},
    	s_{i_0+1}
    	=
    	d_{i_0+1}
    	-\Delta h(\lambda_{i_0}-\lambda_{i_0+1}),
    	\cdots,
    	\\
    	s_{r-i_0-1}
    	&
    	=
    	d_{r-i_0-1}
    	-\Delta h(\lambda_{i_0}-\lambda_{i_0+1}),
    	s_{r-i_0}
    	=
    	d_{r-i_0}
    	-\Delta h(\lambda_{i_0}-\lambda_{i_0+1})
    	-\Delta h\lambda_{r-i_0},
    	\\
    	s_{r-i_0}'
    	&
    	=
    	d_{r-i_0+1}-\Delta h(\lambda_{i_0}-\lambda_{i_0+1})
    	+
    	\Delta h\lambda_{r-i_0+1},
    	\\
    	s_{r-i_0+1}
    	&
    	=
    	d_{r-i_0+1}-\Delta h(\lambda_{i_0}-\lambda_{i_0+1})
    	-\Delta h(\lambda_{r-i_0}-\lambda_{r-i_0+1}),
    	\cdots
    	\\
    	s_r
    	&
    	=
    	d_r
    	-\Delta h(\lambda_{i_0}-\lambda_{i_0+1})
    	-\Delta h(\lambda_{r-i_0}-\lambda_{r-i_0+1}).
    	\end{align*}
    	These quantities are just the heights of
    	the subgroups of
    	$\mathcal{D}_x[\mathfrak{p}]$
    	of the corresponding height
    	and slopes as above.
    	Now the same reasoning as above shows that there is at most one subgroup
    	$H$ of
    	$\mathcal{D}_x[\mathfrak{p}]$
    	such that
    	$\mathrm{ht}(H)=h_i$,
    	resp.,
    	$=h_{i_0}'$,
    	$=h_{r-i_0}'$
    	and
    	$\mathrm{deg}(H)>s_i-\delta$,
    	resp.,
    	$>s_{i_0}'-\delta$,
    	$>s_{r-i_0}'-\delta$.
    	If
    	$i<i_0$,
    	then we are done by the definition of
    	$s_i$.
    	If
    	$i=i_0$,
    	then
    	(assume again that
    	$h_k\le h\le h_{k+1}$
    	and
    	$h_l\le2h_{i_0}-h\le h_{l+1}$)
    	\begin{align*}
    	&
    	2(d_{i_0}-\delta)
    	-2\Delta h\lambda_{i_0}
    	=
    	2(s_{i_0}-\delta)
    	<
    	\mathrm{deg}(H_1)+\mathrm{deg}(H_2)
    	\\
    	\le
    	&
    	\mathrm{deg}(H_1\cap H_2)
    	+
    	\mathrm{deg}(H_1+H_2)
    	\\
    	&
    	\le
    	\mathrm{deg}_{i_0;k,l}
    	-2\Delta h\lambda_{i_0}.
    	\end{align*}
    	The last inequality is true regardless of the relation between
    	$h$ and $\Delta h$.
    	Thus from the first part,
    	we know that this is impossible unless
    	$h=h_i$.    	
    	Similar reasoning applies to the remaining cases.
    \end{proof}

    Using the arguments from
    \cite[Corollary 1.26]{Bijakowski2016},
    we get the existence and uniqueness of canonical subgroups
    $0\subset
    H_1
    \subset
    H_2
    \subset
    \cdots
    \subset
    H_r    
    \subset
    \mathcal{D}_x[\mathfrak{p}]$
    with
    $x\in
    Sh$
    such that its reduction
    $\overline{x}\in
    \mathcal{N}^b$
    with $b\in B(G,X)$
    as in the above lemma,
    the maximal element or next to the maximal element.
    Here we require that
    each $H_i$ is of height $\mathrm{ht}(H_i)=h_i$ and of degree
    $\mathrm{deg}(H_i)>d_i-\delta$.
    Now one can define Hecke operators
    $\mathbb{T}_{H_i}$
    for each subgroup $H_i$
    as follows:
    we write
    $Sh_1^\diamond
    :=
    \cup_b\mathcal{N}^b$
    with $b$ running through the maximal or next-to-the-maximal
    elements in
    $B(G,X)$.
    Then set
    $
    Sh^\diamond
    \subset
    Sh
    $,
    consisting of those
    $x$
    whose reduction
    $\overline{x}\in Sh_1^\diamond$.
    For each
    $H_i$,
    there is
    a $p$-isogeny of kernel $H_i$:
    \[
    \pi_i
    \colon
    \mathcal{A}
    \rightarrow
    \mathcal{A}_i:=\mathcal{A}/H_i.
    \]
    We then define a correspondence
    $Sh(V)_{H_i}$
    over
    $Sh(V)$
    as follows:
    for any $\mathbb{Z}_p$-algebra $A$,
    $Sh(V)_{H_i}(A)$
    is the set of equivalence classes of triples
    $(\mathcal{A},H,\psi_{K_V})$
    where
    $\mathcal{A}/A$ is a
    principally polarized abelian scheme over $A$,
    $\psi_{K_V}$
    a level structure on $\mathcal{A}$ of level $K_V$
    and
    $H\subset
    \mathcal{A}[\mathfrak{p}]$
    a finite flat subgroup of height
    $\mathrm{ht}(H)=h_i$ and of degree
    $\mathrm{deg}(H)>d_i-\delta$.
    Then this correspondence is representable (again denoted) by
    $Sh(V)_{H_i}$.
    We write
    $Sh_{H_i}$ for
    the pull-back of this correspondence along
    the embedding
    $Sh
    \hookrightarrow
    Sh(V)$.
    We have two natural projections
    $\mathrm{pr}_1,\mathrm{pr}_2
    \colon
    Sh_{H_i}
    \rightarrow
    Sh$,
    $\mathrm{pr}_1((\mathcal{A},\psi_K,H))
    =(\mathcal{A},\psi_K)$
    and
    $\mathrm{pr}_2((\mathcal{A},\psi_K,H))=
    (\mathcal{A}/H,\psi_K)$.
    Write
    $\epsilon_{H_i}$
    for
    the element
    in
    $\widetilde{T}_{\widetilde{P}}^+(E_\mathfrak{p})$
    which induces the filtration
    $0\subset H_i\subset\mathcal{A}[\mathfrak{p}]$.
    Now for any $\mathcal{O}$-algebra
    $A$,
    consider the endomorphism
    $\mathbb{T}_{H_i}$
    of
    $H^0(Sh^\diamond,\mathcal{V}_\lambda)$
    given by
    the composition
    of the following morphisms
    \[
    H^0(Sh^\diamond_{/A},\mathcal{V}_\lambda)
    \rightarrow
    H^0(Sh_{H_i/A},
    \mathrm{pr}_2^\ast\mathcal{V}_\lambda)
    \xrightarrow{\pi^\ast}
    H^0(Sh_{H_i/A},\mathrm{pr}_1^\ast
    \mathcal{V}_\lambda)
    \xrightarrow{\frac{1}{m_{\epsilon_{H_i}}}
    	\mathrm{Tr}_{\mathrm{pr}_1}}
    H^0(Sh^\diamond_{/A},
    \mathcal{V}_\lambda).
    \]

    On the abelian scheme
    $\mathcal{A}$ over
    $Sh^\diamond$,
    there is a filtration of canonical subgroups
    as described above
    $0\subset
    H_1\subset
    H_2\subset
    \cdots
    \subset
    H_r
    \subset
    \mathcal{A}[\mathfrak{p}]$.
    For each $x\in
    Sh^\diamond$,
    we write
    $(H_i)_x$
    for the specialization of $H_i$ to $x$.
    Now
    recall the argument in
    the proof of
    Proposition \ref{Integrality of Hecke operator},
    for $x$ in the $\mu$-ordinary locus,
    $\pi^\ast$
    corresponds to the diagonal matrix
    $\epsilon_{H_i}$
    while
    for $x$ not in the $\mu$-ordinary locus,
    $\pi^\ast$
    corresponds to the diagonal matrix
    $\epsilon_{H_i}'$
    which is a product of
    $\epsilon_{H_i}$ and another diagonal matrix
    $M_i\in
    \mathrm{GSp}(V,\psi)$.
    Here
    on the diagonal of $M_i$,
    the $(h_i-1)$-th entry is given by
    $p^{-1+\mathrm{deg}(H_i)_x}$
    and the
    $(\mathrm{dim}V-h_i+1)$-th entry is given by
    $p^{1-\mathrm{deg}(H_i)_x}$,
    the remaining entries are all $1$.
    We put
    \[
    C=
    \max_i
    m_{\epsilon_{H_i}'}.
    \]
    Note that this constant
    $C$
    depends only on
    the Shimura datum
    $(G,X)$
    and the parabolic subgroup
    $\widetilde{P}_L$.

    The next lemma shows that for large weights
    $\lambda\in
    X^\ast(\widetilde{T}_{\widetilde{P}})$,
    the Hecke operators
    $\mathbb{T}_{H_i}$ act as zero
    on the non-$\mu$-ordinary locus
    $Sh_1^{\Sigma,n-\mu}
    =
    Sh_1^\Sigma
    \backslash
    Sh_1^\mu$
    (\textit{cf.}
    \cite[Proposition A.5]{Pilloni2012}):
    \begin{lemma}
    	For any
    	$i=1,2,\cdots,r$
    	and
    	any modular form
    	$f\in
    	H^0(Sh_1^\Sigma,\mathcal{V}_\lambda)$
    	of weight
    	$\lambda\in
    	X^\ast(\widetilde{T}_{\widetilde{P}})$
    	with
    	$\mathrm{val}_p(\lambda(\epsilon_{H_i}'))\ge
    	m_{\epsilon_{H_i}'}+1$,
    	one has
    	\[
    	(\mathbb{T}_{H_i}f)|_{Sh^{n-\mu}_1}=0.
    	\]
    	In particular,
    	for any
    	$f\in
    	H^0(Sh_1^\Sigma,\mathcal{V}_\lambda)$
    	of weight
    	$\lambda
    	\in
    	X^\ast(\widetilde{T}_{\widetilde{P}})$
    	with
    	$\min_i\mathrm{val}_p
    	(\lambda(\epsilon_{H_i}'))
    	\ge C+1$,
    	then
    	\[
    	(e_{\widetilde{P}}
    	f)|_{Sh_1^{n-\mu}}=0.
    	\]
    \end{lemma}
    \begin{proof}    
    	It suffices to show that
    	$\mathbb{T}_{H_i}f$
    	vanishes on
    	$Sh_1^\diamond\backslash Sh_1^\mu$.	
    	We have the following commutative diagram
    	at each point $x$ which is not in the
    	$\mu$-ordinary locus:
    	\[
    	\begin{tikzcd}
    	e^\ast H^1_\mathrm{dR}(\mathcal{A}_x)
    	\arrow[r,"\pi^\ast"]
    	\arrow[d,"\mathbb{D}"]
    	&
    	e^\ast H^1_\mathrm{dR}(\mathcal{A}_x/(H_i)_x)
    	\arrow[d,"\mathbb{D}'"]
    	\\
    	\Lambda_{\mathbb{W}}
    	\arrow[r,"M_i"]
    	&
    	\Lambda_{\mathbb{W}}
    	\end{tikzcd}
    	\]
    	For $F$ an element in
    	$H^0(Sh_{H_i},\mathcal{V}_\lambda)$,
    	we have the transformation rule:
    	for each $g\in G(\mathbb{W}[\frac{1}{p}])$
    	\[
    	(\pi^\ast F)(\mathbb{D}\circ g)
    	=
    	F(\mathbb{D}'\circ M_i^{-1}\circ g).
    	\]
    	By
    	the Iwasawa decomposition of
    	$G$,
    	we can write
    	$M_i^{-1}\circ g
    	=kan$
    	with $k\in G(\mathbb{W})$,
    	$a\in \widetilde{T}_{\widetilde{P}}(\mathbb{W}[\frac{1}{p}])$
    	and
    	$n$ a unipotent element.
    	Moreover,
    	we have
    	$\lambda(a)=\lambda(\epsilon_{H_i}')$.
        By assumption
    	$\mathrm{val}_p(\lambda(\epsilon_{H_i}'))
    	\ge
    	m_{\epsilon_{H_i}'}+1$,
    	thus
    	$\pi^\ast F$ is divisible by $p$,
    	which gives
    	the first part.
    	The second part follows from the first part.
    \end{proof}

    It follows from the definition that
    $\mathbb{T}_{H_i}$,
    when restricted to the $\mu$-ordinary locus,
    coincides with the Hecke operator
    $\mathbb{T}_{\epsilon_{H_i}}$
    as in
    Definition
    \ref{Definition of Hecke operator at p}.
    Since
    $Sh^\diamond$
    is dense in
    $Sh$ and
    $Sh^\Sigma$,
    we can 
    extend $\mathbb{T}_{H_i}$
    to act on
    $H^0(Sh,\mathcal{V}_\lambda(-C^\Sigma))
    =H^0(Sh^\Sigma,\mathcal{V}_\lambda(-C^\Sigma))$.
    Moreover,
    by
    Proposition
    \ref{surjectivity of mod p, without mu-ord},
    we know that for each
    $\lambda\in
    X^\ast_\mathrm{dm}(\widetilde{T}_{\widetilde{P}})$,
    there is a positive integer
    $N(\lambda)$
    such that for any
    $t\ge N(\lambda)$,
    the reduction map
    from
    $H^0(Sh_\infty^\Sigma,
    \mathcal{V}_{\lambda+tN_G\lambda_G}(-C^\Sigma))$
    to
    $H^0(Sh_1^\Sigma,
    \mathcal{V}_{\lambda+tN_G\lambda_G}(-C^\Sigma))$
    is surjective.
    Thus for such
    $\lambda$ and $t$,
    the action of $\mathbb{T}_{H_i}$
    descends to
    $H^0(Sh_1^\Sigma,
    \mathcal{V}_{\lambda+tN_G\lambda_G}(-C^\Sigma))$.

    By the relation
    $\lambda_i+\lambda_{r+1-i}=1$,
    it is easy to see that
    $\epsilon_{H_i}'\ne1$.
    We put
    \[
    n_G:=
    \min_i
    \mathrm{val}_p
    (\lambda_G(\epsilon_{H_i}'))>0,
    \]
    We deduce immediately the following
    corollary
    (\textit{cf.}
    \cite[Proposition A.6]{Pilloni2012}):
    \begin{corollary}
    	For
    	a dominant weight
    	$\lambda
    	\in
    	X^\ast_\mathrm{dm}(\widetilde{T}_{\widetilde{P}})$
    	and
    	$t\in\mathbb{N}$
    	with
    	$t\ge\max(N(\lambda),(C+1)/n_G)$,
    	the multiplication by the Hasse invariant
    	$H
    	\colon
    	H^0(Sh_1^\Sigma,\mathcal{V}_{\lambda+t\lambda_G}(-C^\Sigma))
        \rightarrow
        H^0(Sh_1^\Sigma,\mathcal{V}_{\lambda+(t+N_G)\lambda_G}(-C^\Sigma))$    	
    	induces an
    	isomorphism
    	on the subspaces of $\mu$-ordinary
    	cuspidal modular forms
    	\[
    	H
    	\colon
    	e_{\widetilde{P}}
    	H^0(Sh_1^\Sigma,\mathcal{V}_{\lambda+t\lambda_G}(-C^\Sigma))
    	\rightarrow
    	e_{\widetilde{P}}
    	H^0(Sh_1^\Sigma,\mathcal{V}_{\lambda+(t+N_G)\lambda_G}(-C^\Sigma)).
    	\]
    \end{corollary}
    \begin{proof}
    	By
    	Lemma
    	\ref{Hasse invariant commutes with Hecke operators},
    	we know that
    	this multiplication map commutes with
    	the Hecke operators
    	$\mathbb{T}_\epsilon$
    	for
    	$\epsilon\in\widetilde{T}_{\widetilde{P}}^+(E_\mathfrak{p})$.
    	By the preceding lemma,
    	we know that
    	any
    	$f\in
    	e_{\widetilde{P}}
    	H^0(Sh_1^\Sigma,\mathcal{V}_{\lambda+(t+N_G\lambda_G)}(-C^\Sigma))$
    	vanishes on
    	the non-$\mu$-ordinary locus
    	and thus we can divide
    	$f$ by $H$
    	to get a modular form of weight
    	$\lambda+t\lambda_G$.
    \end{proof}

    Recall the construction of the Hasse invariant
    from §\ref{Hasse invariant}.
    It is defined by the
    $N_G$-th power of the pull-back of
    the Hodge line bundle
    $\omega$ from
    $Sh(V)$
    to
    $Sh$.
    Then one has
    (\textit{cf.}
    \cite[Corollaire A.3]{Pilloni2012})
    \begin{proposition}\label{Finite dimension for large weight}
    	The space
    	$e_{\widetilde{P}}
    	H^0(Sh^{\Sigma}_1,
    	\mathcal{V}_\lambda^\Sigma(-C^\Sigma))$
    	is of finite dimension
    	over
    	$\overline{\mathbb{F}}_p$.
    	Moreover,
    	for $t\gg0$ and any
    	$\lambda\in X^\ast(\widetilde{T}_{\widetilde{P}})$,
    	we have the following
    	identities
    	$(r=1,2,\cdots,\infty)$:
    	\[
    	e_{\widetilde{P}}
    	H^0
    	\bigg(
    	Sh_r^{\Sigma,\mu},
    	\mathcal{V}_{\lambda+tN_G\lambda_G}
    	(-C^\Sigma)
    	\bigg)
    	=
    	e_{\widetilde{P}}
    	H^0
    	\bigg(
    	Sh_r^{\Sigma},
    	\mathcal{V}_{\lambda+tN_G\lambda_G}
    	(-C^\Sigma)
    	\bigg).
    	\]
    \end{proposition}
    \begin{proof}
    	We have,
    	by the properties of Hasse invariant
    	$H$,
    	\[
    	H^0
    	\left(
    	Sh_1^{\Sigma,\mu},
    	\mathcal{V}_\lambda(-C^\Sigma)
    	\right)
    	=
    	\bigcup_{t\in\mathbb{N}}
    	H^{-t}
    	H^0
    	\left(
    	Sh_1^{\Sigma},
    	\mathcal{V}_{\lambda+tN_G\lambda_G}(-C^\Sigma)
    	\right).
    	\]
    	Note that
    	the multiplication by
    	$H$
    	embeds
    	$H^0(Sh_1^{\Sigma,\mu},
    	\mathcal{V}_\lambda)$
    	into
    	$H^0
    	(Sh_1^{\Sigma,\mu},
    	\mathcal{V}_{\lambda+N_G\lambda_G})$,
    	thus for the first part in the proposition,
    	it suffices to
    	treat
    	$\lambda':=
    	\lambda+t_0N_G\lambda_G$
    	in place of
    	$\lambda$
    	with some
    	$t_0\gg0$.
    	Therefore we have
    	\[
    	e_{\widetilde{P}}
    	H^0
    	\left(
    	Sh_1^{\Sigma,\mu},\mathcal{V}_{\lambda'}(-C_\Sigma)
    	\right)
    	=
    	\bigcup_{t\in\mathbb{N}}
    	e_{\widetilde{P}}
    	H^0
    	\left(
    	Sh_1^{\Sigma},
    	\mathcal{V}_{\lambda'+tN_G\lambda_G}(-C_\Sigma)
    	\right).
    	\]
    	However,
    	the preceding corollary shows that
    	the RHS is equal to
    	$e_{\widetilde{P}}
    	H^0(Sh_1^\Sigma,
    	\mathcal{V}_{\lambda'}(-C^\Sigma))$,
    	thus we get,
    	for
    	$t\gg0$,
    	\[
    	e_{\widetilde{P}}
    	H^0(Sh_1^{\Sigma,\mu},
    	\mathcal{V}_{\lambda+tN_G\lambda_G}(-C^\Sigma))
    	=
    	e_{\widetilde{P}}
    	H^0(Sh_1^\Sigma,
    	\mathcal{V}_{\lambda+tN_G\lambda}(-C^\Sigma)).
    	\]
    	The case for
    	$r=\infty$
    	follows from
    	Proposition
    	\ref{surjectivity of mod p, without mu-ord}
    	with
    	$t\gg0$.
    	
    \end{proof}

	\section{Hida theory for Shimura varieties of Hodge type}
	\label{Hida theory for Hodge Shimura}
	\subsection{Control theorems}
	We retain the notations as in the preceding sections
	(\textit{cf.}§\ref{Notations on the Shimura datum}):
	$(G,X)$ is a mixed Shimura datum
	of Hodge type
	(with embedding
	$(G,X)
	\hookrightarrow
	(\mathrm{GSp}(V,\psi),S^\pm)$),
	$\widetilde{P}$ a parabolic subgroup of
	$G$
	(the restriction from a parabolic subgroup
	$\widetilde{P}_V$ of
	$\mathrm{GSp}(V,\psi)$
	to $G$),
	$\widetilde{T}_{\widetilde{P}}=\widetilde{P}/\widetilde{P}^\mathrm{der}
	=
	\widetilde{P}_L/\widetilde{P}_L^\mathrm{der}$
	where
	$\widetilde{P}_L=\widetilde{P}\cap L$.
	We need some more notations.
	Write
	($?=\emptyset,\mathrm{cusp}$)
	\begin{equation}\label{p-adic modular forms, big space}
	\mathcal{M}(\widetilde{P})_?
	:=
	\lim\limits_{\overrightarrow{m}}
	e_{\widetilde{P}}
	\mathbb{V}_{?,m}^{\widetilde{P}_L^\mathrm{der}},
	\quad
	\mathbb{M}(\widetilde{P})_?
	:=
	\mathrm{Hom}_{\mathcal{O}_\mathfrak{p}}
	(\mathcal{M}(\widetilde{P})_?,
	E_\mathfrak{p}/\mathcal{O}_\mathfrak{p})
	\end{equation}
	for the colimit
	and
	its
	Pontryagin dual.
	$\mathbb{M}(\widetilde{P})$
	is the space 
	of $\mu$-ordinary $p$-adic modular forms on
	$Sh^\Sigma$.
	We then put
	\[
	\widetilde{T}_{\widetilde{P}}^1
	:=
	\mathrm{Ker}
	(\widetilde{T}_{\widetilde{P}}(\mathcal{O}_\mathfrak{p})
	\rightarrow
	\widetilde{T}_{\widetilde{P}}(\mathcal{O}/\mathfrak{p})),
	\quad
	\mathcal{W}_{\widetilde{P}}
	:=
	\mathcal{O}_\mathfrak{p}
	[[\widetilde{T}_{\widetilde{P}}(\mathcal{O}_\mathfrak{p})]],
	\quad
	\mathcal{W}_{\widetilde{P}}^1
	:=
	\mathcal{O}_\mathfrak{p}
	[[\widetilde{T}_{\widetilde{P}}^1]].
	\]
	$\mathcal{W}_{\widetilde{P}}^1$
	is the weight space.
	The decomposition
	$\widetilde{T}_{\widetilde{P}}(\mathcal{O}_\mathfrak{p})
	=
	\widetilde{T}_{\widetilde{P}}(\mathcal{O}/\mathfrak{p})
	\times\widetilde{T}_{\widetilde{P}}^1$
	induced from the Teichm\"{u}ller lifting
	shows that we can view
	the Iwasawa weight algebra
	$\mathcal{W}_{\widetilde{P}}^1$
	as a subalgebra of
	$\mathcal{W}_{\widetilde{P}}$,
	and thus the
	$\mathcal{W}_{\widetilde{P}}$-module structures on
	$e_{\widetilde{P}}
	\mathbb{V}_{?,m}^{\widetilde{P}_L^\mathrm{der}}$,
	$\mathcal{M}(\widetilde{P})_?$
	and
	$\mathbb{M}(\widetilde{P})_?$
	($?=\emptyset,\mathrm{cusp}$)
	give rise to
	$\mathcal{W}_{\widetilde{P}}^1$-module
	structures on these spaces.
	Here we gather some of the results that we have
	proved or the corollaries of these results.
	
	\begin{theorem}
		\begin{enumerate}
			\item 
			For any character
			$\lambda\in
			X^{\ast}(\widetilde{T}_{\widetilde{P}})$,
			the $\mathcal{O}_\mathfrak{p}$-module
			$e_{\widetilde{P}}
			\mathbb{V}_{\mathrm{cusp},\infty}^{\widetilde{P}_L^\mathrm{der}}
			[\lambda^{-1}]$
			is free of finite rank
			(denoted by $\mathrm{rk}(\lambda)$).
			Moreover
			the rank
			$\mathrm{rk}(\lambda)$
			depends only on the image of
			$\lambda$
			by the natural projection
			map
			$X^{\ast}(\widetilde{T}_{\widetilde{P}})
			\rightarrow
			X^{\ast}(\widetilde{T}_{\widetilde{P}})
			/\mathbb{Z}N_G\lambda_G$;

			\item 
			For any
			$\lambda
			\in
			X_\mathrm{dm}^\ast(\widetilde{T}_{\widetilde{P}})$,
			we have an isomorphism
			\[
			e_{\widetilde{P}}
			H^0(\mathrm{Ig}_{\infty,1}^\mu/\widetilde{P}_L,
			\widetilde{\mathcal{R}}_\infty[\lambda^{-1}])
			\simeq
			e_{\widetilde{P}}
			\mathbb{V}_\infty^{\widetilde{P}_L^\mathrm{der}}[\lambda^{-1}].
			\]
			If moreover
			$\lambda
			\in
			X_\mathrm{dd}^\ast(\widetilde{T}_{\widetilde{P}})$,
			we can descent the isomorphism to
			$Sh^{\Sigma,\mu}_\infty$:
			\[
			e_{\widetilde{P}}
			H^0(Sh^{\Sigma,\mu}_\infty,
			\widetilde{\mathcal{R}}_\infty
			[\lambda^{-1}])
			\simeq
			e_{\widetilde{P}}
			\mathbb{V}_\infty^{\widetilde{P}_L^\mathrm{der}}[\lambda^{-1}].
			\]

			\item 
			For any
			$\lambda
			\in
			X^\ast(\widetilde{T}_{\widetilde{P}})$
			and
			$t\gg0$
			(depending on $\lambda$),
			we have an isomorphism
			of spaces of cuspidal modular forms:
			\[
			e_{\widetilde{P}}
			H^0(Sh_\infty^\mu,
			\widetilde{\mathcal{R}}_\infty
			[(\lambda+t\lambda_{G})^{-1}])
			\simeq
			e_{\widetilde{P}}
			\mathbb{V}_{\mathrm{cusp},\infty}^{\widetilde{P}_L^\mathrm{der}}
			[(\lambda+t\lambda_{G})^{-1}].
			\]

			\item 
			For any
			$\lambda\in
			X^{\ast}(\widetilde{T}_{\widetilde{P}})$,
			we have the following specialization isomorphism
			\[
			\mathbb{M}(\widetilde{P})_\mathrm{cusp}
			\otimes_{\mathcal{W}_{\widetilde{P}},\lambda}
			\mathcal{O}_\mathfrak{p}
			\simeq
			\mathrm{Hom}_{\mathcal{O}_\mathfrak{p}}
			(e_{\widetilde{P}}
			\mathbb{V}_{\mathrm{cusp},\infty}^{\widetilde{P}_L^\mathrm{der}}
			[\lambda^{-1}],\mathcal{O}_\mathfrak{p});
			\]

			\item 
			The
			$\mathcal{W}_{\widetilde{P}}^1$-module
			$e_{\widetilde{P}}
			\mathbb{V}_{\mathrm{cusp},\infty}^{\widetilde{P}_L^\mathrm{der}}$
			is free of finite rank.		
		\end{enumerate}
	\end{theorem}
    \begin{proof}
    	\begin{enumerate}
    		\item 
    		First note that
    		we have the base change isomorphism by
    		Proposition \ref{surjectivity of mod p, mu-ordinary}:
    		\[
    		e_{\widetilde{P}}
    		\mathbb{V}_{\mathrm{cusp},\infty}^{\widetilde{P}_L^\mathrm{der}}
    		[\lambda^{-1}]
    		\otimes_{\mathcal{O}_\mathfrak{p}}
    		\mathcal{O}/\mathfrak{p}
    		\simeq
    		e_{\widetilde{P}}
    		\mathbb{V}_{\mathrm{cusp},1}^{\widetilde{P}_L^\mathrm{der}}
    		[\lambda^{-1}].
    		\]
    		Now that the dimension of
    		$e_{\widetilde{P}}
    		\mathbb{V}_{\mathrm{cusp},1}^{\widetilde{P}_L^\mathrm{der}}
    		[\lambda^{-1}]$
    		depends on the image of
    		$\lambda$
    		in
    		$X^{\ast}(\widetilde{T}_{\widetilde{P}})
    		/\mathbb{Z}N_G\lambda_G$.
    		Thus we can assume that
    		$\lambda
    		\in
    		X^{\ast}_\mathrm{dm}(\widetilde{T}_{\widetilde{P}})$
    		and now we can apply
    		Proposition \ref{Finite dimension for large weight}
    		to such $\lambda$.

    		\item 
    		The first point is 
    		Proposition
    		\ref{comparison of forms on Igusa and p-adic}.
    		The second point follows from
    		Corollary \ref{descent from Igusa to Sh}.

    		\item
    		This is Proposition
    		\ref{Finite dimension for large weight}.
    		
    		\item 
    		By definition,
    		for any
    		$\lambda
    		\in
    		X^\ast(\widetilde{T}_{\widetilde{P}})$,
    		we have the following natural isomorphism
    		\begin{align*}
    		\mathbb{M}(\widetilde{P})_\mathrm{cusp}
    		\otimes_{\mathcal{W}_{\widetilde{P}},\lambda}
    		\mathcal{O}_\mathfrak{p}
    		&
    		\simeq
    		\mathrm{Hom}_{\mathcal{O}_\mathfrak{p}}
    		(\mathcal{M}(\widetilde{P})_\mathrm{cusp}[\lambda^{-1}],
    		E_\mathfrak{p}/\mathcal{O}_\mathfrak{p})
    		\\
    		&
    		\simeq
    		\mathrm{Hom}_{\mathcal{O}_\mathfrak{p}}
    		(e_{\widetilde{P}}
    		\mathbb{V}_{\mathrm{cusp},\infty}^{\widetilde{P}_L^\mathrm{der}}[\lambda^{-1}]
    		\otimes_{\mathcal{O}_\mathfrak{p}}
    		E_\mathfrak{p}/\mathcal{O}_\mathfrak{p},
    		E_\mathfrak{p}/\mathcal{O}_\mathfrak{p})
    		\\
    		&
    		\simeq
    		\mathrm{Hom}_{\mathcal{O}_\mathfrak{p}}
    		(e_{\widetilde{P}}\mathbb{V}_{\mathrm{cusp},\infty}^{\widetilde{P}_L^\mathrm{der}}
    		[\lambda^{-1}],
    		\mathcal{O}_\mathfrak{p}).
    		\end{align*}
    		
    		\item 
    		We follow
    		\cite[p.36]{Pilloni2012}.
    		For any character
    		$\chi
    		\colon
    		\widetilde{T}_{\widetilde{P}}
    		(\mathcal{O}_\mathfrak{p})
    		\rightarrow
    		\mathcal{O}_\mathfrak{p}^\times$
    		in
    		$X^\ast(\widetilde{T}_{\widetilde{P}})$,
    		we write
    		$r(\chi)$
    		for the
    		$\mathcal{O}_\mathfrak{p}$-rank of
    		$\mathbb{M}(\widetilde{P})_\mathrm{cusp}
    		\otimes_{\mathcal{W}_{\widetilde{P}},\chi}
    		\mathcal{O}_\mathfrak{p}$,
    		which is finite
    		by the points $1$ and $4$.
    		Then we can define a surjective
    		$\mathcal{W}_{\widetilde{P}}^1$-linear map
    		$f_\chi\colon
    		(\mathcal{W}_{\widetilde{P}}^1)^{r(\chi)}
    		\rightarrow
    		e_{\widetilde{P}}
    		\mathbb{V}_{\mathrm{cusp},\infty}^{\widetilde{P}_L^\mathrm{der}}
    		\otimes_{
    			\mathcal{O}_\mathfrak{p}
    			[\widetilde{T}_{\widetilde{P}}(\mathcal{O}/\mathfrak{p})],
    			\chi}
    		\mathcal{O}_\mathfrak{p}$.
    		Now for any other character
    		$\chi'\colon
    		\widetilde{T}_{\widetilde{P}}
    		(\mathcal{O}_\mathfrak{p})
    		\rightarrow
    		\mathcal{O}_\mathfrak{p}^\times$
    		such that
    		$\chi|_{\widetilde{T}_{\widetilde{P}}(\mathcal{O}/\mathfrak{p})}
    		=
    		\chi'|_{\widetilde{T}_{\widetilde{P}}(\mathcal{O}/\mathfrak{p})}
    		$,
    		the induced map
    		\[
    		f_\chi\otimes1
    		\colon
    		(\mathcal{W}_{\widetilde{P}}^1)^{r(\chi)}
    		\otimes_{\mathcal{W}_{\widetilde{P}}^1,\chi'}
    		\mathcal{O}_\mathfrak{p}
    		\rightarrow
    		e_{\widetilde{P}}
    		\mathbb{V}_{\mathrm{cusp},\infty}^{\widetilde{P}_L^\mathrm{der}}
    		\otimes_{\mathcal{W}_{\widetilde{P}},\chi}
    		\mathcal{O}_\mathfrak{p}
    		\]
    		is an isomorphism
    		since both sides are free of rank
    		$r(\chi)$
    		over
    		$\mathcal{O}_\mathfrak{p}$.
    		By the density of characters
    		$\chi'\in X^{\ast}(\widetilde{T}_{\widetilde{P}})$
    		in
    		$\mathrm{Spec}(\mathcal{W}_{\widetilde{P}}^1)$,
    		we see that
    		$f_\chi$
    		is in fact an isomorphism.
    		On the other hand,
    		we have a decomposition
    		\[
    		\mathbb{M}(\widetilde{P})_\mathrm{cusp}
    		=
    		\oplus_{\chi^\circ}
    		\mathbb{M}(\widetilde{P})_\mathrm{cusp}
    		\otimes_{
    			\mathcal{O}_\mathfrak{p}
    			[\widetilde{T}_{\widetilde{P}}(\mathcal{O}/\mathfrak{p})],\chi^\circ}
    		\mathcal{O}_\mathfrak{p},
    		\]
    		where
    		$\chi^\circ$
    		runs through characters
    		$\chi^\circ
    		\colon
    		\widetilde{T}_{\widetilde{P}}
    		(\mathcal{O}/\mathfrak{p})
    		\rightarrow
    		(\mathcal{O}/\mathfrak{p})^\times$.
    		Since there are only finitely many such
    		$\chi^\circ$,
    		we get immediately that
    		$\mathbb{M}(\widetilde{P})_\mathrm{cusp}$
    		is free of finite rank over
    		$\mathcal{W}_{\widetilde{P}}^1$.
    	\end{enumerate}
    \end{proof}

	\subsection{Hida families}	
	Recall that $G$ is a connected reductive group over
	$\mathbb{Q}$
	which embeds in
	$\mathrm{GSp}(V,\psi)$.
	We have also fixed a compact open subgroup
	$K$ of
	$G(\mathbb{A}_\mathrm{f})$.
	Let
	$S'$ be the subset of rational primes $\ell$ of $\mathbb{Q}$
	such that the $\ell$-th component of $K$ is not
	maximal open compact in $G(\mathbb{Q}_\ell)$.
	Write
	$S=S'\cup\{p\}$.
	For each rational prime $\ell$ of
	$\mathbb{Q}$ not in $S$
	such that $G$ is unramified over $\ell$,
	we fix
	one hyperspecial subgroup $\mathcal{K}_\ell$ of
	$G(\mathbb{Q}_\ell)$
	which is the intersection
	$G(\mathbb{Q}_\ell)
	\cap
	\mathcal{K}_{V,\ell}$
	where
	$\mathcal{K}_{V,\ell}=
	\mathrm{GSp}(V,\psi)(\mathbb{Z}_\ell)$
	is a maximal compact open subgroup of
	$\mathrm{GSp}(V,\psi)(\mathbb{Q}_\ell)$.
	Then
	we define the local commutative spherical Hecke algebra
	at $\ell$
	as usual
	in the following way:
	\[
	\mathcal{H}_\ell
	:=
	\mathbb{Z}[G(\mathbb{Q}_\ell)//\mathcal{K}_\ell].
	\]

	For each double coset
	$\mathcal{K}_\ell\beta\mathcal{K}_\ell$
	with $\beta\in G(\mathbb{Q}_\ell)$,
	one can define a spherical Hecke operator
	$\mathbb{T}_\beta$
	on the spaces of $\mu$-ordinary $p$-adic
	automorphic forms
	$\mathbb{M}(\widetilde{P})$
	and
	$\mathbb{M}(\widetilde{P})_\mathrm{cusp}$
	similar to Definition
	\ref{Definition of Hecke operator at p}
	as follows:
	we first define a correspondence
	$X_{V,\beta}$ over
	$Sh(V)$
	as in
	\textit{cf.}\cite[§VII.3]{FaltingsChai1990}.
	Then we write
	$X_\beta$ for the pull-back of
	$X_{V,\beta}$
	along the embedding
	$Sh
	\rightarrow
	Sh(V)$.
	We have two natural projections
	$\mathrm{pr}_1,\mathrm{pr}_2
	\colon
	X_\beta
	\rightarrow
	Sh$,
	a universal isogeny $\pi^{G,\beta}$
	between universal abelian schemes
	$\pi^{G,\beta}
	\colon
	\mathcal{A}
	\rightarrow
	\mathcal{A}'$.
	Then for any sheaf $\mathcal{F}$ over
	$Sh$,
	we put
	\[
	\mathbb{T}_\beta
	\colon
	H^0(Sh,\mathcal{F})
	\rightarrow
	H^0(X_\beta,\mathrm{pr}_2^\ast\mathcal{F})
	\xrightarrow{(\pi^{G,\beta})^\ast}
	H^0(X_\beta,\mathrm{pr}_1^\ast\mathcal{F})
	\xrightarrow{\mathrm{Tr}_{\mathrm{pr}_1}}
	H^0(Sh,\mathcal{F}).
	\]
	Similarly,
	one can define $\mathbb{T}_\beta$
	on the Igusa towers.
	We know moreover that these Hecke operators
	$\mathbb{T}_\beta$
	commutes with each other for different $\ell$ and
	also commutes with
	those operators
	$\mathbb{T}_\epsilon$
	for $\epsilon\in\widetilde{T}_{\widetilde{P}}^+(E_\mathfrak{p})$
	as well as the action of
	$\widetilde{T}_{\widetilde{P}}(\mathcal{O}_\mathfrak{p})$.
	\begin{definition}
		We define the cuspidal
		$\widetilde{P}$-ordinary Hecke algebra
		\[
		e_{\widetilde{P}}
		\mathcal{H}^{\widetilde{P}}_\mathrm{cusp}
		:=
		\mathcal{W}_{\widetilde{P}}
		\langle
		\mathcal{H}_\ell|\ell\notin S
		\rangle
		\subset
		\mathrm{End}_{\mathcal{W}_{\widetilde{P}}}
		(\mathbb{M}(\widetilde{P})_\mathrm{cusp}).
		\]
		
		Each irreducible component of
		$\mathrm{Spec}
		(e_{\widetilde{P}}
		\mathcal{H}_\mathrm{cusp}^{\widetilde{P}})$
		is called a \textbf{Hida family}.
	\end{definition}
	
	\begin{remark}
		
		By construction,
		the $\overline{\mathcal{O}_\mathfrak{p}}$-points of
		$\mathrm{Spec}
		(e_{\widetilde{P}}\mathcal{H}_\mathrm{cusp}^{\widetilde{P}})$
		correspond bijectively with the eigen-systems for
		the action of
		$e_{\widetilde{P}}\mathcal{H}_\mathrm{cusp}^{\widetilde{P}}$
		on the space
		$\mathbb{M}(\widetilde{P})_\mathrm{cusp}$.
		By the results in the preceding subsection,
		we know that the Hecke algebra
		$e_{\widetilde{P}}\mathcal{H}^{\widetilde{P}}_\mathrm{cusp}$
		is a faithfully flat
		$\mathcal{W}_{\widetilde{P}}^1$-algebra of finite rank.
		Thus we see that for any
		$\lambda\in
		X^{\ast}(\widetilde{T}_{\widetilde{P}})$,
		any eigenform
		$f\in
		e_{\widetilde{P}}
		\mathbb{V}_{\mathrm{cusp},\infty}^{\widetilde{P}_L^\mathrm{der}}
		[\lambda^{-1}]
		\otimes_{\mathcal{O}_\mathfrak{p}}
		\overline{\mathcal{O}_\mathfrak{p}}$
		for the Hecke algebra
		$e_{\widetilde{P}}\mathcal{H}_\mathrm{cusp}^{\widetilde{P}}$
		and any other weight
		$\lambda'\in
		X^{\ast}(\widetilde{T}_{\widetilde{P}})$
		with the same image as $\lambda$
		in
		$X^{\ast}(\widetilde{T}_{\widetilde{P}})
		/\mathbb{Z}N_G\lambda_G$,
		there is an eigenform
		$f'\in
		e_{\widetilde{P}}
		\mathbb{V}_{\mathrm{cusp},\infty}^{\widetilde{P}_L^\mathrm{der}}
		[(\lambda')^{-1}]
		\otimes_{\mathcal{O}_\mathfrak{p}}
		\overline{\mathcal{O}_\mathfrak{p}}
		$
		for this Hecke algebra
		such that the eigen-systems
		corresponding to $f$ and $f'$
		have the same image in
		$\mathrm{Spec}(e_{\widetilde{P}}
		\mathcal{H}_\mathrm{cusp}^{\widetilde{P}}
		\otimes_{\mathcal{O}_\mathfrak{p}}
		\overline{\mathbb{F}}_p)$.
		Therefore,
		one can choose a sequence
		of characters
		$\lambda_n
		\in
		X^{\ast}_\mathrm{dd}(\widetilde{T}_{\widetilde{P}})$
		which converge $p$-adically to $\lambda$
		and a sequence of classical
		modular
		forms
		$f_n\in
		e_{\widetilde{P}}
		\mathbb{V}_{\mathrm{cusp},\infty}^{\widetilde{P}_L^\mathrm{der}}
		[\lambda_n^{-1}]
		\otimes_{\mathcal{O}_\mathfrak{p}}
		\overline{\mathcal{O}_\mathfrak{p}}
		=
		e_{\widetilde{P}}
		H^0(Sh_\infty^\mu,
		\widetilde{\mathcal{R}}_\infty
		[\lambda_n^{-1}])
		\otimes_{\mathcal{O}_\mathfrak{p}}
		\overline{\mathcal{O}_\mathfrak{p}}$
		such that the corresponding eigen-systems of these $f_n$
		converge $p$-adically to that of $f$.
	\end{remark}


\begin{thebibliography}{widestlabel}
		
		
		
		
		
		
	    \bibitem[Bij16]{Bijakowski2016}
	    S.Bijakowski,
	    Analytic continuation on Shimura varieties wit
	    $\mu$-ordinary locus,
	    Algebra and Number Theory
	    10-4(2016),
	    pp.843-885.
	    

	    
	    \bibitem[BR17]{BrascaRosso2017}
	    R.Brasca and G.Rosso,
	    Hida theory over some unitary Shimura varieties without
	    ordinary locus,
	    \url{https://arxiv.org/abs/1711.05546v2}.
	    
	    
		\bibitem[CFS]{ChenFarguesShen2017}
		M.Chen,L.Fargues and X.Shen,
		On the structure of some $p$-adic period domains,
		preprint,
		\url{http://arxiv.org/abs/1710.06935}.
		
		
		
		
		
		
		\bibitem[Con11]{Conrad2011}
		B.Conrad,
		Reductive group schemes,
		Notes for the SGA3 Summer School,
		Luminy 2011.
		
		
		
		
		\bibitem[Del77]{Deligne1977}
		P.Deligne,
		Vari\'{e}t\'{e} de Shimura: Interpr\'{e}tation modulaire
		et techniques de construction de mod\`{e}les
		canoniques,
		in A.Borel, W.Casselman (ed.):
		\textit{Automorphic forms, representations and
		$L$-functions 2},
	    Proc.Symp.Pure.Math.
	    \textbf{33}(1977),
	    pp.247-289.
		
		
		\bibitem[EGA]{EGA1961}
		A.Grothendieck and J.Dieudonn\'{e},
		\textit{\'{E}l\'{e}ments de g\'{e}m\'{e}trie alg\'{e}brique},
		I-IV,
		Publ.Math.I.H.E.S.,
		\textbf{4},
		\textbf{8},
		\textbf{11},
		\textbf{17},
		\textbf{20},
		\textbf{24},
		\textbf{28},
		\textbf{32},
		1961-1967.
		
		
		
		\bibitem[EHLS16]{EischenHarrisLiSkinner2016}
		E.Eischen,M.Harris,J.-S.Li and C.Skinner,
		$p$-adic $L$-functions for unitary groups,
		available at
		\url{https://arxiv.org/abs/1602.01776}.
		
		
		
		\bibitem[EM17]{EischenMantovan2017}
		E.Eischen and E.Mantovan,
		$p$-adic families of automorphic forms in the
		$\mu$-ordinary setting,
		available at
		\url{http://arxiv.org/1710.01864}.
		
		
		
		
		
		\bibitem[FC90]{FaltingsChai1990}
		G.Faltings and C.-L.Chai,
		\textit{Degeneration of abelian varieties},
		Ergebniss der Mathematik und ihrer
		Grenzbgebiete(3),
		vol.22,
		Springer-Verlag,
		Berlin,1990.
		
		
		\bibitem[Far10]{Fargues2010}
		L.Fargues,
		La filtration de Harder-Narasimhan des
		sch\'{e}mas en groupes finis et plats,
		J.reine angew.Math.
		645(2010).
		
		
		
				
		
		
		\bibitem[GK16]{GoldringKoskivirta2016}
		W.Goldring and J-S.Koskivirta,
		Strata Hasse invariants, Hecke algebras and
		Galois representations,
		preprint,
		\url{https://arxiv.org/abs/1507.05032}.
		
		
		
		\bibitem[GS93]{GreenbergStevens1993}
		R.Greenberg and G.Stevens,
		$p$-adic $L$-functions and $p$-adic periods of
		modular forms,
		Invent.Math.\textbf{111}
		(1993),
		no.2,
		pp.407-447.
		
		
		\bibitem[Her18]{Hernandez2018}
		V.Hernandez,
		Invariants de Hasse $\mu$-ordinaires,
		Annales de l'Institut Fourier,
		Tome 68(2018)
		no.4,
		pp.1519-1607.
		
		\bibitem[Hi86a]{HidaIwasawaModules}
		H.Hida,
		Iwasawa modules attached to congruences of cusp forms, Ann. Sci. \'{E}ocle Norm. Sup. (4) \textbf{19} (1986), no.2, pp.231-273.
		
		
		\bibitem[Hid86b]{Hida1986GaloisRepresentations}
		H.Hida,
		Galois representations into
		$\mathrm{GL}_2(\mathbb{Z}[[X]])$
		attached to ordinary cusp forms,
		Invent.Math.\textit{85} (1985),
		pp.545-613.
		
		\bibitem[Hid95]{Hida1995}
		H.Hida,
		Control theorems of $p$-ordinary
		cohomology groups for
		$\mathrm{SL}(n)$,
		Bull.SMF \textbf{123}(1995),
		pp.425-475.
		
		\bibitem[Hid02]{Hida2002}
		H.Hida,
		Control theorem for coherent sheaves on
		Shimura varieties of
		PEL type,
		Jour.Inst.Math.Jussieu,\textbf{1},2002.
		
		
		\bibitem[Jan03]{Jantzen2003}
		J.C.Jantzen,
		\textit{Representations of algebraic groups},
		2nd edition,
		Mathematical Surveys and Monographs,
		vol. 107,
		American Mathematical Society,
		Providence, RI, 2003.
		
		
		
		\bibitem[Kis10]{Kisin2010}
		M.Kisin,
		Integral models for Shimura varieties of abelian
		type,
		J.Amer.Math.Soc.
		\textbf{23} (2010),
		no.4,
		pp.967-1012.
		
		
		
		\bibitem[Kis16]{Kisin2016}
		M.Kisin,
		Mod $p$ points on Shimura varieties of abelian type,
		J.A.M.S.
		\textbf{30}(3),
		2016,
		no.1.
		
		
		
		\bibitem[KW15]{KoskivirtaWedhorn2015}
		J-S.Koskivirta and T.Wedhorn,
		Generalized Hasse invariants for
		Shimura varieties of Hodge type
		(2015).
		
		
		
		
		\bibitem[Kot85]{Kottwitz1985}
		R.Kottwitz,
		Isocrystals with additional structure,
		Compositio Math.
		\textbf{56} (1985),
		pp.201-220.
		
		
		\bibitem[Kot97]{Kottwitz1997}
		R.Kottwitz,
		Isocrystals with additional structure II,
		Compositio Math.
		\textbf{109}(1997),
		no.3,
		pp.255-339.
		
		
		\bibitem[Lan16]{Lan2016}
		K-W.Lan,
		Higher Koecher's principle,
		Mathematical Research Letters,
		23(1)(2016),
		pp.163-199.
		
		
		\bibitem[Liu15b]{Liu2015}
		Z.Liu, $p$-adic $L$-functions for ordinary families on symplectic groups,
		J.Inst.Math.Jussieu,
		(2019).
		
		
		\bibitem[Mad12]{Madapusi2012}
		K.Madapusi,
		Toroidal compactifications of integral models
		of Shimura varieties of Hodge type,
		preprint,
		available at
		\url{https://arxiv.org/abs/1211.1731}.
		
		
		
	    \bibitem[Maz89]{Mazur1989}
	    B.Mazur,
	    Deforming Galois representations,
	    \textit{Galois groups over $\mathbb{Q}$}
	    (Y.Ihara, K.Ribet, J.-P.Serre eds.),
	    Springer Verlag (1989),
	    pp.385-435.
		
		
		\bibitem[Mil88]{Milne1988}
		J.Milne,
		Canonical models of mixed Shimura varieties
		and automorphic vector bundles,
		in
		\textit{Automorphic forms, Shimura varieties and
			$L$-functions},
		Vol.I,
		Proceedings of the conference held at 
		the University of Michigan,
		Ann Arbor, Michigan,
		July 6-16, 1988.
		
		
		
		
		\bibitem[Mil94]{Milne1994}
		J.Milne,
		Shimura varieties and motives,
		in
		\textit{Motives},
		Proc. of Symp. in Pure Math.
		\textbf{55},
		Part 2,
		pp.447-523.
		
		
		\bibitem[Mum70]{Mumford1970}
		D.Mumford,
		\textit{Abelian Varieties},
		Tata Institute of Fundamental Research Studies in
		Mathematics,
		vol.5, Oxford University Press, 1970.
		
		
		\bibitem[OZ02]{OortZink2002}
		F.Oort and T.Zink,
		Families of $p$-divisible groups with
		constant Newton polygon,
		Doc.Math.
		\textbf{7} (2002),
		pp.183-201.
		
		
		
		\bibitem[Pil12]{Pilloni2012}
		V.Pilloni,
		Sur la th\'{e}orie de Hida pour le groupe
		$\mathrm{GSp}_{2g}$,
		Bull.Soc.math.France,
		\textbf{140}(3), 2012, pp.335-400.
		
		
		
		\bibitem[Pil12b]{Pilloni2012b}
		V.Pilloni,
		Modlarit\'{e}, formes de Siegel et surfaces ab\'{e}liennes,
		J.reine angew Math.
		666(2012),
		pp.35-82.
		
		
		
		\bibitem[RR96]{RapoportRichartz1996}
		M.Rapoport and M.Richartz,
		On the classification and specialization of
		$F$-isocrystals with additional structure,
		Compositio Math.
		\textbf{103} (1996),
		pp.153-181.
		
		
		
		
		\bibitem[Ser72]{Serre1972}
		J.-P.Serre,
		Formes modulaires et fonctions z\^{e}ta $p$-adiques,
		\textit{Modular functions of one variable},
		III
		(Proc.Internat.Summer School, Univ.Antwerp, 1972),
		Springer Lecture Notes \textbf{350},
		pp.191-268.
		
		
		
		
		\bibitem[SU14]{SkinnerUrban}
		C.Skinner and E.Urban,
		The Iwasawa main conjectures for $\mathrm{GL}_2$,
		Invent. Math. \textbf{195(1)} (2014), pp.1-277.
		
		
		\bibitem[SZ16]{ShankarZhou2016}
		A.Shankar and R.Zhou,
		Serre-Tate theory for Shimura varieties of Hodge type,
		\url{https://arxiv.org/abs/1612.06456v1}.
		
		
		
		
		
		
		\bibitem[TW95]{TaylorWiles1995}
		R.Taylor and A.Wiles,
		Ring-theoretic properties of certain Hecke algebras,
		Ann.Math.\textbf{141} (1995),
		pp.553-572.
		
		
		
		\bibitem[Vas99]{Vasiu1999}
		A.Vasiu,
		Integral canonical models of Shimura
		varieties of preabelian type,
		Asian.J.Math.
		\textbf{3} (1999),
		pp.401-518.
		
		
		
		
		
		\bibitem[Wed99]{Wedhorn1999}
		T.Wedhorn,
		Ordinariness in good reductions of Shimura
		varieties of PEL-type,
		Ann.Sci.de l'\'{E}NS,
		\textbf{32} (1999),
		pp.575-618.
		
		
		
		
		
		
		\bibitem[Wil95]{Wiles1995}
		A.Wiles,
		Modular elliptic curves and Fermat's last ttheorem,
		Ann.Math.\textbf{141} (1995),
		pp.443-551.
		
		
		\bibitem[Wor13]{Wortmann2013}
		D.Wortmann,
		The $\mu$-ordinary locus of
		Shimura varieties of Hodge type,
		preprint,
		available at
		\url{https://arxiv.org/abs/1310.6444}
		
		
		
		
		
		\bibitem[Zha18]{Zhang2018}
		C.Zhang,
		Stratifications in good reductions of
		Shimura varieties of abelian type,
		preprint,
		available at
		\url{https://arxiv.org/abs/1707.00439v2}
		
		
		
		
%		\bibitem[Zha15]{ZhangChao2015}
%		C.Zhang,
%		Stratifications and foliations for good
%		reductions of Shimura varieties of
%		Hodge type,
%		preprint,
%		\url{https://arxiv.org/abs/1512.08102}
	\end{thebibliography}
\end{document}